\documentclass[11pt]{amsart}

\usepackage{
    amsmath,
    amsfonts,
    amssymb,
    amsthm,
    amscd,
    booktabs,
    comment,
    enumitem,
    etoolbox,
    gensymb,    
    mathtools,
    mathdots
}
\usepackage[usenames,dvipsnames]{xcolor}
\usepackage[all]{xy}

\makeatletter
\@namedef{subjclassname@2020}{%
  \textup{2020} Mathematics Subject Classification}
\makeatother

\newtoggle{comments}
\newtoggle{details}
\newtoggle{detailsnote}

\iftoggle{comments}{%
  \newcommand{\acomments}[1]{
    \ \\
    {\color{red}
      \textbf{AS:} #1
    }
    \ \\
    }
  \newcommand{\question}[1]{
    \ \\
    {\color{blue}
      \textbf{Question:} #1
    }
    \ \\
    }
}{%
  \newcommand{\acomments}[1]{}
  \newcommand{\question}[1]{}
}

\iftoggle{details}{%
  \newcommand{\details}[1]{
      \ \\
      {\color{OliveGreen}
        \textbf{Details:} #1
      }
      \\
  }
}{%
  \newcommand{\details}[1]{}
}

\usepackage[T1]{fontenc}
\usepackage{bbm}                     
\usepackage[colorlinks=true, linkcolor=blue, citecolor=blue, urlcolor=blue, breaklinks=true]{hyperref}
\usepackage{dsfont}

\DeclareFontFamily{OT1}{pzc}{}
\DeclareFontShape{OT1}{pzc}{m}{it}{<-> s * [1.10] pzcmi7t}{}
\DeclareMathAlphabet{\mathpzc}{OT1}{pzc}{m}{it}

\usepackage[capitalize]{cleveref}   

\crefname{defin}{Definition}{Definitions}
\crefname{eg}{Example}{Examples}
\crefname{lem}{Lemma}{Lemmas}
\crefname{theo}{Theorem}{Theorems}
\crefname{equation}{}{}
\crefname{enumi}{}{}

\newtheorem{innercustomthm}{{\bf Theorem}}
\newenvironment{customthm}[1]
  {\renewcommand\theinnercustomthm{#1}\innercustomthm}
  {\endinnercustomthm}

\newcommand\go{\mathsf{I}}              

\usepackage{tikz}
\usetikzlibrary{arrows,patterns}
\usepackage{tikz-cd}

\tikzset{pinhead/.style={gray,fill=yellow!40!white}}
\newcommand{\pin}[3]{
    \path (#1) node[inner sep=1.6pt] (x) {} to (#2) node[rectangle,rounded corners,draw,pinhead,inner sep=2.5pt](y) {$\color{black}\scriptstyle#3$};
    \draw[Triangle Cap-,thick,gray] (x)--(y);
    \fill[black] (#1) circle (2pt);
}

\newcommand{\pinpin}[4]{
\path (#1) node[inner sep=1.6pt] (x) {} to (#3) node[rectangle,rounded corners,draw,pinhead,inner sep=2.5pt](y) {$\color{black}\scriptstyle#4$};
    \draw[Triangle Cap-,thick,gray] (x)--(#2)--(y);
    \fill[black] (#1) circle (2pt);
    \fill[black] (#2) circle (2pt);
}

\newcommand\singdot[2][red]{
    \filldraw[fill=white, draw=#1] (#2) circle (1.8pt)
}
\newcommand\multdot[4][red]{
    \filldraw[fill=white, draw=#1] (#2) circle (1.8pt) node[anchor=#3] {{\color{#1} \dotlabel{#4}}}
}

\newcommand{\dotlabel}[1]{$\scriptstyle{#1}$}

\newcommand{\xtoken}[1]{
    \filldraw[green] (#1) circle (1.5pt)
}

\newcommand{\bluetoken}[1]{
    \filldraw[blue] (#1) circle (1.5pt)
}

\newcommand\teleport[2]{
    \draw[blue] (#1) -- (#2);
    \filldraw[blue] (#1) circle (1.5pt);
    \filldraw[blue] (#2) circle (1.5pt);
}
\newcommand\xteleport[2]{
    \draw[green] (#1) -- (#2);
    \filldraw[green] (#1) circle (1.5pt);
    \filldraw[green] (#2) circle (1.5pt);
}
\tikzset{wei/.style={draw=red,double=red!40!white,double distance=1.5pt,thin}}
\tikzset{anchorbase/.style={>=To,baseline={([yshift=-0.5ex]current bounding box.center)}}}

\newcommand{\xtokstrand}[1][a]{
    \begin{tikzpicture}[centerzero]
        \draw (0,-0.2) -- (0,0.2);
        \xtoken{0,0};
    \end{tikzpicture}
}

\tikzset{anchorbase/.style={>=To,baseline={([yshift=-0.5ex]current bounding box.center)}}}
\tikzset{ 
    centerzero/.style={>=To,baseline={([yshift=-0.5ex](#1))}},
    centerzero/.default={0,0}
}
\tikzset{wipe/.style={white,line width=4pt}}

\newtheorem{theo}{Theorem}[section]

\newtheorem{prop}[theo]{Proposition}
\newtheorem{lem}[theo]{Lemma}
\newtheorem{cor}[theo]{Corollary}

\theoremstyle{definition}
\newtheorem{defn}[theo]{Definition}
\newtheorem{rmk}[theo]{Remark}
\newtheorem{egg}[theo]{Example}

\numberwithin{equation}{section}
\allowdisplaybreaks

\setenumerate[1]{label=(\alph*)}          

\begin{document}

\title{Tensor product degenerate spin affine Hecke superalgebras}

\author{Thomas Moran}
\address[T.M.]{
  Department of Mathematics and Statistics \\
  University of Ottawa \\
  Ottawa, ON K1N 6N5, Canada
}
\urladdr{\href{https://sites.google.com/view/tmoran/}{https://sites.google.com/view/tmoran/}}
\email{tmora083@uottawa.ca}

\begin{abstract}
We define a new class of superalgebras, called \emph{tensor product degenerate spin affine Hecke superalgebras}, and study their structure theory. We establish an isomorphism that relates the tensor product degenerate spin affine Hecke superalgebras to the tensor product degenerate affine Hecke--Clifford superalgebras. The relationship between these superalgebras is a tensor product analogue of the relationship between the degenerate spin affine Hecke superalgebras and degenerate affine Hecke--Clifford superalgebras.
\end{abstract}



\ifboolexpr{togl{comments} or togl{details}}{%
  {\color{magenta}DETAILS OR COMMENTS ON}
}{%
}

\maketitle
\thispagestyle{empty}

\tableofcontents

\section{Introduction}

\subsection{Degenerate spin affine Hecke superalgebras}
The spin (or projective) representations of the symmetric group $ S_{n} $ were first developed by Schur in \cite{Schur}. There, Schur introduced a double cover $ \tilde{S}_{n} $ of $ S_{n} $:
\begin{equation*}
1 \rightarrow \mathbb{Z}_{2} \rightarrow \tilde{S}_{n} \rightarrow S_{n} \rightarrow 1.
\end{equation*}
Write $ \mathbb{Z}_{2} = \{1,z\} $. Then a spin representation of $ S_{n} $ is equivalent to a representation of the algebra $ \mathbb{C}S_{n}^{-} := \mathbb{C}\tilde{S}_{n} / \langle z+1 \rangle $. The algebra $ \mathbb{C}S_{n}^{-} $ has generators $ T_{i} $, for $ 1 \leq i \leq n-1 $, subject to the relations
\begin{align*}
T_{i}^{2} &= 1, & 1 \leq i \leq n-1, 
\\ T_{i}T_{i+1}T_{i} &= T_{i+1}T_{i}T_{i+1}, & 1 \leq i \leq n-2,
\\ T_{i}T_{j} &= -T_{j}T_{i}, & 1 \leq i,j \leq n-1, \ |i-j| > 1.
\end{align*}
The algebra $ \mathbb{C}S_{n}^{-} $ is naturally a superalgebra by requiring each $ T_{i} $ to be odd. This superalgebra is related to the Hecke--Clifford superalgebra $ H_{n}(\operatorname{Cl}) $ via an isomorphism
\begin{equation*}
H_{n}(\operatorname{Cl}) \cong \operatorname{Cl}_{n} \otimes \mathbb{C}S_{n}^{-}.
\end{equation*}
It follows that $ \mathbb{C}S_{n}^{-} $ and $ H_{n}(\operatorname{Cl}) $ are Morita superequivalent (see Corollary~\ref{baguette345}). In \cite[§3.3]{Wang}, Wang defined the degenerate spin affine Hecke superalgebra $ \operatorname{SH}_{n}^{\operatorname{aff}} $, which is an affinization of the superalgebra $ \mathbb{C}S_{n}^{-} $. It was shown in \cite[Thm.~4.1]{Wang} that the degenerate spin affine Hecke superalgebra $ \operatorname{SH}_{n}^{\operatorname{aff}} $ is related to the degenerate affine Hecke--Clifford superalgebra $ H_{n}^{\operatorname{aff}}(\operatorname{Cl}) $ via an isomorphism
\begin{equation} \label{glebe}
H_{n}^{\operatorname{aff}}(\operatorname{Cl}) \cong \operatorname{Cl}_{n} \otimes \operatorname{SH}_{n}^{\operatorname{aff}}.
\end{equation}
It follows that $ \operatorname{SH}_{n}^{\operatorname{aff}} $ and $ H_{n}^{\operatorname{aff}}(\operatorname{Cl}) $ are Morita superequivalent. We set $ \operatorname{SH}_{n}^{Q} $ (resp. $ H_{n}^{Q}(\operatorname{Cl}) $) to be the cyclotomic quotient of $ \operatorname{SH}_{n}^{\operatorname{aff}} $ (resp. $ H_{n}^{\operatorname{aff}}(\operatorname{Cl}) $) associated to a polynomial $ Q \in \mathbb{C}[x^{2}] \setminus 0 $.

\subsection{Tensor product Hecke algebras}

In recent years, there has been significant interest in ``tensor product'' Hecke algebras. Loosely speaking, these algebras are obtained by introducing red strands, which then interact with the black strands corresponding to the original Hecke algebra. This graphical operation was first introduced by Webster \cite{Webster}, who defined the tensor product algebras, which were used to categorify tensor products of simple modules over a quantum group. Another example of such tensor product Hecke algebras is provided by the higher-level affine Hecke algebras (or the type \(F\) affine Hecke algebras) from \cite[Def.~2.7]{Maksimau-Stroppel} and \cite[Def.~5.5]{Webster2}. A further example is given by the higher-level affine wreath product superalgebras, which were defined by the author in \cite{Moran}. These superalgebras are dependent on a Frobenius superalgebra \(A\). Specializing \(A\) to be the two-dimensional Clifford superalgebra, the higher-level affine wreath product superalgebras yield a tensor product version of the degenerate affine Hecke--Clifford superalgebras.

The goal of this paper is to combine the above-mentioned approaches by defining tensor product versions of the degenerate spin affine Hecke superalgebras. Furthermore, motivated by \cref{glebe}, we show that these superalgebras are Morita superequivalent to the tensor product degenerate affine Hecke--Clifford superalgebras.

\subsection{Objects of study} 

We now give an overview of the main objects of study in this paper. We will explain in the following subsection how these objects are related. 

In this paper, we define the \emph{tensor product degenerate spin affine Hecke supercategory} $ \mathcal{T\!AS} $. The objects of $ \mathcal{T\!AS} $ are generated by the elements of the set $ (\mathbb{C}[x^{2}] \setminus 0) \cup \{\go\} $, where $ \go $ is a formal symbol. The morphisms are generated by 
\begin{equation} \label{hotcold}
    \begin{tikzpicture}[centerzero, thick]
        \draw[-] (-0.2,-0.2) -- (0.2,0.2);
        \draw[-] (0.2,-0.2) -- (-0.2,0.2);
    \end{tikzpicture}
      \ \colon
    \go \otimes \go \to \go \otimes \go, \qquad
 \begin{tikzpicture}[anchorbase, thick]
        \draw[-] (0,0) -- (0,0.6);
        \singdot{0,0.3};
\end{tikzpicture}
\ \colon
\go \to \go, \quad
\begin{tikzpicture}[centerzero, thick]
        \draw[-] (0.2,-0.2) -- (-0.2,0.2);
        \draw[-, wei] (-0.2,-0.2) -- (0.2,0.2);
        \node at (-0.2,-0.4) {\tiny $ Q $};
\end{tikzpicture}
\ \colon Q \otimes \go \to \go \otimes Q, \qquad
\begin{tikzpicture}[centerzero, thick]
        \draw[-] (-0.2,-0.2) -- (0.2,0.2);
        \draw[-,wei] (0.2,-0.2) -- (-0.2,0.2);
        \node at (0.2,-0.4) {\tiny $ Q $};
\end{tikzpicture}
\colon \go \otimes Q \to Q \otimes \go, 
\end{equation}
for all $ Q \in \mathbb{C}[x^{2}] \setminus 0 $, and they satisfy the relations \cref{spin1}, \cref{spin2}, and \cref{spin3,spin4,spin6,spin8}. Here, the red-black crossings are even, while the black-black crossing $ \begin{tikzpicture}[centerzero, thick]
        \draw[-] (-0.2,-0.2) -- (0.2,0.2);
        \draw[-] (0.2,-0.2) -- (-0.2,0.2);
    \end{tikzpicture} $ and dot $  \begin{tikzpicture}[anchorbase, thick]
        \draw[-] (0,-0.2) -- (0,0.2);
        \singdot{0,0};
\end{tikzpicture} $ are odd. The degenerate spin affine Hecke supercategory $ \mathcal{AS} $ (generated by $  \begin{tikzpicture}[centerzero, thick]
        \draw[-] (-0.2,-0.2) -- (0.2,0.2);
        \draw[-] (0.2,-0.2) -- (-0.2,0.2);
    \end{tikzpicture} $ and $ \begin{tikzpicture}[anchorbase, thick]
        \draw[-] (0,-0.2) -- (0,0.2);
        \singdot{0,0};
\end{tikzpicture} $) is a monoidal subcategory of $ \mathcal{T\!AS} $, and the odd polynomial supercategory $ \mathpzc{OPol} $ (generated by $ \begin{tikzpicture}[anchorbase, thick]
        \draw[-] (0,-0.2) -- (0,0.2);
        \singdot{0,0};
\end{tikzpicture} $) faithfully embeds into $ \mathcal{T\!AS} $. The \emph{tensor product degenerate spin affine Hecke superalgebra} $ \operatorname{SH}^{\operatorname{aff}}_{d,\mathbf{Q}} $ is defined as a path algebra of $ \mathcal{T\!AS} $, and it admits a natural cyclotomic quotient $ \operatorname{SH}^{\operatorname{cyc}}_{d,\mathbf{Q}} $. These superalgebras are dependent on an integer $ d \in \mathbb{N} $ and a word $ \mathbf{Q} $, where the entries in the word $ \mathbf{Q} $ are elements of $ \mathbb{C}[x^{2}] \setminus 0 $. If $ \mathbf{Q} = \varnothing $ is the empty word, then $ \operatorname{SH}^{\operatorname{aff}}_{d,\mathbf{Q}} \cong \operatorname{SH}^{\operatorname{aff}}_{d} $. If $ \mathbf{Q} = Q $ for some $ Q \in \mathbb{C}[x^{2}] \setminus 0 $ (that is, $ \mathbf{Q} $ is a word of length one), then $ \operatorname{SH}^{\operatorname{cyc}}_{d,\mathbf{Q}} = \operatorname{SH}^{\operatorname{cyc}}_{d,Q} \cong   \operatorname{SH}^{Q}_{d} $.

A second supercategory which we define in this paper is $ \mathcal{T\!AS}(\operatorname{Cl}) $. This supercategory is obtained from $ \mathcal{T\!AS} $ by adding an additional odd generating morphism $ \begin{tikzpicture}[centerzero, thick]
        \draw[-] (0,-0.2) -- (0,0.2);
        \xtoken{0,0};
    \end{tikzpicture} $, called the \emph{Clifford token}, which squares to the identity, and supercommutes with the generating morphisms \cref{hotcold}. The supercategory $ \mathcal{AS}(\operatorname{Cl}) $ (generated by $ \begin{tikzpicture}[centerzero, thick]
        \draw[-] (-0.2,-0.2) -- (0.2,0.2);
        \draw[-] (0.2,-0.2) -- (-0.2,0.2);
    \end{tikzpicture} $, $ \begin{tikzpicture}[anchorbase, thick]
        \draw[-] (0,-0.2) -- (0,0.2);
        \singdot{0,0};
\end{tikzpicture} $, and $ \begin{tikzpicture}[anchorbase, thick]
        \draw[-] (0,-0.2) -- (0,0.2);
        \xtoken{0,0};
\end{tikzpicture} $) is a monoidal subcategory of $ \mathcal{T\!AS}(\operatorname{Cl}) $, and the supercategory $ \mathpzc{OPol}(\operatorname{Cl}) $ (generated by $ \begin{tikzpicture}[anchorbase, thick]
        \draw[-] (0,-0.2) -- (0,0.2);
        \singdot{0,0};
\end{tikzpicture} $ and $ \begin{tikzpicture}[anchorbase, thick]
        \draw[-] (0,-0.2) -- (0,0.2);
        \xtoken{0,0};
\end{tikzpicture} $) faithfully embeds into $ \mathcal{T\!AS}(\operatorname{Cl}) $. The path algebras of $ \mathcal{T\!AS}(\operatorname{Cl}) $ are denoted by $ \operatorname{SH}^{\operatorname{aff}}_{d,\mathbf{Q}}(\operatorname{Cl}) $, and their cyclotomic quotients are denoted by $ \operatorname{SH}^{\operatorname{cyc}}_{d,\mathbf{Q}}(\operatorname{Cl}) $.

A third supercategory which we define in this paper is the \emph{tensor product degenerate affine Hecke--Clifford supercategory} $ \mathcal{T\!AH}(\operatorname{Cl}) $. The objects of $ \mathcal{T\!AH}(\operatorname{Cl}) $ are generated by the elements of the set $ (\mathbb{C}[x^{2}] \setminus 0) \cup \{\go\} $. The morphisms are generated by 
\begin{equation} \label{balcony}
    \begin{tikzpicture}[centerzero, thick]
        \draw[-] (-0.2,-0.2) -- (0.2,0.2);
        \draw[-] (0.2,-0.2) -- (-0.2,0.2);
    \end{tikzpicture}
      \ \colon
    \go \otimes \go \to \go \otimes \go,
    \quad
    \begin{tikzpicture}[anchorbase, thick]
        \draw[-] (0,-0.3) -- (0,0.3);
        \bluetoken{0,0};
    \end{tikzpicture}
    \colon \go \to \go, \quad
 \begin{tikzpicture}[anchorbase, thick]
        \draw[-] (0,0) -- (0,0.6);
        \singdot{0,0.3};
\end{tikzpicture}
\ \colon
\go \to \go, \quad
\begin{tikzpicture}[centerzero, thick]
        \draw[-] (0.2,-0.2) -- (-0.2,0.2);
        \draw[-, wei] (-0.2,-0.2) -- (0.2,0.2);
        \node at (-0.2,-0.4) {\tiny $ Q $};
\end{tikzpicture}
\ \colon Q \otimes \go \to \go \otimes Q, \quad
\begin{tikzpicture}[centerzero, thick]
        \draw[-] (-0.2,-0.2) -- (0.2,0.2);
        \draw[-,wei] (0.2,-0.2) -- (-0.2,0.2);
        \node at (0.2,-0.4) {\tiny $ Q $};
\end{tikzpicture}
\colon \go \otimes Q \to Q \otimes \go, 
\end{equation}
for all $ Q \in \mathbb{C}[x^{2}] \setminus 0 $, and they satisfy the relations \cref{Pol(Cl)}, \cref{Cliff_hw-1}, \cref{Cliff_hw0}, \cref{Cliff_hw1,Cliff_hw2,Cliff_hw3,Cliff_hw5,Cliff_green}. Here, the \emph{Clifford token} $  \begin{tikzpicture}[anchorbase, thick]
        \draw[-] (0,-0.2) -- (0,0.2);
        \bluetoken{0,0};
    \end{tikzpicture} $ is odd, and the other generating morphisms in \cref{balcony} are even. The degenerate affine Hecke--Clifford supercategory $ \mathcal{AH}(\operatorname{Cl}) $ (generated by $ \begin{tikzpicture}[centerzero, thick]
        \draw[-] (-0.2,-0.2) -- (0.2,0.2);
        \draw[-] (0.2,-0.2) -- (-0.2,0.2);
    \end{tikzpicture} $, $ \begin{tikzpicture}[anchorbase, thick]
        \draw[-] (0,-0.2) -- (0,0.2);
        \bluetoken{0,0};
    \end{tikzpicture} $, and $  \begin{tikzpicture}[anchorbase, thick]
        \draw[-] (0,-0.2) -- (0,0.2);
        \singdot{0,0};
\end{tikzpicture} $) is a monoidal subcategory of $ \mathcal{T\!AH}(\operatorname{Cl}) $, and the Clifford polynomial supercategory $ \mathpzc{Pol}(\operatorname{Cl}) $ (generated by $ \begin{tikzpicture}[anchorbase, thick]
        \draw[-] (0,-0.2) -- (0,0.2);
        \bluetoken{0,0};
    \end{tikzpicture} $ and $  \begin{tikzpicture}[anchorbase, thick]
        \draw[-] (0,-0.2) -- (0,0.2);
        \singdot{0,0};
\end{tikzpicture} $) faithfully embeds into $ \mathcal{T\!AH}(\operatorname{Cl}) $. The \emph{tensor product degenerate affine Hecke--Clifford superalgebra} $ H_{d,\mathbf{Q}}^{\operatorname{aff}}(\operatorname{Cl}) $ is defined as a path algebra of $ \mathcal{T\!AH}(\operatorname{Cl}) $, and it admits a cyclotomic quotient $ H_{d,\mathbf{Q}}^{\operatorname{cyc}}(\operatorname{Cl}) $. These superalgebras are dependent on an integer $ d \in \mathbb{N} $ and a word $ \mathbf{Q} $, where the entries in the word $ \mathbf{Q} $ are elements of $ \mathbb{C}[x^{2}] \setminus 0 $. If $ \mathbf{Q} = \varnothing $ is the empty word, then $ H_{d,\mathbf{Q}}^{\operatorname{aff}}(\operatorname{Cl}) \cong H_{d}^{\operatorname{aff}}(\operatorname{Cl}) $. If $ \mathbf{Q} = Q $ for some $ Q \in \mathbb{C}[x^{2}] \setminus 0 $ (that is, $ \mathbf{Q} $ is a word of length one), then $ H_{d,\mathbf{Q}}^{\operatorname{cyc}}(\operatorname{Cl}) = H_{d,Q}^{\operatorname{cyc}}(\operatorname{Cl})  \cong H_{d}^{Q}(\operatorname{Cl}) $.

\subsection{Main results} 

Our first main result is the following.

\begin{customthm} {\bf A} [Theorem~\ref{Morita_equiv2}]
We have an isomorphism of monoidal supercategories 
\begin{equation} \label{side_lamp}
\mathcal{F} \colon \mathcal{T\!AS}(\operatorname{Cl}) \xrightarrow{\cong} \mathcal{T\!AH}(\operatorname{Cl}).
\end{equation}
\end{customthm} 

The isomorphism \cref{side_lamp} is an extension of the isomorphisms $ \mathpzc{OPol}(\operatorname{Cl}) \cong \mathpzc{Pol}(\operatorname{Cl}) $ and $ \mathcal{AS}(\operatorname{Cl}) \cong \mathcal{AH}(\operatorname{Cl}) $ (see Proposition~\ref{Next} and Proposition~\ref{AS-AH-iso}). In particular, the above-mentioned supercategories fit together into the following commutative diagram:
\begin{equation}
\begin{tikzcd}
\mathpzc{OPol} \arrow[d] \arrow[r] & \mathcal{AS} \arrow[r] \arrow[d] & \mathcal{T\!AS} \arrow[d] \\
\mathpzc{OPol}(\operatorname{Cl}) \arrow[d, "\cong"] \arrow[r] & \mathcal{AS}(\operatorname{Cl}) \arrow[r] \arrow[d, "\cong"] & \mathcal{T\!AS}(\operatorname{Cl}) \arrow[d, "\cong"] \\
\mathpzc{Pol}(\operatorname{Cl}) \arrow[r] & \mathcal{AH}(\operatorname{Cl}) \arrow[r] & \mathcal{T\!AH}(\operatorname{Cl})
\end{tikzcd}
\end{equation}
Our next main result is the following.

\begin{customthm} {\bf B} [Theorem~\ref{Morita_result}]
Let $ d \in \mathbb{N} $, and let $ \mathbf{Q} $ be a word whose entries lie in $ \mathbb{C}[x^{2}] \setminus 0 $. Then we have superalgebra isomorphisms
\begin{align}
H_{d,\mathcal{F}(\mathbf{Q})}^{\operatorname{aff}}(\operatorname{Cl}) \cong \operatorname{SH}^{\operatorname{aff}}_{d,\mathbf{Q}}(\operatorname{Cl}) \cong \operatorname{Cl}_{d} \otimes \operatorname{SH}^{\operatorname{aff}}_{d,\mathbf{Q}}, \label{free}
\\ H_{d,\mathcal{F}(\mathbf{Q})}^{\operatorname{cyc}}(\operatorname{Cl}) \cong \operatorname{SH}^{\operatorname{cyc}}_{d,\mathbf{Q}}(\operatorname{Cl}) \cong \operatorname{Cl}_{d} \otimes \operatorname{SH}^{\operatorname{cyc}}_{d,\mathbf{Q}}.
\end{align}
In particular, the superalgebras $ \operatorname{SH}^{\operatorname{aff}}_{d,\mathbf{Q}} $ and $ H_{d,\mathcal{F}(\mathbf{Q})}^{\operatorname{aff}}(\operatorname{Cl}) $ are Morita superequivalent, and the superalgebras $ \operatorname{SH}^{\operatorname{cyc}}_{d,\mathbf{Q}} $ and $ H_{d,\mathcal{F}(\mathbf{Q})}^{\operatorname{cyc}}(\operatorname{Cl}) $ are Morita superequivalent. 
\end{customthm} 

If $ \mathbf{Q} = \varnothing $ is the empty word, then we recover that $ \operatorname{SH}^{\operatorname{aff}}_{d} $ and $ H_{d}^{\operatorname{aff}}(\operatorname{Cl}) $ are Morita superequivalent (see \cref{glebe}). If $ \mathbf{Q} = Q $ for some $ Q \in \mathbb{C}[x^{2}] \setminus 0 $ (that is, $ \mathbf{Q} $ is a word of length one), then it follows that $ \operatorname{SH}^{Q}_{d} $ and $ H_{d}^{\psi(Q)}(\operatorname{Cl}) $ are Morita superequivalent. (Here, $ \psi \colon \mathbb{C}[x^{2}] \setminus 0 \rightarrow \mathbb{C}[x^{2}] \setminus 0 $ is some map, whose definition is given in Section~\ref{Isomorphism1}.) This information can be summarized in the following diagrams (where $ A \sim_{M} B $ means that \(A\) and \(B\) are Morita superequivalent):
\begin{equation}
\begin{tikzcd}
\operatorname{SH}^{\operatorname{aff}}_{d,\mathbf{Q}} \arrow[r, "\sim_M", <->] \arrow[d, dashed, "\mathbf{Q}=\varnothing"'] 
  & H_{d,\mathcal{F}(\mathbf{Q})}^{\operatorname{aff}}(\operatorname{Cl}) \arrow[d, dashed, "\mathbf{Q}=\varnothing"] \\
\operatorname{SH}^{\operatorname{aff}}_{d} \arrow[r, "\sim_M", <->] 
  & H_{d}^{\operatorname{aff}}(\operatorname{Cl})
\end{tikzcd} \ , 
\qquad
\begin{tikzcd}
\operatorname{SH}^{\operatorname{cyc}}_{d,\mathbf{Q}} \arrow[r, "\sim_M", <->] \arrow[d, dashed, "\mathbf{Q}=Q"'] 
  & H_{d,\mathcal{F}(\mathbf{Q})}^{\operatorname{cyc}}(\operatorname{Cl}) \arrow[d, dashed, "\mathbf{Q}=Q"] \\
\operatorname{SH}^{Q}_{d} \arrow[r, "\sim_M", <->] 
  & H_{d}^{\psi(Q)}(\operatorname{Cl})
\end{tikzcd}
\end{equation}

\subsection{Directions of future research}

In a future paper, we will prove quantum analogues of the results in the current paper. Namely, we will define tensor product versions of the spin affine Hecke superalgebra \cite{Wang2} and affine Hecke--Clifford superalgebra \cite{Jones-Nazarov}, and we will prove that they are Morita superequivalent. This result will be a generalization of \cite[Thm.~5.1]{Wang2}. We now list some other potential directions of future research: 

\begin{enumerate}
\item[\textbullet] Savage and Stuart showed in \cite[Cor.~7.5]{Savage-Stuart} that the odd nilHecke superalgebra is Morita superequivalent to the Clifford nilHecke superalgebra. More generally, Kang, Kashiwara and Tsuchioka proved in \cite[Thm.~3.13]{Kang-Kashiwara-Tsuchioka} that the quiver Hecke superalgebra is Morita superequivalent to the quiver Hecke--Clifford superalgebra. We expect that one can find tensor product analogues of these results.
\item[\textbullet] In \cite{Song-XWang}, Song and Wang defined the affine web category $ \mathpzc{QWeb}^{\bullet} $ of type \(Q\). It should be possible to define a tensor product version of $ \mathpzc{QWeb}^{\bullet} $. The resulting category should then contain both $ \mathpzc{QWeb}^{\bullet} $ and $ \mathcal{T\!AH}(\operatorname{Cl}) $ as monoidal subcategories.
\end{enumerate}

\subsection*{Note on the arXiv version}

For the interested reader, the tex file of the arXiv version of this paper includes hidden details of some straightforward computations and arguments that are omitted in the pdf file. These details can be displayed by switching the \texttt{details} toggle to true in the tex file and recompiling.

\subsection*{Acknowledgements} I would like to thank Alistair Savage for his guidance, and for his very helpful comments and suggestions whilst working on this project.

This research was supported by the Natural Sciences and Engineering Research Council of Canada (NSERC), funding reference number RGPIN-2023-03842.

\section{Preliminaries} \label{Preliminaries_paper}

In this section, we recall some basic facts about superalgebras, Morita superequivalences, monoidal supercategories, and monoids. 

\subsection{Superalgebras}

Throughout this paper, our ground field is the complex numbers $ \mathbb{C} $. We define $ \mathbb{N} $ to be the set of non-negative integers, and we define $ \mathbb{N}_{+} $ to be the set of positive integers. All algebras and modules in this paper are associative superalgebras and supermodules. For a homogeneous element $ v $ in a superalgebra or supermodule, we let $\overline{v} \in \mathbb{Z}_{2} $ denote its parity. We say that two homogeneous elements $ a,b $ in a superalgebra $ A $ \emph{commute} if 
\begin{equation*}
ab = (-1)^{\bar{a}\bar{b}}ba.
\end{equation*} 
The \emph{center} $ Z(A) $ of \(A\) is the span of all homogeneous elements $ a \in A $ that commute with every homogeneous element $ b \in A $. The \emph{even center} of \(A\) is the set of those elements in $ Z(A) $ that are even. We say that an element $ a \in A $ is \emph{central} if $ a \in Z(A) $, and we say that \(A\) is \emph{commutative} if $ Z(A) = A $. An element $ a \in A $ is a \emph{left zero-divisor} if there exists a nonzero $ b \in A $ such that $ ab = 0 $. Similarly, $ a \in A $ is a \emph{right zero-divisor} if there exists a nonzero $ b \in A $ such that $ ba = 0 $. We say that an element $ a \in A $ is \emph{regular} if \(a\) is neither a left nor right zero-divisor in \(A\). So \(a\) is regular if and only if the maps $ A \rightarrow A $, $ b \mapsto ab $, and $ A \rightarrow A $, $ b \mapsto ba $, are injective.

\begin{defn} \label{opportunity}
We define $ \mathcal{E}_{A} $ to be the set of elements in $ A $ that are even, central, and regular. We also define the set $ \widehat{\mathcal{E}}_{A} := \mathcal{E}_{A} \cup \{\go\} $, where $ \go $ is a formal symbol.
\end{defn}

If $ B $ is another superalgebra, then multiplication in the algebra $ A \otimes B $ is defined by
\begin{equation*}
    (a \otimes b) (a' \otimes b') = (-1)^{\bar a' \bar b} aa' \otimes bb'
\end{equation*}
for homogeneous $ a,a' \in A $, $ b,b' \in B $. We have an isomorphism of superalgebras $ A \otimes B \cong B \otimes A $ given by $ a \otimes b \mapsto (-1)^{\bar{a}\bar{b}} (b \otimes a) $, $ a \in A $, $ b \in B $.

\subsection{Modules over superalgebras}

If \(A\) is a superalgebra, then an $ A $-\emph{supermodule homomorphism} between \(A\)-supermodules \(V\) and \(W\) is a (not necessarily homogeneous) $ \mathbb{C} $-linear map $ f : V \rightarrow W $ such that
\begin{equation*}
f(av) = (-1)^{\bar{f}\bar{a}} af(v), \quad a \in A, \ v \in V.
\end{equation*}
We define $ A\text{-smod} $ to be the supercategory consisting of $ A $-supermodules with $ A $-supermodule homomorphisms. (In this paper, a \emph{supercategory} means a category whose morphism spaces are vector superspaces and composition is parity-preserving.) We have the \emph{parity shift functor}
\begin{equation*}
\Pi \colon A\text{-smod} \rightarrow A\text{-smod}.
\end{equation*}
For an object $ V \in A\text{-smod} $, $ \Pi(V) $ has the same underlying vector space structure as $ V $ but with opposite $ \mathbb{Z}_{2} $-grading. The action of $ A $ on $ \Pi(V) $ is given by $ a \cdot v = (-1)^{\bar{a}}av $, where $ av $ is the action on $ V $.

\subsection{Morita superequivalence} \label{Morita_superequivalence}

\begin{defn}[{\cite[§7.2]{Wang-Zhao}}] \label{Morita_sup}
Let \(A\) and \(B\) be two superalgebras. Then we say that \(A\) and \(B\) are \emph{Morita superequivalent} if there exist superfunctors $ \mathfrak{F} \colon A\text{-smod} \rightarrow B\text{-smod} $ and $ \mathfrak{G} \colon B\text{-smod} \rightarrow A\text{-smod} $ such that one of the following holds:
\begin{enumerate}[label=\roman*)]
\item There exist natural isomorphisms $ \mathfrak{F} \circ \mathfrak{G} \cong \operatorname{Id} $ and $ \mathfrak{G} \circ \mathfrak{F} \cong \operatorname{Id} $. \label{one}
\item There exist natural isomorphisms $ \mathfrak{F} \circ \mathfrak{G} \cong \operatorname{Id} \oplus \Pi $ and $ \mathfrak{G} \circ \mathfrak{F} \cong \operatorname{Id} \oplus \Pi $.
\end{enumerate}
\end{defn}

Note that, in case \ref{one} above, the superalgebras \(A\) and \(B\) are Morita equivalent in the usual sense. The most important example of Morita superequivalence needed in this paper is given in Corollary~\ref{baguette345} below.

\begin{rmk}
In the literature, there are several different but closely related definitions of Morita superequivalence; for example, see \cite[§2.4]{Kang-Kashiwara-Tsuchioka}, \cite[§2.2]{Kleshchev-Livesey}, and \cite[§3.1]{Wang2}. In this paper, we choose to use the definition of Morita superequivalence given by Wang and Zhao in \cite[§7.2]{Wang-Zhao}.
\end{rmk}

\begin{defn}
Given $ n \in \mathbb{N} $, we define the \emph{Clifford superalgebra} $ \operatorname{Cl}_{n} $ to be the free associative $ \mathbb{C} $-superalgebra on the odd generators $ c_{1},\ldots, c_{n} $, subject to the relations
\begin{align}
c_{i}^{2} &= 1, & 1 \leq i \leq n, \label{Clifford1}
\\ c_{i}c_{j} &= -c_{j}c_{i}, & 1 \leq i,j \leq n, \ i \neq j. \label{Clifford2}
\end{align}
We set $ \operatorname{Cl} := \operatorname{Cl}_{1} $.
\end{defn}

We set $ U_{n} $ to be the \emph{Clifford supermodule}, which is the unique irreducible $ \operatorname{Cl}_{n} $-supermodule up to isomorphism (see \cite[2.10]{Brundan-Kleshchev}). The supermodule $ U_{n} $ has dimension $ 2^{n/2} $ if \(n\) is even, and has dimension $ 2^{(n+1)/2} $ if \(n\) is odd.

\begin{prop} \label{baguette}
Let \(A\) and \(B\) be two superalgebras, and assume that $ B \cong A \otimes \operatorname{Cl}_{n} $ for some integer $ n \in \mathbb{N} $. Define the exact superfunctors
\begin{align*}
&\mathfrak{F} \colon A\text{-smod} \rightarrow B\text{-smod}, \qquad \mathfrak{F} := - \otimes U_{n},
\\ &\mathfrak{G} \colon B\text{-smod} \rightarrow A\text{-smod}, \qquad \mathfrak{G} := \operatorname{Hom}_{\operatorname{Cl}_{n}}(U_{n}, -).
\end{align*}
Then the following holds:
\begin{enumerate}[label=\roman*)]
\item If $ n $ is even, then $ \mathfrak{F} \circ \mathfrak{G} \cong \operatorname{Id} $ and $ \mathfrak{G} \circ \mathfrak{F} \cong \operatorname{Id} $.
\item If $ n  $ is odd, then $ \mathfrak{F} \circ \mathfrak{G} \cong \operatorname{Id} \oplus \Pi $ and $ \mathfrak{G} \circ \mathfrak{F} \cong \operatorname{Id} \oplus \Pi $.
\end{enumerate}
\end{prop}

\begin{proof}
This is analogous to the proofs of \cite[Prop.~13.2.2]{Kleshchev}, \cite[Lem.~9.9]{Brundan-Kleshchev2}, and \cite[Thm.~3.4]{Brundan-Kleshchev}. Hence it will be omitted.
\end{proof}

\begin{cor} \label{baguette345}
Let \(A\) and \(B\) be two superalgebras, and assume that $ B \cong A \otimes \operatorname{Cl}_{n} $ for some integer $ n \in \mathbb{N} $. Then the superalgebras \(A\) and \(B\) are Morita superequivalent.
\end{cor}

\subsection{Monoidal supercategories}

In this paper, we will work with \emph{strict monoidal supercategories}, in the sense of \cite{Brundan-Ellis}. We refer the reader to \cite[§2]{Brundan-Savage-Webster2} for a summary of this topic well adapted to the current work, or to \cite{Brundan-Ellis} for a thorough treatment. We summarize here a few crucial properties that play an important role in the present paper.

In a \emph{strict monoidal supercategory}, morphisms satisfy the \emph{super interchange law}:
\begin{equation}\label{interchange}
    (f' \otimes g) \circ (f \otimes g')
    = (-1)^{\bar f \bar g} (f' \circ f) \otimes (g \circ g').
\end{equation}
We denote the unit object by $ \mathds{1} $ and the identity morphism of an object $X$ by $1_X$.  We will use the usual calculus of string diagrams, representing the horizontal composition $f \otimes g$ (resp.\ vertical composition $f \circ g$) of morphisms $f$ and $g$ diagrammatically by drawing $f$ to the left of $g$ (resp.\ drawing $f$ above $g$).  Care is needed with horizontal levels in such diagrams due to the signs arising from the super interchange law:
\begin{equation}\label{intlaw}
    \begin{tikzpicture}[anchorbase]
        \draw (-0.5,-0.5) -- (-0.5,0.5);
        \draw (0.5,-0.5) -- (0.5,0.5);
        \filldraw[fill=white,draw=black] (-0.5,0.15) circle (5pt);
        \filldraw[fill=white,draw=black] (0.5,-0.15) circle (5pt);
        \node at (-0.5,0.15) {$\scriptstyle{f}$};
        \node at (0.5,-0.15) {$\scriptstyle{g}$};
    \end{tikzpicture}
    \quad=\quad
    \begin{tikzpicture}[anchorbase]
        \draw (-0.5,-0.5) -- (-0.5,0.5);
        \draw (0.5,-0.5) -- (0.5,0.5);
        \filldraw[fill=white,draw=black] (-0.5,0) circle (5pt);
        \filldraw[fill=white,draw=black] (0.5,0) circle (5pt);
        \node at (-0.5,0) {$\scriptstyle{f}$};
        \node at (0.5,0) {$\scriptstyle{g}$};
    \end{tikzpicture}
    \quad=\quad
    (-1)^{\bar f\bar g}\
    \begin{tikzpicture}[anchorbase]
        \draw (-0.5,-0.5) -- (-0.5,0.5);
        \draw (0.5,-0.5) -- (0.5,0.5);
        \filldraw[fill=white,draw=black] (-0.5,-0.15) circle (5pt);
        \filldraw[fill=white,draw=black] (0.5,0.15) circle (5pt);
        \node at (-0.5,-0.15) {$\scriptstyle{f}$};
        \node at (0.5,0.15) {$\scriptstyle{g}$};
    \end{tikzpicture}
    \ .
\end{equation}
A strict monoidal supercategory $ \mathcal{C} $ with one generating object $\go$ gives rise to a \emph{tower of superalgebras}
\begin{equation*}
\operatorname{End}_{\mathcal{C}}(\go^{\otimes n}), \quad n \in \mathbb{N}.
\end{equation*}
We use this idea to introduce various families of superalgebras in an extremely efficient way, giving a presentation of $ \mathcal{C} $ with a small number of generating morphisms and relations.  If we then wish to have a presentation of the endomorphism algebras of $ \mathcal{C} $ as \emph{superalgebras}, we use the following result.

\begin{prop} \label{zebra}
Suppose $ \mathcal{C} $ is a strict $ \mathbb{C} $-linear monoidal supercategory with one generating object $\go$ and generating morphisms $f_i \in \operatorname{End}_{\mathcal{C}}(\go^{\otimes n_{i}}) $, $i \in I$, subject to the relations $R_j \in \operatorname{End}_{\mathcal{C}}(\go^{\otimes m_{j}})$, $j \in J$.  Then for a given $ n \in \mathbb{N} $, $\operatorname{End}_{\mathcal{C}}(\go^{\otimes n})$ is generated as a superalgebra by the elements
\begin{equation}
 1_{\go}^{\otimes k} \otimes f_{i} \otimes 1_{\go}^{\otimes (n-n_{i}-k)},\quad i \in I,\ 0 \leq k \leq n-n_{i},
\end{equation}
    subject to the relations
    \begin{equation} \label{zebra1}
        1_{\go}^{\otimes k} \otimes R_j \otimes 1_{\go}^{\otimes(n-m_{j}-k)},\quad j \in J,\ 0 \leq k \leq n-m_j,
    \end{equation}
    and the relations
    \begin{equation} \label{zebra2}
        \left( 1_{\go}^{\otimes k_{1}} \otimes f_i \otimes 1_{\go}^{\otimes k_2} \right)
        \left( 1_{\go}^{\otimes l_{1}} \otimes f_j \otimes 1_{\go}^{\otimes l_2} \right)
        - (-1)^{\overline{f_i} \ \overline{f_j}}
        \left( 1_{\go}^{\otimes l_{1}} \otimes f_j \otimes 1_{\go}^{\otimes l_2} \right)
        \left( 1_{\go}^{\otimes k_{1}} \otimes f_i \otimes 1_{\go}^{\otimes k_2} \right)
    \end{equation}
    for $i,j \in I$, $k_1+n_i+k_2 = n = l_1 + n_j + l_2$, $k_2 \geq n_j + l_2$.
\end{prop}

\begin{proof}
This follows from a \emph{super} version of Theorems 5.2 and 5.4 in \cite{Liu}. Note that the relations \cref{zebra2} correspond to the super interchange law \cref{intlaw}.
\end{proof}

\subsection{Monoids}

For a given set \(X\), we define $ F(X) $ to be the free monoid on \(X\). That is, $ F(X) $ is the set of all words in \(X\), and the product of any two words is their concatenation. We denote the length of a word $ w \in F(X) $ by $ |w| $. A \emph{subword} of $ w $ is an element in $ F(X) $ which is obtained by deleting some (not necessarily consecutive) letters in \(w\).

\begin{defn} \label{snowing}
Let $ u = u_{1} \cdots u_{n} $ and $ v = v_{1} \cdots v_{m} $ be two words in $ F(X) $. Then a word $ w = w_1 \cdots w_{n+m} $ is said to be a \emph{$(u,v)$-shuffle} if there exists a partition $ \{1,\ldots,n+m\} = I \cup J $, where $ I = \{ i_{1} < \cdots < i_{n}\} $ and $ J = \{j_{1} < \cdots < j_{m}\} $, such that $ w_{i_{r}} = u_{r} $ and $w _{j_{s}} = v_{s} $ for all $ 1 \leq r \leq n $, $ 1 \leq s \leq m $. We define $ F(X)_{u,v} $ to be the set of all $ (u,v) $-shuffles in $ F(X) $.
\end{defn}

\begin{egg}
Let $ X = \{a,b,c\} $. Then the word $ abcb $ is an $ (ac,bb) $-shuffle. The words $ cbab $ and $ acbbc $ are not $ (ac,bb) $-shuffles.
\end{egg}

\begin{lem} \label{pencil_sharp_2}
Let $ X $ and $ Y $ be sets, and suppose that we have a monoid isomorphism $ h \colon F(X) \rightarrow F(Y) $. Let $ u,v \in F(X) $ be two words. Then we have
\begin{equation*} 
h(F(X)_{u,v}) = F(Y)_{h(u),h(v)}.
\end{equation*}
\end{lem}

\begin{proof}
This is clear, and so the proof will be omitted.
\end{proof}


\section{Tensor product degenerate spin affine Hecke superalgebras $ \operatorname{SH}_{d,\mathbf{Q}}^{\operatorname{aff}} $} \label{HLDSAHA_Section}

In this section, we define the tensor product degenerate spin affine Hecke superalgebras $ \operatorname{SH}_{d,\mathbf{Q}}^{\operatorname{aff}} $ and their cyclotomic quotients $ \operatorname{SH}_{d,\mathbf{Q}}^{\operatorname{cyc}} $. We begin by providing a recap of the degenerate spin affine Hecke superalgebras, which were first introduced by Wang in \cite{Wang}.

\subsection{Odd polynomial superalgebras} \label{partytime}

\begin{defn} \label{Odd_category}
We define the \emph{odd polynomial category} $ \mathpzc{OPol} $ to be the strict $ \mathbb{C} $-linear monoidal supercategory generated by one object $ \go $, and one odd morphism 
\begin{equation*}
\begin{tikzpicture}[centerzero, thick]
        \draw[-] (0,-0.2) -- (0,0.2);
        \singdot{0,0};
\end{tikzpicture} 
\ \colon \go \to \go.
\end{equation*}
We refer to the morphism $ \begin{tikzpicture}[centerzero, thick]
        \draw[-] (0,-0.2) -- (0,0.2);
        \singdot{0,0};
\end{tikzpicture} $ as a \emph{dot}. For $ n \in \mathbb{N} $, we define the \emph{odd polynomial superalgebra} to be 
\begin{equation*}
\operatorname{OPol}_{n} := \operatorname{End}_{\mathpzc{OPol}}(\go^{\otimes n}).
\end{equation*}
\end{defn}

Before stating the following proposition, note that throughout this paper, we number strands from \emph{left to right}.

\begin{prop} \label{oddrel_def}
Let $ n \in \mathbb{N}_{+} $. Then the odd polynomial superalgebra $ \operatorname{OPol}_{n} $ is isomorphic to the free associative superalgebra on the odd generators $ x_{1},\ldots,x_{n} $, modulo the relation
\begin{align} \label{oddrel}
x_{i}x_{j} &= -x_{j}x_{i}, & 1 \leq i,j \leq n, \ i \neq j.
\end{align}
Under this isomorphism, $ x_{i} $ corresponds to the dot on the \(i\)-th strand.
\end{prop}

\begin{proof}
This follows from Proposition \ref{zebra}.
\end{proof}

In what follows, we identify $ \operatorname{OPol}_{n} $ with the superalgebra presented in Proposition~\ref{oddrel_def}. The superalgebra $ \operatorname{OPol}_{n} $ has basis $ \{x_{1}^{k_{1}}\cdots x_{n}^{k_{n}} : k_{1},\ldots, k_{n} \in \mathbb{N}\} $. We set $ x := x_{1} \in \operatorname{OPol}_{1} $. Recall that, for a superalgebra \(A\), we have the set $ \mathcal{E}_{A} $ from Definition \ref{opportunity}.

\begin{lem} \label{sitting}
The even center of $ \operatorname{OPol}_{1} $ is equal to $ \mathbb{C}[x^{2}] $. Furthermore, $ \mathcal{E}_{\operatorname{OPol}_{1}} = \mathbb{C}[x^{2}] \setminus 0 $.
\end{lem}

\begin{proof}
The first statement is easy to see. The second statement then follows from the first statement and the fact that $ \operatorname{OPol}_{1} $ is a domain.
\end{proof}

\details{Every polynomial in $ x^{2} $ commutes with \(x\), and so lies in the even center of $ \operatorname{OPol}_{1} $. Conversely, every element in the even center of $ \operatorname{OPol}_{1} $ is even, and so is a polynomial in $ x^{2} $.}

Given an integer $ 1 \leq i \leq n $, there is an injective algebra homomorphism
\begin{equation} \label{homomorphism_in}
\operatorname{OPol}_{1} \rightarrow \operatorname{OPol}_{n}, \quad  x \mapsto x_{i}.
\end{equation}
We denote the image of an element $ f \in \operatorname{OPol}_{1} $ under this homomorphism by $ f_{i} $.

\begin{lem} \label{uncharted5678_1256}
Let $ n \in \mathbb{N} $, $ 1 \leq i \leq n $ and $ f \in \mathcal{E}_{\operatorname{OPol}_{1}} $. Then $ f_{i} \in \mathcal{E}_{\operatorname{OPol}_{n}} $.
\end{lem}

\begin{proof}
Since $ f \in \mathbb{C}[x^{2}] \setminus 0 $ (see Lemma~\ref{sitting}), it follows that $ f_{i} \in \mathbb{C}[x_{i}^{2}] \setminus 0 $. This implies that $ f_{i} $ is even, central, and regular in $ \operatorname{OPol}_{n} $.
\end{proof}

We have an action of the symmetric group $ S_{n} $ on $ \operatorname{OPol}_{n} $, given by $ w(x_{i}) = x_{w(i)} $ for $ w \in S_{n} $, $ 1 \leq i \leq n $. We denote the action of \(w\) on $ f \in \operatorname{OPol}_{n} $ by $ w(f) $. 

\begin{defn}[{\cite[§2.2.1]{Ellis-Khovanov-Lauda}}] \label{oddDemazure}
We define the \emph{odd Demazure operators} to be the parity-reversing $ \mathbb{C} $-linear operators $ D_{i} \colon \operatorname{OPol}_{n} \rightarrow \operatorname{OPol}_{n} $, $ 1 \leq i \leq n-1 $, determined by
\begin{equation}
D_{i}(1) = 0, \quad
D_{i}(x_{j}) =
\begin{cases}
1, & \text{if } j \in \{i,i+1\},\\
0, & \text{otherwise},
\end{cases}
\label{oddDemazure_6757}
\end{equation}
and the twisted Leibniz rule
\begin{equation} \label{oddLeibnizrule}
D_{i}(fg) = D_{i}(f)g + (-1)^{\bar{f}} s_{i}(f)D_{i}(g), \qquad f,g \in \operatorname{OPol}_{n}.
\end{equation}
We set $ D = D_{1} \colon \operatorname{OPol}_{2} \rightarrow \operatorname{OPol}_{2} $.
\end{defn}

\subsection{Degenerate spin affine Hecke superalgebras}

\begin{defn} \label{BlackBear}
We define the \emph{degenerate spin affine Hecke supercategory} $ \mathcal{AS} $ to be the strict $ \mathbb{C} $-linear monoidal supercategory generated by one object $ \go $, and two odd morphisms
\begin{equation*}
\begin{tikzpicture}[centerzero, thick]
        \draw[-] (0,-0.2) -- (0,0.2);
        \singdot{0,0};
\end{tikzpicture} 
\ \colon \go \to \go,
\qquad
    \begin{tikzpicture}[centerzero, thick]
        \draw[-] (-0.2,-0.2) -- (0.2,0.2);
        \draw[-] (0.2,-0.2) -- (-0.2,0.2);
    \end{tikzpicture}
      \ \colon
    \go \otimes \go \to \go \otimes \go,
\end{equation*}
subject to the relations 
\begin{gather} 
\label{spin1}
\begin{tikzpicture}[anchorbase, thick]
        \draw[-] (0.2,-0.5) to[out=up,in=down] (-0.2,0) to[out=up,in=down] (0.2,0.5);
        \draw[-] (-0.2,-0.5) to[out=up,in=down] (0.2,0) to[out=up,in=down] (-0.2,0.5);
\end{tikzpicture}
\ =\
\begin{tikzpicture}[anchorbase, thick]
        \draw[-] (-0.2,-0.5) -- (-0.2,0.5);
        \draw[-] (0.2,-0.5) -- (0.2,0.5);
\end{tikzpicture}
\ ,\quad
\begin{tikzpicture}[anchorbase, thick]
        \draw[-] (0.4,-0.5) -- (-0.4,0.5);
        \draw[-] (0,-0.5) to[out=up, in=down] (-0.4,0) to[out=up,in=down] (0,0.5);
        \draw[-] (-0.4,-0.5) -- (0.4,0.5);
\end{tikzpicture}
\ =\
\begin{tikzpicture}[anchorbase, thick]
        \draw[-] (0.4,-0.5) -- (-0.4,0.5);
        \draw[-] (0,-0.5) to[out=up, in=down] (0.4,0) to[out=up,in=down] (0,0.5);
        \draw[-] (-0.4,-0.5) -- (0.4,0.5);
\end{tikzpicture} \ , \\
\label{spin2}
\begin{tikzpicture}[centerzero, thick]
        \draw[-] (0.3,-0.4) -- (-0.3,0.4);
        \draw[-] (-0.3,-0.4) -- (0.3,0.4);
        \singdot{-0.15,-0.2};
\end{tikzpicture}
\ + \
\begin{tikzpicture}[centerzero, thick]
        \draw[-] (-0.3,-0.4) -- (0.3,0.4);
        \draw[-] (0.3,-0.4) -- (-0.3,0.4);
        \singdot{0.171,0.228};
\end{tikzpicture}
\ = \ 
\begin{tikzpicture}[centerzero, thick]
        \draw (-0.2,-0.4) -- (-0.2,0.4);
        \draw (0.2,-0.4) -- (0.2,0.4);
\end{tikzpicture} \ .
\end{gather}
For $ n \in \mathbb{N} $, we define the \emph{degenerate spin affine Hecke superalgebra} to be
\begin{equation*}
\operatorname{SH}^{\operatorname{aff}}_{n} := \operatorname{End}_{\mathcal{AS}}(\go^{\otimes n}).
\end{equation*}
\end{defn}

By \cref{spin1} and \cref{spin2}, we have 
\begin{equation} \label{spin2.5}
\begin{tikzpicture}[centerzero, thick]
        \draw[-] (0.3,-0.4) -- (-0.3,0.4);
        \draw[-] (-0.3,-0.4) -- (0.3,0.4);
        \singdot{-0.15,0.2};
\end{tikzpicture}
\ + \
\begin{tikzpicture}[centerzero, thick]
        \draw[-] (-0.3,-0.4) -- (0.3,0.4);
        \draw[-] (0.3,-0.4) -- (-0.3,0.4);
        \singdot{0.171,-0.228};
\end{tikzpicture}
\ = \ 
\begin{tikzpicture}[centerzero]
        \draw (-0.2,-0.4) -- (-0.2,0.4);
        \draw (0.2,-0.4) -- (0.2,0.4);
\end{tikzpicture} \ .
\end{equation}

\begin{prop} \label{down}
The degenerate spin affine Hecke superalgebra $ \operatorname{SH}^{\operatorname{aff}}_{n} $ is isomorphic to the free associative superalgebra on the odd generators $ x_{1},\ldots, x_{n}, T_{1},\ldots, T_{n-1} $, subject to the relations \cref{oddrel} and
\begin{align}
T_{i}^{2} &= 1, & 1 \leq i \leq n-1, \label{upper1}
\\ T_{i}T_{i+1}T_{i} &= T_{i+1}T_{i}T_{i+1}, & 1 \leq i \leq n-2, \label{upper2}
\\ T_{i}T_{j} &= -T_{j}T_{i}, & 1 \leq i,j \leq n-1, \ |i-j| > 1, \label{upper3}
\\ T_{i}x_{i} &= -x_{i+1}T_{i} + 1, & 1 \leq i \leq n-1, \label{upper4}
\\ T_{i}x_{j} &= -x_{j}T_{i}, & 1 \leq i \leq n-1, \ 1 \leq j \leq n, \ j \neq i,i+1. \label{upper5}
\end{align}
Under this isomorphism, $ T_{i} $ corresponds to the crossing of the \(i\)-th and $ (i+1) $-th strands, and $ x_{i} $ corresponds to the dot on the \(i\)-th strand.
\end{prop}

\begin{proof}
This follows from Proposition \ref{zebra}.
\end{proof}

In what follows, we identify $ \operatorname{SH}^{\operatorname{aff}}_{n} $ with the superalgebra presented in Proposition \ref{down}. For each $ w \in S_{n} $, choose a reduced expression $ w = s_{i_{1}}\cdots s_{i_{k}} $. Then we define
\begin{equation*}
T_{w} = T_{i_{1}}\cdots T_{i_{k}} \in \operatorname{SH}^{\operatorname{aff}}_{n}.
\end{equation*}
Up to a sign, $ T_{w} $ is independent of the choice of reduced expression.

\begin{prop}[{\cite[Prop.~3.4(1)]{Wang}}] \label{cell}
The superalgebra $ \operatorname{SH}^{\operatorname{aff}}_{n} $ has basis
\begin{equation*}
\{x_{1}^{k_{1}} \cdots x_{n}^{k_{n}} T_{w} : k_{1},\ldots,k_{n} \in \mathbb{N}, \ w \in S_{n}\}.
\end{equation*}
\end{prop}

\begin{prop}[{\cite[Prop.~3.4(3)]{Wang}}] \label{middle}
The even center of $ \operatorname{SH}^{\operatorname{aff}}_{n} $ is equal to $ \mathbb{C}[x_{1}^{2},\ldots,x_{n}^{2}]^{S_{n}} $.
\end{prop}

\begin{lem}
For all $ f \in \operatorname{OPol}_{n} $ and $ 1 \leq i \leq n-1 $, we have in $ \operatorname{SH}^{\operatorname{aff}}_{n} $ that
\begin{equation} \label{groot}
T_{i}f = (-1)^{\bar{f}}s_{i}(f)T_{i} + D_{i}(f).
\end{equation}
\end{lem}

\begin{proof}
The equation \cref{groot} clearly holds for $ f $ of polynomial degree equal to \(0\) or \(1\). Now suppose that \cref{groot} holds for $ f,g \in \operatorname{OPol}_{n} $. Then
\begin{multline*}
T_{i}(fg) = (-1)^{\bar{f}}s_{i}(f)T_{i}g + D_{i}(f)g \\ = (-1)^{\bar{f}+\bar{g}}s_{i}(fg)T_{i} + (-1)^{\bar{f}}s_{i}(f)D_{i}(g) + D_{i}(f)g \ \stackrel{\mathclap{\cref{oddLeibnizrule}}}{=} \ (-1)^{\overline{fg}}s_{i}(fg)T_{i} + D_{i}(fg),
\end{multline*}
and so the result follows by induction.
\end{proof}

\subsection{Tensor product degenerate spin affine Hecke superalgebras}

In this section, we define the tensor product degenerate spin affine Hecke superalgebras. Recall from Definition \ref{opportunity} that, for a superalgebra \(A\), the set $ \widehat{\mathcal{E}}_{A} $ is equal to $ \mathcal{E}_{A} \cup \{\go\} $.

\begin{defn} \label{HLDSAHC}
We define the \emph{tensor product degenerate spin affine Hecke supercategory} $ \mathcal{T\!AS} $ to be the strict $ \mathbb{C} $-linear monoidal supercategory defined as follows. The objects are generated by the elements in the set $ \widehat{\mathcal{E}}_{\operatorname{OPol}_{1}} $. The morphisms are generated by
\begin{gather}
    \begin{tikzpicture}[centerzero, thick]
        \draw[-] (-0.2,-0.2) -- (0.2,0.2);
        \draw[-] (0.2,-0.2) -- (-0.2,0.2);
    \end{tikzpicture}
      \ \colon
    \go \otimes \go \to \go \otimes \go, \qquad
 \begin{tikzpicture}[anchorbase, thick]
        \draw[-] (0,0) -- (0,0.6);
        \singdot{0,0.3};
\end{tikzpicture}
\ \colon
\go \to \go, \label{velocity1} \\
\begin{tikzpicture}[centerzero, thick]
        \draw[-] (0.2,-0.2) -- (-0.2,0.2);
        \draw[-, wei] (-0.2,-0.2) -- (0.2,0.2);
        \node at (-0.2,-0.4) {\tiny $ Q $};
\end{tikzpicture}
\ \colon Q \otimes \go \to \go \otimes Q, \qquad
\begin{tikzpicture}[centerzero, thick]
        \draw[-] (-0.2,-0.2) -- (0.2,0.2);
        \draw[-,wei] (0.2,-0.2) -- (-0.2,0.2);
        \node at (0.2,-0.4) {\tiny $ Q $};
\end{tikzpicture}
\colon \go \otimes Q \to Q \otimes \go, \label{velocity2}
\end{gather}
for all $ Q \in \mathcal{E}_{\operatorname{OPol}_{1}} $. The generating morphisms in \cref{velocity1} are odd, and the generating morphisms in \cref{velocity2} are even. Before stating the relations, we say a bit more about our diagrammatic conventions. When a dot is labelled by a multiplicity, we mean that we are taking its power under vertical composition. If $ f = \sum_{r\geq 0} \lambda_{r} x^{r} \in \operatorname{OPol}_{1} $, $ \lambda_{r} \in \mathbb{C} $, then we pin \(f\) to a string by defining
\begin{equation*}
    \begin{tikzpicture}[centerzero, thick]
        \draw[-] (0,-0.5) -- (0,0.5);
             \pin{0,0}{-.7,0}{f};
    \end{tikzpicture}
    = \begin{tikzpicture}[centerzero, thick]
        \draw[-] (0,-0.5) -- (0,0.5);
             \pin{0,0}{.7,0}{f};
    \end{tikzpicture}\
    :=  
\sum_{r\geq 0} \lambda_{r}
\begin{tikzpicture}[centerzero, thick]
        \draw[-] (-0.3,-0.5) -- (-0.3,0.5);
        \multdot{-0.3,0}{east}{r};
\end{tikzpicture} \ .
\end{equation*}
We define $ x := x_{1} \in \operatorname{OPol}_{2} $ and $ y := x_{2} \in \operatorname{OPol}_{2} $. If $ f = \sum_{r,s \geq 0} \lambda_{r,s} x^{r}y^{s} \in \operatorname{OPol}_{2} $, $\lambda_{r,s} \in \mathbb{C} $, then we define
\begin{equation*} 
\begin{tikzpicture}[centerzero, thick]
        \draw[-] (-0.3,-0.5) -- (-0.3,0.5);
        \draw[-] (0.3,-0.5) -- (0.3,0.5);
        \pinpin{-0.3,0}{0.3,0}{-1,0}{f};
\end{tikzpicture} 
=
\begin{tikzpicture}[centerzero, thick]
        \draw[-] (-0.3,-0.5) -- (-0.3,0.5);
        \draw[-] (0.3,-0.5) -- (0.3,0.5);
        \pinpin{0.3,0}{-0.3,0}{1,0}{f};
\end{tikzpicture} 
:=
\sum_{r,s \geq 0} \lambda_{r,s}
\begin{tikzpicture}[centerzero, thick]
        \draw[-] (-0.3,-0.5) -- (-0.3,0.5);
        \draw[-] (0.3,-0.5) -- (0.3,0.5);
        \multdot{-0.3,0}{east}{r};
        \multdot{0.3,0}{west}{s};
\end{tikzpicture} \ .
\end{equation*}
The relations on the morphisms are given by \cref{spin1}, \cref{spin2}, and 
\begin{gather}
\begin{tikzpicture}[centerzero, thick]
        \draw[-] (-0.3,-0.4) -- (0.3,0.4);
        \draw[-, wei] (0.3,-0.4) -- (-0.3,0.4);
        \singdot{-0.15,-0.2};
        \node at (0.3,-0.6) {\tiny $ Q $};
\end{tikzpicture}
= 
\begin{tikzpicture}[centerzero, thick]
        \draw[-] (-0.3,-0.4) -- (0.3,0.4);
        \draw[-, wei] (0.3,-0.4) -- (-0.3,0.4);
        \singdot{0.15,0.2};
        \node at (0.3,-0.6) {\tiny $ Q $};
\end{tikzpicture}
\ , \qquad
\begin{tikzpicture}[centerzero, thick]
        \draw[-] (0.3,-0.4) -- (-0.3,0.4);
        \draw[-, wei] (-0.3,-0.4) -- (0.3,0.4);
        \singdot{0.15,-0.2};
        \node at (-0.3,-0.6) {\tiny $ Q $};
\end{tikzpicture}
= 
\begin{tikzpicture}[centerzero, thick]
        \draw[-] (0.3,-0.4) -- (-0.3,0.4);
        \draw[-, wei] (-0.3,-0.4) -- (0.3,0.4);
        \singdot{-0.15,0.2};
        \node at (-0.3,-0.6) {\tiny $ Q $};
\end{tikzpicture} \ , \label{spin3}
\\ 
\begin{tikzpicture}[centerzero, thick]
        \draw[-] (-0.2,-0.5) to[out=up,in=down] (0.2,0) to[out=up,in=down] (-0.2,0.5);
        \draw[-, wei] (0.2,-0.5) to[out=up,in=down] (-0.2,0) to[out=up,in=down] (0.2,0.5);
        \node at (0.2,-0.7) {\tiny $ Q $};
\end{tikzpicture}
=
\begin{tikzpicture}[centerzero, thick]
        \pin{-0.2,0}{-1,0}{Q};
        \draw[-] (-0.2,-0.5) -- (-0.2,0.5);
        \draw[-, wei] (0.2,-0.5) -- (0.2,0.5);
        \node at (0.2,-0.7) {\tiny $ Q $};
\end{tikzpicture} \ , \label{spin4}
\qquad 
\begin{tikzpicture}[centerzero, thick]
        \draw[-] (0.2,-0.5) to[out=up,in=down] (-0.2,0) to[out=up,in=down] (0.2,0.5);
        \draw[-, wei] (-0.2,-0.5) to[out=up,in=down] (0.2,0) to[out=up,in=down] (-0.2,0.5);
        \node at (-0.2,-0.7) {\tiny $ Q $};
\end{tikzpicture}
\ = \
\begin{tikzpicture}[centerzero, thick]
        \draw[-, wei] (-0.2,-0.5) -- (-0.2,0.5);
        \draw[-] (0.2,-0.5) -- (0.2,0.5);
        \node at (-0.2,-0.7) {\tiny $ Q $};
        \pin{0.2,0}{1,0}{Q};
\end{tikzpicture} \ , 
\\ 
\begin{tikzpicture}[centerzero, thick]
        \draw[-] (0.4,-0.5) -- (-0.4,0.5);
        \draw[-] (0,-0.5) to[out=up, in=down] (-0.4,0) to[out=up,in=down] (0,0.5);
        \draw[-, wei] (-0.4,-0.5) -- (0.4,0.5);
        \node at (-0.4,-0.7) {\tiny $ Q $};
\end{tikzpicture}
\ =\
\begin{tikzpicture}[centerzero, thick]
        \draw[-] (0.4,-0.5) -- (-0.4,0.5);
        \draw[-] (0,-0.5) to[out=up, in=down] (0.4,0) to[out=up,in=down] (0,0.5);
        \draw[-, wei] (-0.4,-0.5) -- (0.4,0.5);
        \node at (-0.4,-0.7) {\tiny $ Q $};
\end{tikzpicture} \ , \label{spin6}
\qquad
\begin{tikzpicture}[centerzero, thick]
        \draw[-] (0,-0.5) to[out=up, in=down] (-0.4,0) to[out=up,in=down] (0,0.5);
        \draw[-] (-0.4,-0.5) -- (0.4,0.5);
        \draw[-, wei] (0.4,-0.5) -- (-0.4,0.5);
        \node at (0.4,-0.7) {\tiny $ Q $};
\end{tikzpicture}
\ =\
\begin{tikzpicture}[centerzero, thick]
        \draw[wipe] (0,-0.5) to[out=up, in=down] (0.4,0) to[out=up,in=down] (0,0.5);
        \draw[-] (0,-0.5) to[out=up, in=down] (0.4,0) to[out=up,in=down] (0,0.5);
        \draw[-] (-0.4,-0.5) -- (0.4,0.5);
        \draw[-, wei] (0.4,-0.5) -- (-0.4,0.5);
        \node at (0.4,-0.7) {\tiny $ Q $};
\end{tikzpicture}
\ , 
\\ \label{spin8}
\begin{tikzpicture}[centerzero, thick]
        \draw[-] (0.4,-0.5) -- (-0.4,0.5);
        \draw[-] (-0.4,-0.5) -- (0.4,0.5);
        \node at (0,-0.7) {\tiny $ Q $};
        \draw[-, wei] (0,-0.5) to[out=up, in=down] (-0.4,0) to[out=up,in=down] (0,0.5);
\end{tikzpicture}
\ =\
\begin{tikzpicture}[centerzero, thick]
        \draw[-] (2,-0.5) -- (1.2,0.5);
        \draw[-] (1.2,-0.5) -- (2,0.5);
        \node at (1.6,-0.7) {\tiny $ Q $};
        \draw[-, wei] (1.6,-0.5) to[out=up, in=down] (2,0) to[out=up,in=down] (1.6,0.5);
\end{tikzpicture}
\ + \
\begin{tikzpicture}[centerzero, thick]
        \node at (0,-0.7) {\tiny $ Q $};
        \draw[-] (-0.6,-0.5) -- (-0.6,0.5);
        \draw[-, wei] (0,-0.5) -- (0,0.5);
        \draw[-] (0.6,-0.5) -- (0.6,0.5);
        \pinpin{0.6,0}{-0.6,0}{1.8,0}{D(Q_{1})};
\end{tikzpicture} \ ,
\end{gather}
for all $ Q \in \mathcal{E}_{\operatorname{OPol}_{1}} $. This concludes the definition of $ \mathcal{T\!AS} $.
\end{defn}

Recall that, for a given set \(X\), $ F(X) $ is the free monoid on \(X\). Then the set of objects in $ \mathcal{T\!AS} $ is $ F(\widehat{\mathcal{E}}_{\operatorname{OPol}_{1}}) $. 

\begin{defn} \label{fluffy_spin}
Given an integer $ d \in \mathbb{N} $ and a word $ \mathbf{Q} \in F(\mathcal{E}_{\operatorname{OPol}_{1}}) $, we define $ \Gamma_{d,\mathbf{Q}} $ to be the set of $ (\go^{d},\mathbf{Q}) $-shuffles in $ F(\widehat{\mathcal{E}}_{\operatorname{OPol}_{1}}) $. That is, 
\begin{equation} \label{fluffy_5678}
\Gamma_{d,\mathbf{Q}} = F(\widehat{\mathcal{E}}_{\operatorname{OPol}_{1}})_{\go^{d},\mathbf{Q}},
\end{equation}
where $ F(\widehat{\mathcal{E}}_{\operatorname{OPol}_{1}})_{\go^{d},\mathbf{Q}} $ is given as in Definition~\ref{snowing}. 
\end{defn}

\begin{defn} \label{HLDSAHA}
Let $ d \in \mathbb{N}$ and $ \mathbf{Q} \in F(\mathcal{E}_{\operatorname{OPol}_{1}}) $. Then we define the $ (d,\mathbf{Q}) $-\emph{degenerate spin affine Hecke superalgebra} to be 
\begin{equation} \label{HLDSAHA22}
\operatorname{SH}^{\operatorname{aff}}_{d,\mathbf{Q}} := \operatorname{End}_{\operatorname{Add}(\mathcal{T\!AS})}\left(\bigoplus_{\mathbf{i} \in \Gamma_{d,\mathbf{Q}}}\mathbf{i}\right).
\end{equation}
Here, $ \operatorname{Add}(\mathcal{T\!AS}) $ is the additive envelope of $ \mathcal{T\!AS} $. 
\end{defn}

We will often refer to the superalgebra $ \operatorname{SH}^{\operatorname{aff}}_{d,\mathbf{Q}} $ as a \emph{tensor product degenerate spin affine Hecke superalgebra}. An arbitrary element in $ \operatorname{SH}^{\operatorname{aff}}_{d,\mathbf{Q}} $ is a $ \mathbb{C} $-linear combination of diagrams containing dots and crossings. The parameter \(d\) corresponds to the number of black strands in the diagrams of $ \operatorname{SH}^{\operatorname{aff}}_{d,\mathbf{Q}} $, and the length $ |\mathbf{Q}| $ of $ \mathbf{Q} $ corresponds to the number of red strands. The product of two diagrams in $ \operatorname{SH}^{\operatorname{aff}}_{d,\mathbf{Q}} $ is their vertical composition if this is defined, and if their vertical composition is not defined, then the product is zero.

\begin{egg}
Suppose $ d = 2 $ and $ \mathbf{Q} = QQ' $ for some $ Q, Q' \in \mathcal{E}_{\operatorname{OPol}_{1}} $. Then two such elements of $ \operatorname{SH}^{\operatorname{aff}}_{d,\mathbf{Q}} $ are given below:
\begin{equation*} 
E_{1} = 
\begin{tikzpicture}[centerzero, thick]
         \draw[-] (0.6,-0.7) to[out=up, in=down] (0.9,0) to[out=up, in=down] (0.3,0.7);
         \draw[-] (0.9,-0.7) to[out=up, in=down] (0,0.7);
         \draw[-,wei] (0,-0.7) to[out=up, in=down] (0.6,0.7);
         \draw[-,wei] (0.3,-0.7) to[out=up, in=down] (0.9,0.7);
         \singdot{0.9,0};
         \node at (0,-0.9) {\tiny $ Q $};
         \node at (0.3,-0.9) {\tiny $ Q' $};
\end{tikzpicture}
+ 3
\begin{tikzpicture}[centerzero, thick]
         \draw[-] (0,-0.7) to[out=up, in=down] (0.3,0.7);
         \draw[-,wei] (0.3,-0.7) to[out=up, in=down] (-0.3,0.7);
         \draw[-,wei] (0.6,-0.7) -- (0.6,0.7);
         \draw[-] (0.9,-0.7) -- (0.9,0.7);
         \singdot{0.9,-0.2};
         \node at (0.3,-0.9) {\tiny $ Q $};
         \node at (0.6,-0.9) {\tiny $ Q' $};
\end{tikzpicture} \ , \quad 
E_{2} = 
\begin{tikzpicture}[centerzero, thick]
         \draw[-] (0.6,-0.7) to[out=up, in=down] (0,0.7);
         \draw[-] (0,-0.7) to[out=up, in=down] (0.9,0.7);
         \draw[-,wei] (0.9,-0.7) to[out=up, in=down] (0.6,0.7);
         \draw[-,wei] (0.3,-0.7) to[out=up, in=down] (0,0) to[out=up, in=down] (0.3,0.7);
         \node at (0.3,-0.9) {\tiny $ Q $};
         \node at (0.9,-0.9) {\tiny $ Q' $};
\end{tikzpicture} \ .
\end{equation*}
We then have in $ \operatorname{SH}^{\operatorname{aff}}_{d,\mathbf{Q}} $ that
\begin{equation*}
E_{1}E_{2} = 3 \
\begin{tikzpicture}[centerzero, thick]
         \draw[-] (0.6,-0.7) to[out=up, in=down] (0,7/30) to[out=up, in=down] (0.3,0.7);
         \draw[-] (0,-0.7) to[out=up, in=down] (0.9,0.7);
         \draw[-,wei] (0.9,-0.7) to[out=up, in=down] (0.6,0.7);
         \draw[-,wei] (0.3,-0.7) to[out=up, in=down] (0,-7/30) to[out=up, in=down] (0.3,7/30) to[out=up, in=down] (0,0.7);
         \node at (0.3,-0.9) {\tiny $ Q $};
         \node at (0.9,-0.9) {\tiny $ Q' $};
         \singdot{0.8,0.35};
\end{tikzpicture}
, \quad
E_{2}E_{1} = 0.
\end{equation*}
\end{egg}

\begin{rmk} 
Suppose that $ \mathbf{Q} \in F(\mathcal{E}_{\operatorname{OPol}_{1}}) $ is the empty word. Then it will follow from the basis theorem (Theorem~\ref{HL-basis_Spin}) that, in this case, $ \operatorname{SH}^{\operatorname{aff}}_{d,\mathbf{Q}} $ is isomorphic to the degenerate spin affine Hecke superalgebra $ \operatorname{SH}^{\operatorname{aff}}_{d} $ from Definition~\ref{BlackBear}.
\end{rmk}

\subsection{A monoidal functor}

We have a natural functor $ \mathcal{AS} \rightarrow \mathcal{T\!AS} $ that sends the generating objects and generating morphisms in $ \mathcal{AS} $ to the ones of the same name in $ \mathcal{T\!AS} $. The following theorem provides a left inverse for this functor.

\begin{theo} \label{bright_light}
There is a strict $ \mathbb{C} $-linear monoidal functor
\begin{equation*}
\Omega \colon \mathcal{T\!AS} \rightarrow \mathcal{AS}
\end{equation*}
which sends $ \go $ to $ \go $ and $ Q \in \mathcal{E}_{\operatorname{OPol}_{1}} $ to $ \mathds{1} $ (the unit object of $ \mathcal{AS} $), and sends the generating morphisms as follows:
\begin{gather*}
\Omega \left(\begin{tikzpicture}[centerzero, thick]
      \draw[-] (0.6,-0.4) -- (0.6,0.4);
      \singdot{0.6,0};
\end{tikzpicture}\right) = 
\begin{tikzpicture}[centerzero, thick]
      \draw[-] (0.6,-0.4) -- (0.6,0.4);
      \singdot{0.6,0};
\end{tikzpicture} \ ,
\quad 
\Omega \left(\begin{tikzpicture}[centerzero, thick]
      \draw[-] (0.9,-0.4) -- (0.3,0.4);
      \draw[-] (0.3,-0.4) -- (0.9,0.4);
\end{tikzpicture}\right) 
=
\begin{tikzpicture}[centerzero, thick]
      \draw[-] (0.9,-0.4) -- (0.3,0.4);
      \draw[-] (0.3,-0.4) -- (0.9,0.4);
\end{tikzpicture} \ , 
\quad
\Omega \left(\begin{tikzpicture}[centerzero, thick]
      \draw[-] (0.9,-0.4) -- (0.3,0.4);
      \draw[-,wei] (0.3,-0.4) -- (0.9,0.4);
      \node at (0.3,-0.6) {\tiny $ Q $};
\end{tikzpicture}\right) 
=
\begin{tikzpicture}[centerzero, thick]
      \draw[-] (0.6,-0.4) -- (0.6,0.4);
\end{tikzpicture} \ ,
\quad 
\Omega \left(\begin{tikzpicture}[centerzero, thick]
      \draw[-] (0.3,-0.4) -- (0.9,0.4);
      \draw[-,wei] (0.9,-0.4) -- (0.3,0.4);
      \node at (0.9,-0.6) {\tiny $ Q $};
\end{tikzpicture}\right) 
= 
\begin{tikzpicture}[centerzero, thick]
        \draw[-] (0,-0.4) -- (0,0.4);
             \pin{0,0}{.7,0}{Q};
\end{tikzpicture} \ ,
\end{gather*}
for all $ Q \in \mathcal{E}_{\operatorname{OPol}_{1}} $. 
\end{theo}

\begin{proof}
To prove the existence of $ \Omega $, it suffices to verify the relations \cref{spin1}, \cref{spin2}, and \cref{spin3,spin4,spin6,spin8}. Here, we will only check \cref{spin8} and the second relation of \cref{spin6}, since the others are straightforward. For the rest of the proof, we fix an element $ Q \in \mathcal{E}_{\operatorname{OPol}_{1}} $. For the second relation of \cref{spin6}, we compute that
\begin{align*}
\Omega \left(\begin{tikzpicture}[centerzero, thick]
        \draw[wipe] (0,-0.5) to[out=up, in=down] (0.4,0) to[out=up,in=down] (0,0.5);
        \draw[-] (0,-0.5) to[out=up, in=down] (0.4,0) to[out=up,in=down] (0,0.5);
        \draw[-] (-0.4,-0.5) -- (0.4,0.5);
        \draw[-, wei] (0.4,-0.5) -- (-0.4,0.5);
        \node at (0.4,-0.7) {\tiny $ Q $};
\end{tikzpicture}\right)
= 
\begin{tikzpicture}[centerzero, thick]
      \draw[-] (0.9,-0.4) -- (0.3,0.4);
      \draw[-] (0.3,-0.4) -- (0.9,0.4);
      \pin{0.45,-0.2}{-0.2,-0.2}{Q};
      \pin{0.75,-0.2}{1.4,-0.2}{Q};
\end{tikzpicture}
= 
\begin{tikzpicture}[centerzero, thick]
      \draw[-] (0.9,-0.4) -- (0.3,0.4);
      \draw[-] (0.3,-0.4) -- (0.9,0.4);
      \pin{0.45,0.2}{-0.2,0.2}{Q};
      \pin{0.75,0.2}{1.4,0.2}{Q};
\end{tikzpicture}
= 
\Omega \left(\begin{tikzpicture}[centerzero, thick]
        \draw[-] (0,-0.5) to[out=up, in=down] (-0.4,0) to[out=up,in=down] (0,0.5);
        \draw[-] (-0.4,-0.5) -- (0.4,0.5);
        \draw[-, wei] (0.4,-0.5) -- (-0.4,0.5);
        \node at (0.4,-0.7) {\tiny $ Q $};
\end{tikzpicture}\right),
\end{align*}
where the second equality can be seen as follows. Since $ Q \in \mathcal{E}_{\operatorname{OPol}_{1}} $, we have by Lemma \ref{sitting} that $ Q \in \mathbb{C}[x^{2}] $. Hence $ Q_{1} \in \mathbb{C}[x_{1}^{2}] $ and $ Q_{2} \in \mathbb{C}[x_{2}^{2}] $, and so $ Q_{1}Q_{2} \in \mathbb{C}[x_{1}^{2}, x_{2}^{2}] $. Furthermore, $ Q_{1}Q_{2} $ is symmetric, and so Proposition~\ref{middle} yields that $ Q_{1}Q_{2} $ lies in the even center of $ \operatorname{SH}_{2}^{\operatorname{aff}} $. In particular, $ Q_{1}Q_{2} $ and $ T_{1} $ commute.
\\ \indent For \cref{spin8}, we compute that
\begin{equation*}
\Omega \left(\begin{tikzpicture}[centerzero, thick]
        \draw[-] (0.4,-0.5) -- (-0.4,0.5);
        \draw[-] (-0.4,-0.5) -- (0.4,0.5);
        \node at (0,-0.7) {\tiny $ Q $};
        \draw[-, wei] (0,-0.5) to[out=up, in=down] (-0.4,0) to[out=up,in=down] (0,0.5);
\end{tikzpicture}\right)
= \begin{tikzpicture}[centerzero, thick]
      \draw[-] (0.9,-0.4) -- (0.3,0.4);
      \draw[-] (0.3,-0.4) -- (0.9,0.4);
      \pin{0.45,-0.2}{-0.2,-0.2}{Q};
\end{tikzpicture}
\quad \stackrel{\mathclap{\cref{groot}}}{=} \quad \begin{tikzpicture}[centerzero, thick]
      \draw[-] (0.9,-0.4) -- (0.3,0.4);
      \draw[-] (0.3,-0.4) -- (0.9,0.4);
      \pin{0.75,0.2}{1.4,0.2}{Q};
\end{tikzpicture}
+ 
\begin{tikzpicture}[centerzero, thick]
      \draw[-] (0.7,-0.4) -- (0.7,0.4);
      \draw[-] (0.3,-0.4) -- (0.3,0.4);
        \pinpin{0.3,0}{0.7,0}{1.8,0}{D(Q_{1})};
\end{tikzpicture} 
\ = \Omega \left(\begin{tikzpicture}[centerzero, thick]
        \draw[-] (2,-0.5) -- (1.2,0.5);
        \draw[-] (1.2,-0.5) -- (2,0.5);
        \node at (1.6,-0.7) {\tiny $ Q $};
        \draw[-, wei] (1.6,-0.5) to[out=up, in=down] (2,0) to[out=up,in=down] (1.6,0.5);
\end{tikzpicture}
\ + \
\begin{tikzpicture}[centerzero, thick]
        \node at (0,-0.7) {\tiny $ Q $};
        \draw[-] (-0.6,-0.5) -- (-0.6,0.5);
        \draw[-, wei] (0,-0.5) -- (0,0.5);
        \draw[-] (0.6,-0.5) -- (0.6,0.5);
        \pinpin{0.6,0}{-0.6,0}{1.8,0}{D(Q_{1})};
\end{tikzpicture}\right). \qedhere
\end{equation*}
\end{proof}

\begin{rmk}
Using a similar proof to Theorem \ref{bright_light}, one can show that there is a strict $ \mathbb{C} $-linear monoidal functor $ \Omega' \colon \mathcal{T\!AS} \rightarrow \mathcal{AS} $ which sends $ \go $ to $ \go $ and $ Q \in \mathcal{E}_{\operatorname{OPol}_{1}} $ to $ \mathds{1} $, and sends the generating morphisms as follows:
\begin{gather*}
\Omega' \left(\begin{tikzpicture}[centerzero, thick]
      \draw[-] (0.6,-0.4) -- (0.6,0.4);
      \singdot{0.6,0};
\end{tikzpicture}\right) = 
\begin{tikzpicture}[centerzero, thick]
      \draw[-] (0.6,-0.4) -- (0.6,0.4);
      \singdot{0.6,0};
\end{tikzpicture} \ ,
\quad 
\Omega' \left(\begin{tikzpicture}[centerzero, thick]
      \draw[-] (0.9,-0.4) -- (0.3,0.4);
      \draw[-] (0.3,-0.4) -- (0.9,0.4);
\end{tikzpicture}\right) 
=
\begin{tikzpicture}[centerzero, thick]
      \draw[-] (0.9,-0.4) -- (0.3,0.4);
      \draw[-] (0.3,-0.4) -- (0.9,0.4);
\end{tikzpicture} \ , 
\quad
\Omega' \left(\begin{tikzpicture}[centerzero, thick]
      \draw[-] (0.9,-0.4) -- (0.3,0.4);
      \draw[-,wei] (0.3,-0.4) -- (0.9,0.4);
      \node at (0.3,-0.6) {\tiny $ Q $};
\end{tikzpicture}\right) 
=
\begin{tikzpicture}[centerzero, thick]
        \draw[-] (0,-0.4) -- (0,0.4);
             \pin{0,0}{.7,0}{Q};
\end{tikzpicture}  \ ,
\quad 
\Omega' \left(\begin{tikzpicture}[centerzero, thick]
      \draw[-] (0.3,-0.4) -- (0.9,0.4);
      \draw[-,wei] (0.9,-0.4) -- (0.3,0.4);
      \node at (0.9,-0.6) {\tiny $ Q $};
\end{tikzpicture}\right) 
=
\begin{tikzpicture}[centerzero, thick]
      \draw[-] (0.6,-0.4) -- (0.6,0.4);
\end{tikzpicture}\ ,
\end{gather*}
for all $ Q \in \mathcal{E}_{\operatorname{OPol}_{1}} $. 
\end{rmk}

The functor $ \Omega $ is clearly surjective on objects, and we will obtain in Corollary \ref{herd_Spin} that $ \Omega $ is faithful. However, the functor $ \Omega $ is not full (and so is not an equivalence of supercategories). Indeed, let $ Q $ be any element in $ \mathcal{E}_{\operatorname{OPol}_{1}} $ (for example, one could take $ Q = 1 $). Then it follows from the forms of the generating morphisms of $ \mathcal{T\!AS} $ in \cref{velocity1} and \cref{velocity2} that $ \operatorname{Hom}_{\mathcal{T\!AS}}(\mathds{1},Q) = 0 $. Thus the induced map
\begin{equation*}
\operatorname{Hom}_{\mathcal{T\!AS}}(\mathds{1},Q) \rightarrow \operatorname{End}_{\mathcal{AS}}(\mathds{1}) \cong \mathbb{C}, \quad r \mapsto \Omega(r)
\end{equation*}
is not surjective.

\subsection{Basis theorem} \label{keeper}

In this section, we find bases of the morphism spaces of $ \mathcal{T\!AS} $. 

\begin{lem} \label{checkmate}
Let $ \mathbf{i}, \mathbf{j} \in F(\widehat{\mathcal{E}}_{\operatorname{OPol}_{1}}) $ be two objects in $ \mathcal{T\!AS} $. Then $ \operatorname{Hom}_{\mathcal{T\!AS}}(\mathbf{i},\mathbf{j}) = 0 $ unless $ \mathbf{i}, \mathbf{j} \in \Gamma_{d,\mathbf{Q}} $ for some $ d \in \mathbb{N} $, $ \mathbf{Q} \in F(\mathcal{E}_{\operatorname{OPol}_{1}}) $.
\end{lem}

\begin{proof}
This follows from the forms of the generating morphisms of $ \mathcal{T\!AS} $ in \cref{velocity1} and \cref{velocity2}.
\end{proof}

In light of Lemma~\ref{checkmate}, for the rest of this section, we fix an integer $ d \in \mathbb{N} $ and a word $ \mathbf{Q} \in F(\mathcal{E}_{\operatorname{OPol}_{1}}) $. We aim to find a basis of $ \operatorname{Hom}_{\mathcal{T\!AS}}(\mathbf{i}, \mathbf{j}) $ for all $ \mathbf{i}, \mathbf{j} \in \Gamma_{d,\mathbf{Q}} $.

\begin{defn} \label{Generalized_crossings}
For each $ \mathbf{i}, \mathbf{j} \in \Gamma_{d,\mathbf{Q}} $, and each permutation $ w \in S_{d} $, we choose a diagram $ T_{\mathbf{j},w,\mathbf{i}} \in \operatorname{Hom}_{\mathcal{T\!AS}}(\mathbf{i}, \mathbf{j}) $ with the following properties:
\begin{enumerate}
\item[\textbullet] The diagram $ T_{\mathbf{j},w,\mathbf{i}} $ is dotless.
\item[\textbullet]  The sequence of transpositions on the black strands of $ T_{\mathbf{j},w,\mathbf{i}} $ is a reduced expression for \(w\). 
\item[\textbullet] Each pair of red and black strands cross at most once.
\end{enumerate}
\end{defn}

We note that such a choice of diagram with the above properties is not unique. However, we see no reason to prefer one choice over another.

\begin{egg}
Suppose $ d = 3 $ and $ \mathbf{Q} = QQ' $ for some $ Q,Q' \in \mathcal{E}_{\operatorname{OPol}_{1}} $. Let $ \mathbf{i} := \go Q \go Q' \go \in \Gamma_{d,\mathbf{Q}} $ and $ \mathbf{j} := \go Q Q' \go \go \in \Gamma_{d,\mathbf{Q}} $. Set $ w = s_{2}s_{1} \in S_{3} $. Then two possible choices for the element $ T_{\mathbf{j},w,\mathbf{i}} $ are given below:
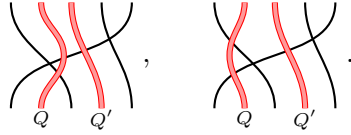
\begin{equation*}
\begin{tikzpicture}[centerzero, thick]
       \draw[-] (0.8,-0.7) to[out=up, in=down] (0,0.7);
       \draw[-] (1.6,-0.7) to[out=up, in=down] (1.2,0.7);    
       \draw[-] (0,-0.7) to[out=up, in=down] (1.6,0.7);
       \draw[-,wei] (0.4,-0.7) to[out=up, in=down] (0.7,0) to[out=up, in=down] (0.4,0.7);
       \draw[-,wei] (1.2,-0.7) to[out=up, in=down] (0.8,0.7);
       \node at (0.4,-0.85) {\tiny $ Q $};
       \node at (1.2,-0.85) {\tiny $ Q' $};
\end{tikzpicture} \ , 
\qquad
\begin{tikzpicture}[centerzero, thick]
       \draw[-] (0.8,-0.7) to[out=up, in=down] (0,0.7);
       \draw[-] (1.6,-0.7) to[out=up, in=down] (1.2,0.7);
       \draw[-] (0,-0.7) to[out=up, in=down] (1.6,0.7);
       \draw[-,wei] (0.4,-0.7) to[out=up, in=down] (0.2,0) to[out=up, in=down] (0.4,0.7);
       \draw[-,wei] (1.2,-0.7) to[out=up, in=down] (0.8,0.7);
       \node at (0.4,-0.85) {\tiny $ Q $};
       \node at (1.2,-0.85) {\tiny $ Q' $};
\end{tikzpicture} \ .
\end{equation*}
\end{egg}

For $ \mathbf{i} \in \Gamma_{d,\mathbf{Q}} $ and $ 1 \leq i \leq d $, we set $ x_{i,\mathbf{i}} \in \operatorname{End}_{\mathcal{T\!AS}}(\mathbf{i}) $ to be a dot on the $ i $-th black strand. Then, if $ \alpha = (\alpha_{1},\ldots,\alpha_{d}) \in \mathbb{N}^{d} $, we set 
\begin{equation*}
\mathbf{x}_{\mathbf{i}}^{\alpha} = (x_{1,\mathbf{i}})^{\alpha_{1}}\cdots (x_{d,\mathbf{i}})^{\alpha_{d}} \in \operatorname{End}_{\mathcal{T\!AS}}(\mathbf{i}).
\end{equation*}
We define the following subset of $ \operatorname{Hom}_{\mathcal{T\!AS}}(\mathbf{i}, \mathbf{j}) $:
\begin{align}
_{\mathbf{j}} \mathcal{B}_{\mathbf{i}} := \{\mathbf{x}_{\mathbf{j}}^{\alpha}T_{\mathbf{j},w,\mathbf{i}} : \alpha \in \mathbb{N}^{d}, \ w \in S_{d} \}.
\end{align}

\begin{lem} \label{greenpen}
Let $ \mathbf{i}, \mathbf{j} \in \Gamma_{d,\mathbf{Q}} $. Then the set $ _{\mathbf{j}} \mathcal{B}_{\mathbf{i}} $ spans $ \operatorname{Hom}_{\mathcal{T\!AS}}(\mathbf{i}, \mathbf{j}) $.
\end{lem}

\begin{proof}
We prove this by induction on the number of crossings. If \(E\) is a diagram with no crossings, then \(E\) contains only dots, from which it clearly follows that \(E\) lies in the span of $ _{\mathbf{j}} \mathcal{B}_{\mathbf{i}} $.
\\ \indent Now let $ k > 0 $, and let \(E\) be a diagram with \(k\) crossings. First, by using the relations \cref{spin2}, \cref{spin2.5}, and \cref{spin3}, one can move the dots to the top of \(E\). Note that this process may create additional terms with fewer than \(k\) crossings, but these terms all lie in the span of $ _{\mathbf{j}} \mathcal{B}_{\mathbf{i}} $ by the induction hypothesis. If there is a pair of strands in \(E\) that cross more than once, then \(E\) can be written as a $ \mathbb{C} $-linear combination of diagrams with fewer than $ k $ crossings (the proof of this is analogous to the proof of \cite[Lem.~4.10(3)]{Webster}). So we may from here on assume that any two strands in \(E\) cross at most once. In this case, the sequence of black strands in \(E\) corresponds to some reduced expression for some element $ w \in S_{d} $. We can now apply the relations \cref{spin6}, \cref{spin8}, and the second relation in \cref{spin1} to get from \(E\) to $ T_{\mathbf{j},w,\mathbf{i}} $. Note again that this process may create additional terms with fewer crossings (due to \cref{spin8}), but these terms all lie in the span of $ _{\mathbf{j}} \mathcal{B}_{\mathbf{i}} $ by the induction hypothesis. Thus we have that \(E\) lies in the span of $ _{\mathbf{j}} \mathcal{B}_{\mathbf{i}} $.
\end{proof}

Given elements $ \mathbf{i}, \mathbf{j} \in \Gamma_{d,\mathbf{Q}} $, the functor $ \Omega $ from Theorem~\ref{bright_light} induces a $ \mathbb{C} $-linear map 
\begin{equation*}
\Omega \colon \operatorname{Hom}_{\mathcal{T\!AS}}(\mathbf{i}, \mathbf{j}) \rightarrow \operatorname{SH}_{d}^{\operatorname{aff}}.
\end{equation*}

\begin{lem} \label{cactus2_1123}
Let $ \mathbf{i}, \mathbf{j} \in \Gamma_{d,\mathbf{Q}} $ and $ w \in S_{d} $. Then
\begin{equation} \label{cactus_1123}
\Omega(T_{\mathbf{j},w,\mathbf{i}}) = h_{w,w}T_{w} + \sum_{\substack{u \in S_{d} \\ L(u) < L(w)}} h_{u,w}T_{u}
\end{equation}
for some $ h_{u,w} \in \operatorname{OPol}_{d}$, where $ h_{w,w} $ is regular in $ \operatorname{OPol}_{d} $. Here, \(L\) is the length function on $ S_{d} $.
\end{lem}

\begin{proof}
Using the definition of $ \Omega $ from Theorem \ref{bright_light}, we see that $ \Omega(T_{\mathbf{j},w,\mathbf{i}}) $ is a diagram whose sequence of transpositions on the black strands is a reduced expression for \(w\), where the strands carry pins, and each such pin is labelled by an element of $ \mathcal{E}_{\operatorname{OPol}_{1}} $. For example, if $ Q,Q' \in \mathcal{E}_{\operatorname{OPol}_{1}} $, then 
\begin{equation*}
\Omega\left(\begin{tikzpicture}[centerzero, thick]
       \draw[-] (0,-0.7) to[out=up, in=down] (1.6,0.7);
       \draw[-] (0.8,-0.7) to[out=up, in=down] (0,0.7);
       \draw[-,wei] (0.4,-0.7) to[out=up, in=down] (0.2,0) to[out=up, in=down] (0.4,0.7);
       \draw[-,wei] (1.2,-0.7) to[out=up, in=down] (0.8,0.7);
       \draw[-] (1.6,-0.7) to[out=up, in=down] (1.2,0.7);
       \node at (0.4,-0.85) {\tiny $ Q $};
       \node at (1.2,-0.85) {\tiny $ Q' $};
\end{tikzpicture}\right)
=
\begin{tikzpicture}[centerzero, thick]
       \draw[-] (0,-0.7) to[out=up, in=down] (1.6,0.7);
       \draw[-] (0.8,-0.7) to[out=up, in=down] (0,0.7);
       \draw[-] (1.6,-0.7) to[out=up, in=down] (1.2,0.7);
       \pin{0.2,-0.28}{-0.6,-0.28}{Q};
       \pin{1,0.1}{0.6,0.8}{Q'};
\end{tikzpicture} \ .
\end{equation*}
Next, we inductively apply \cref{groot} to slide the pins to the top of the diagram, at the cost of adding a linear combination of diagrams with fewer than $ L(w) $ crossings. It then follows that $ \Omega(T_{\mathbf{j},w,\mathbf{i}}) $ is equal to an expansion of the form given in \cref{cactus_1123}. 

It remains to show that $ h_{w,w} $ is a regular element of $ \operatorname{OPol}_{d} $. By the process described above, $ h_{w,w} $ is some composition of pins, each labelled by an element of $ \mathcal{E}_{\operatorname{OPol}_{1}} $. By Lemma~\ref{uncharted5678_1256}, each such pin is a regular element of $ \operatorname{OPol}_{d} $. Thus, since $ h_{w,w} $ is a product of regular elements, we obtain that $ h_{w,w} $ is itself regular.
\end{proof}

\begin{theo} \label{HL-basis_Spin}
Let $ \mathbf{i}, \mathbf{j} \in \Gamma_{d,\mathbf{Q}} $. Then the set $ _{\mathbf{j}} \mathcal{B}_{\mathbf{i}} $ is a $ \mathbb{C} $-basis of $ \operatorname{Hom}_{\mathcal{T\!AS}}(\mathbf{i}, \mathbf{j}) $.
\end{theo}

\begin{proof}
The space $ \operatorname{Hom}_{\mathcal{T\!AS}}(\mathbf{i}, \mathbf{j}) $ is a left $ \operatorname{OPol}_{d} $-supermodule, with action given by 
\begin{equation*}
x_{i} \cdot y = x_{i,\mathbf{j}}y, \qquad 1 \leq i \leq d, \ y \in  \operatorname{Hom}_{\mathcal{T\!AS}}(\mathbf{i}, \mathbf{j}).
\end{equation*}
Then in order to prove the theorem, it suffices to show that $ \operatorname{Hom}_{\mathcal{T\!AS}}(\mathbf{i}, \mathbf{j}) $ is free as a left $ \operatorname{OPol}_{d} $-supermodule, with basis $ \{T_{\mathbf{j},w,\mathbf{i}} : w \in S_{d}\} $. By Lemma~\ref{greenpen}, the set $ \{T_{\mathbf{j},w,\mathbf{i}} : w \in S_{d}\} $ spans $ \operatorname{Hom}_{\mathcal{T\!AS}}(\mathbf{i}, \mathbf{j}) $ (as a left $ \operatorname{OPol}_{d} $-supermodule), and so it remains to show linear independence. Thus we assume that we have a nontrivial linear dependence equation
\begin{equation} \label{depequation3975_11}
\sum_{w \in S_{d}} f_{w}T_{\mathbf{j},w,\mathbf{i}} = 0,
\end{equation}
where $ f_{w} \in \operatorname{OPol}_{d} $. Choose an element $ v \in S_{d} $ of maximal length such that $ f_{v} \neq 0 $. Then \cref{depequation3975_11} becomes
\begin{equation} \label{depequation3976_11}
\sum_{\substack{w \in S_{d} \\ L(w) \leq L(v)}} f_{w}T_{\mathbf{j},w,\mathbf{i}} = 0.
\end{equation}
Then computing the image of \cref{depequation3976_11} under $ \Omega $, and using \cref{cactus_1123}, we obtain that
\begin{equation} \label{nook3975_11}
\sum_{\substack{w \in S_{d} \\ L(w) \leq L(v)}} f_{w}\left( h_{w,w}T_{w} + \sum_{\substack{u \in S_{d} \\ L(u) < L(w)}} h_{u,w}T_{u}\right) = 0.
\end{equation}
Then \cref{nook3975_11} and Proposition \ref{cell} imply that $ f_{v}h_{v,v} = 0 $. Since $ h_{v,v} $ is regular (see Lemma \ref{cactus2_1123}), we obtain that $ f_{v} = 0 $, which is a contradiction. The result now follows.
\end{proof}

\begin{cor} \label{herd_Spin}
The functor $ \Omega $ is faithful.
\end{cor}

\begin{proof}
This follows from the proof of Theorem \ref{HL-basis_Spin}.
\end{proof}

Given an integer $ n \in \mathbb{N}_{+} $ and a $ \mathbb{C} $-superalgebra \(A\), we define $ M_{n}(A) $ to be the superalgebra of $ n \times n $ matrices with entries in \(A\).

\begin{cor}
Let $ n := |\Gamma_{d,\mathbf{Q}}| $. Then the superalgebra $ \operatorname{SH}^{\operatorname{aff}}_{d,\mathbf{Q}} $ is isomorphic to a subalgebra of $ M_{n}(\operatorname{SH}^{\operatorname{aff}}_{d}) $.
\end{cor}

\begin{proof}
The faithful functor $ \Omega \colon \mathcal{T\!AS} \rightarrow \mathcal{AS} $ induces an injective superalgebra homomorphism
\begin{equation*}
\operatorname{SH}_{d,\mathbf{Q}}^{\operatorname{aff}} \quad \stackrel{\mathclap{\cref{HLDSAHA22}}}{=} \quad \operatorname{End}_{\operatorname{Add}(\mathcal{T\!AS})}\left(\bigoplus_{\mathbf{i} \in \Gamma_{d,\mathbf{Q}}}\mathbf{i}\right) \rightarrow \operatorname{End}_{\operatorname{Add}(\mathcal{AS})}\left(\bigoplus_{i=1}^{n} \go^{\otimes d}\right) = M_{n}(\operatorname{SH}^{\operatorname{aff}}_{d}).
\end{equation*}
The result now follows.
\end{proof}

\subsection{Description of the even center} \label{h-affwreathcenter}

In this section, we fix an integer $ d \in \mathbb{N}_{+} $ and a word $ \mathbf{Q} \in F(\mathcal{E}_{\operatorname{OPol}_{1}}) $. We aim to find the even center of the superalgebra $ \operatorname{SH}^{\operatorname{aff}}_{d,\mathbf{Q}} $. Set $ \boldsymbol{\omega} := \mathbf{Q}\go^{d} \in \Gamma_{d,\mathbf{Q}} $ to be the concatenation of the words $ \mathbf{Q} $ and $ \go^{d} $.

\begin{lem} \label{salt3}
We have an isomorphism of superalgebras $ \operatorname{End}_{\mathcal{T\!AS}}(\boldsymbol{\omega}) \cong \operatorname{SH}^{\operatorname{aff}}_{d} $.
\end{lem}

\begin{proof}
There is an obvious superalgebra homomorphism $ \operatorname{SH}^{\operatorname{aff}}_{d} \rightarrow \operatorname{End}_{\mathcal{T\!AS}}(\boldsymbol{\omega}) $ that adds $ |\mathbf{Q}| $ red strands to the left of each diagram in $ \operatorname{SH}^{\operatorname{aff}}_{d} $. It then follows from Proposition~\ref{cell} and Theorem~\ref{HL-basis_Spin} that this map is bijective.
\end{proof}

By Theorem~\ref{HL-basis_Spin}, we have an injective superalgebra homomorphism
\begin{equation} \label{lifeguard1890}
\operatorname{OPol}_{d} \rightarrow 
\operatorname{SH}^{\operatorname{aff}}_{d,\mathbf{Q}}, \quad x_{i} \mapsto \sum_{\mathbf{i} \in \Gamma_{d,\mathbf{Q}}} x_{i,\mathbf{i}}.
\end{equation}
In what follows, we view $ \operatorname{OPol}_{d} $ as a subalgebra of $ \operatorname{SH}^{\operatorname{aff}}_{d,\mathbf{Q}} $ under this superalgebra homomorphism. We also view $ \mathbb{C}[x_{1}^{2},\ldots,x_{d}^{2}] $ and $ \mathbb{C}[x_{1}^{2},\ldots,x_{d}^{2}]^{S_{d}} $ as subalgebras of $ \operatorname{SH}^{\operatorname{aff}}_{d,\mathbf{Q}} $ under the inclusions 
\begin{equation*}
\mathbb{C}[x_{1}^{2},\ldots,x_{d}^{2}]^{S_{d}} \subseteq \mathbb{C}[x_{1}^{2},\ldots,x_{d}^{2}] \subseteq \operatorname{OPol}_{d} \subseteq \operatorname{SH}^{\operatorname{aff}}_{d,\mathbf{Q}}.
\end{equation*}

\begin{prop} \label{Higher_Spin_center}
The even center of $ \operatorname{SH}^{\operatorname{aff}}_{d,\mathbf{Q}} $ is equal to $ \mathbb{C}[x_{1}^{2},\ldots,x_{d}^{2}]^{S_{d}} $.
\end{prop}

\begin{proof}
The fact that $ \mathbb{C}[x_{1}^{2},\ldots,x_{d}^{2}]^{S_{d}} $ is contained in the even center of $ \operatorname{SH}^{\operatorname{aff}}_{d,\mathbf{Q}} $ is straightforward to show. Conversely, suppose that \(z\) lies in the even center of $ \operatorname{SH}^{\operatorname{aff}}_{d,\mathbf{Q}} $. Then $ z1_{\boldsymbol{\omega}} $ is an element of the even center of $ \operatorname{End}_{\mathcal{T\!AS}}(\boldsymbol{\omega}) $. Since the isomorphism of Lemma~\ref{salt3} restricts to an isomorphism between the even center of $ \operatorname{SH}^{\operatorname{aff}}_{d} $ and the even center of $ \operatorname{End}_{\mathcal{T\!AS}}(\boldsymbol{\omega}) $, we obtain by Proposition~\ref{middle} that $ z1_{\boldsymbol{\omega}} = f1_{\boldsymbol{\omega}} $ for some $ f \in \mathbb{C}[x_{1}^{2},\ldots,x_{d}^{2}]^{S_{d}} $. Let $ \mathbf{i} \in \Gamma_{d,\mathbf{Q}} $, and set $ E = T_{\mathbf{i},1,\boldsymbol{\omega}} $ (see Definition~\ref{Generalized_crossings}). Then 
\begin{equation*}
z1_{\mathbf{i}}E = Ez1_{\boldsymbol{\omega}} = Ef1_{\boldsymbol{\omega}} = f1_{\mathbf{i}}E.
\end{equation*}
This implies that $ z1_{\mathbf{i}} = f1_{\mathbf{i}} $, since the map $ \operatorname{End}_{\mathcal{T\!AS}}(\mathbf{i}) \rightarrow \operatorname{Hom}_{\mathcal{T\!AS}}(\boldsymbol{\omega},\mathbf{i}) $, $ y \mapsto yE $ is injective by Theorem~\ref{HL-basis_Spin}. Therefore,
\begin{equation*}
z = \sum_{\mathbf{i} \in \Gamma_{d,\mathbf{Q}}} z1_{\mathbf{i}} = \sum_{\mathbf{i} \in \Gamma_{d,\mathbf{Q}}} f1_{\mathbf{i}} = f \in \mathbb{C}[x_{1}^{2},\ldots,x_{d}^{2}]^{S_{d}}. \qedhere
\end{equation*}
\end{proof}

\details{We briefly explain here why $ \mathbb{C}[x_{1}^{2},\ldots,x_{d}^{2}]^{S_{d}} $ is contained in the even center of $ \operatorname{SH}^{\operatorname{aff}}_{d,\mathbf{Q}} $. Let $ z \in \mathbb{C}[x_{1}^{2},\ldots,x_{d}^{2}]^{S_{d}} $. Then we must show that $ z $ commutes with the identity morphisms $ 1_{\mathbf{i}} $, the dots $ x_{i,\mathbf{i}} $, the red-black crossing, and the black-black crossings (since $ \operatorname{SH}^{\operatorname{aff}}_{d,\mathbf{Q}} $ is generated by these elements). It is clear that $ z1_{\mathbf{i}} = 1_{\mathbf{i}}z $ and $ zx_{i,\mathbf{i}} = x_{i,\mathbf{i}}z $. Next, we have by \cref{spin3} that \(z\) commutes with red-black crossings. Finally, \(z\) commutes with black-black crossings by Proposition~\ref{middle}.}

\begin{rmk}
Let $ \mathbf{Q}, \mathbf{Q}' \in F(\mathcal{E}_{\operatorname{OPol}_{1}}) $. Then, by Proposition~\ref{Higher_Spin_center}, the even center of $ \operatorname{SH}^{\operatorname{aff}}_{d,\mathbf{Q}} $ is isomorphic to the even center of $ \operatorname{SH}^{\operatorname{aff}}_{d,\mathbf{Q}'} $.
\end{rmk}

\subsection{Cyclotomic quotients} \label{Cyclotomic_quotients}

In this section, we define cyclotomic quotients of the tensor product degenerate spin affine Hecke superalgebras. Given a word $ \mathbf{i} \in F(\widehat{\mathcal{E}}_{\operatorname{OPol}_{1}}) $ of length $ |\mathbf{i}| \geq 1 $, we define $ \mathbf{i}_{1} \in \widehat{\mathcal{E}}_{\operatorname{OPol}_{1}} $ to be the first entry in the word $ \mathbf{i} $.

\begin{defn} \label{spin_cyclotomic_quotient}
Let $ d \in \mathbb{N}_{+} $ and $ \mathbf{Q} \in F(\mathcal{E}_{\operatorname{OPol}_{1}}) $. Then we define the \emph{cyclotomic $ (d,\mathbf{Q}) $-degenerate spin Hecke superalgebra} $ \operatorname{SH}^{\operatorname{cyc}}_{d,\mathbf{Q}} $ to be the quotient of $ \operatorname{SH}^{\operatorname{aff}}_{d,\mathbf{Q}} $ by the two-sided ideal generated by the elements in the set $ \{1_{\mathbf{i}} : \mathbf{i} \in \Gamma_{d,\mathbf{Q}}, \ \mathbf{i}_{1} = \go \} $.
\end{defn}

In other words, any diagram in $ \operatorname{SH}^{\operatorname{aff}}_{d,\mathbf{Q}} $ that has a piece of black strand to the left of all the red strands is equal to \(0\) in the quotient $ \operatorname{SH}^{\operatorname{cyc}}_{d,\mathbf{Q}} $. 

\begin{egg}
Suppose $ d = 2 $ and $ \mathbf{Q} = QQ' $ for some $ Q, Q' \in \mathcal{E}_{\operatorname{OPol}_{1}} $. Consider the elements $ E_{1}, E_{2} \in \operatorname{SH}^{\operatorname{aff}}_{d,\mathbf{Q}} $ given below:
\begin{equation*} 
E_{1} = 
\begin{tikzpicture}[centerzero, thick]
         \draw[-] (0.6,-0.7) to[out=up, in=down] (0.9,0) to[out=up, in=down] (0.3,0.7);
         \draw[-] (0.9,-0.7) to[out=up, in=down] (0,0.7);
         \draw[-,wei] (0,-0.7) to[out=up, in=down] (0.6,0.7);
         \draw[-,wei] (0.3,-0.7) to[out=up, in=down] (0.9,0.7);
         \singdot{0.9,0};
         \node at (0,-0.9) {\tiny $ Q $};
         \node at (0.3,-0.9) {\tiny $ Q' $};
\end{tikzpicture}
+ 3
\begin{tikzpicture}[centerzero, thick]
         \draw[-] (0,-0.7) to[out=up, in=down] (0.3,0.7);
         \draw[-,wei] (0.3,-0.7) to[out=up, in=down] (-0.3,0.7);
         \draw[-,wei] (0.6,-0.7) -- (0.6,0.7);
         \draw[-] (0.9,-0.7) -- (0.9,0.7);
         \singdot{0.9,-0.2};
         \node at (0.3,-0.9) {\tiny $ Q $};
         \node at (0.6,-0.9) {\tiny $ Q' $};
\end{tikzpicture} \ , \quad 
E_{2} = 
\begin{tikzpicture}[centerzero, thick]
         \draw[-] (0.6,-0.7) to[out=up, in=down] (-0.3,0) to[out=up, in=down] (0.6,0.7);
         \draw[-] (0.9,-0.7) -- (0.9,0.7);
         \draw[-,wei] (0,-0.7) -- (0,0.7);
         \draw[-,wei] (0.3,-0.7) -- (0.3,0.7);
         \node at (0,-0.9) {\tiny $ Q $};
         \node at (0.3,-0.9) {\tiny $ Q' $};
\end{tikzpicture} \ .
\end{equation*}
Then both $ E_{1} $ and $ E_{2} $ lie in the kernel of the projection map $ \operatorname{SH}^{\operatorname{aff}}_{d,\mathbf{Q}} \rightarrow \operatorname{SH}^{\operatorname{cyc}}_{d,\mathbf{Q}} $.
\end{egg}

We can also define cyclotomic quotients of the degenerate spin affine Hecke superalgebras from Definition~\ref{BlackBear}.

\begin{defn}
Let $ n \in \mathbb{N}_{+} $ and $ Q \in \mathcal{E}_{\operatorname{OPol}_{1}} $. Then we define the \emph{cyclotomic degenerate spin Hecke superalgebra} $ \operatorname{SH}^{Q}_{n} $ to be the quotient of $ \operatorname{SH}^{\operatorname{aff}}_{n} $ by the two-sided ideal generated by the element $ Q_{1} \in \operatorname{OPol}_{n} $.
\end{defn}

The following proposition says that, when $ \mathbf{Q} \in F(\mathcal{E}_{\operatorname{OPol}_{1}}) $ is a word of length one, the cyclotomic $ (d,\mathbf{Q}) $-degenerate spin Hecke superalgebra is isomorphic to a cyclotomic degenerate spin Hecke superalgebra.

\begin{prop} \label{frontlines}
Let $ d \in \mathbb{N}_{+} $ and $ Q \in \mathcal{E}_{\operatorname{OPol}_{1}} $. Then we have an isomorphism of superalgebras 
\begin{equation*}
\operatorname{SH}^{\operatorname{cyc}}_{d,Q}  \cong \operatorname{SH}^{Q}_{d}.
\end{equation*}
\end{prop}

\begin{proof}
The proof of this result is analogous to the proofs of \cite[Thm.~4.18]{Webster} and \cite[Thm.~5.34]{Moran}. Hence it will be omitted.
\end{proof}

\details{We have a non-unital injective superalgebra homomorphism $ \varphi \colon \operatorname{SH}_{d}^{\operatorname{aff}} \rightarrow \operatorname{SH}_{d,Q}^{\operatorname{aff}} $ that adds a single red strand to the left of each diagram. Composing with the projection map $ p \colon \operatorname{SH}_{d,Q}^{\operatorname{aff}} \rightarrow \operatorname{SH}_{d,Q}^{\operatorname{cyc}} $ yields a unital superalgebra homomorphism $ \phi \colon \operatorname{SH}_{d}^{\operatorname{aff}} \rightarrow \operatorname{SH}_{d,Q}^{\operatorname{cyc}} $. So we have the following commutative diagram of maps:
\begin{equation*}
\begin{tikzpicture}[auto]
  \node (A) at (0,0) {$ \operatorname{SH}_{d}^{\operatorname{aff}} $};
  \node (B) at (2.5,0) {$ \operatorname{SH}_{d,Q}^{\operatorname{aff}} $};
  \node (C) at (5,0) {$ \operatorname{SH}_{d,Q}^{\operatorname{cyc}} $};
  \draw[->] (A) -- node[above] {$ \varphi $} (B);
  \draw[->] (B) -- node[above] {$ p $} (C);
  \draw[->, bend left=40] (A) to node[above] {$ \phi $} (C);
\end{tikzpicture}
\end{equation*}
By Theorem \ref{HL-basis_Spin}, there is a spanning set of $ \operatorname{SH}_{d,Q}^{\operatorname{cyc}} $ that lies in the image of $ \phi $. This gives that $ \phi $ is surjective. Set $ \mathcal{I} $ to be the two-sided ideal of $ \operatorname{SH}_{d}^{\operatorname{aff}} $ generated by $ Q_{1} $. Then in order to prove the result, it only remains to show that the kernel of $ \phi $ is equal to $ \mathcal{I} $. That the latter is contained in the former follows by an application of \cref{spin4}. Namely, we have
\begin{equation*}
\phi\left(\begin{tikzpicture}[centerzero, thick]
        \draw[-] (0.3,-0.5) -- (0.3,0.5);
        \draw[-] (0.8,-0.5) -- (0.8,0.5);
        \node at (1.4,0) {$ \cdots $};
        \pin{0.3,0}{-0.5,0}{Q};
        \draw[-] (2,-0.5) -- (2,0.5);
\end{tikzpicture}\ \right)
\ =
\begin{tikzpicture}[centerzero, thick]
        \draw[wei] (-0.9,-0.5) -- (-0.9,0.5);
        \draw[-] (0.3,-0.5) -- (0.3,0.5);
        \draw[-] (0.8,-0.5) -- (0.8,0.5);
        \node at (1.4,0) {$ \cdots $};
        \node at (-0.9,-0.7) {\tiny $ Q $};
        \pin{0.3,0}{-0.5,0}{Q};
        \draw[-] (2,-0.5) -- (2,0.5);
\end{tikzpicture}
\ =
\begin{tikzpicture}[centerzero, thick]
        \draw[-] (0.2,-0.5) to[out=up,in=down] (-0.2,0) to[out=up,in=down] (0.2,0.5);
        \draw[-, wei] (-0.2,-0.5) to[out=up,in=down] (0.2,0) to[out=up,in=down] (-0.2,0.5);
        \draw[-] (0.7,-0.5) -- (0.7,0.5);
        \draw[-] (1.9,-0.5) -- (1.9,0.5);
        \node at (1.3,0) {$ \cdots $};        
        \node at (-0.2,-0.7) {\tiny $ Q $};
\end{tikzpicture}
\ = 0.
\end{equation*}
On the other hand, suppose $ y \in \operatorname{ker}(\phi) $. Then by the injectivity of $ \varphi $, it suffices to show that $ \varphi(y) \in \varphi(\mathcal{I}) $.
\\ \indent Note first that since $ y \in \operatorname{ker}(\phi) $, and $ \phi = p \circ \varphi $, we have that $ \varphi(y) $ is in the kernel of the projection map \(p\). So $ \varphi(y) $ is a $ \mathbb{C} $-linear combination of diagrams with a piece of black strand to the left of the red strand. If $ E $ is such a diagram, then we prove by induction that $ E \in \varphi(\mathcal{I}) $, from which it will follow by $ \mathbb{C} $-linearity that $ \varphi(y) \in \varphi(\mathcal{I}) $. The value $ c $ that we induct on is half the number of red-black crossings plus the number of black-black crossings to the left of the red strand in \(E\). For the base case, if $ c = 1 $, then there is a single black strand in $ E $ which crosses over the red strand and immediately crosses back:
\begin{equation*}
\begin{tikzpicture}[centerzero, thick]
        \draw[-] (0.2,-0.5) to[out=up,in=down] (-0.2,0) to[out=up,in=down] (0.2,0.5);
        \draw[-, wei] (-0.2,-0.5) to[out=up,in=down] (0.2,0) to[out=up,in=down] (-0.2,0.5);
        \node at (-0.2,-0.7) {\tiny $ Q $};
\end{tikzpicture} \ .
\end{equation*}
Note that this double crossing may contain some dots on the black strand, but these can be moved out by the relation \cref{spin3}. Then an application of \cref{spin4} shows that $ E \in \varphi(\mathcal{I}) $, giving that the base case holds. Next, if $ c > 1 $, then \(E\) either has a red-black double crossing or a red-black triangle, where the red strand is on the right:
\begin{equation*}
\begin{tikzpicture}[centerzero, thick]
        \draw[-] (0.2,-0.5) to[out=up,in=down] (-0.2,0) to[out=up,in=down] (0.2,0.5);
        \draw[-, wei] (-0.2,-0.5) to[out=up,in=down] (0.2,0) to[out=up,in=down] (-0.2,0.5);
        \node at (-0.2,-0.7) {\tiny $ Q $};
\end{tikzpicture} \ ,
\qquad 
\begin{tikzpicture}[centerzero, thick]
        \draw[-] (2,-0.5) -- (1.2,0.5);
        \draw[-] (1.2,-0.5) -- (2,0.5);
        \node at (1.6,-0.7) {\tiny $ Q $};
        \draw[-, wei] (1.6,-0.5) to[out=up, in=down] (2,0) to[out=up,in=down] (1.6,0.5);
\end{tikzpicture} \ .
\end{equation*}
Again, this double crossing or triangle may contain some dots on the black strands, but these can be moved out by the relation \cref{spin3}. Then applying \cref{spin4} if \(E\) has a double crossing, or applying \cref{spin8} if \(E\) has a triangle, shows that \(E\) can be written as a $ \mathbb{C} $-linear combination of diagrams with smaller \(c\). Thus by the induction hypothesis, we have $ E \in \varphi(\mathcal{I}) $.}

\section{The superalgebras $ \operatorname{SH}_{d,\mathbf{Q}}^{\operatorname{aff}}(\operatorname{Cl}) $} \label{higher-level_CSHA}

In this section, we define the superalgebras $ \operatorname{SH}_{d,\mathbf{Q}}^{\operatorname{aff}}(\operatorname{Cl}) $ and their cyclotomic quotients $ \operatorname{SH}_{d,\mathbf{Q}}^{\operatorname{cyc}}(\operatorname{Cl}) $. The main result of this section is Proposition~\ref{greenleaves}, which says that $ \operatorname{SH}_{d,\mathbf{Q}}^{\operatorname{aff}}(\operatorname{Cl}) $ is isomorphic to $ \operatorname{Cl}_{d} \otimes \operatorname{SH}_{d,\mathbf{Q}}^{\operatorname{aff}} $, and $ \operatorname{SH}_{d,\mathbf{Q}}^{\operatorname{cyc}}(\operatorname{Cl}) $ is isomorphic to $ \operatorname{Cl}_{d} \otimes \operatorname{SH}_{d,\mathbf{Q}}^{\operatorname{cyc}} $.

\subsection{The superalgebras $ \operatorname{OPol}_{n}(\operatorname{Cl}) $}

\begin{defn} \label{Odd_category_Cl}
We define the supercategory $ \mathpzc{OPol}(\operatorname{Cl}) $ to be the strict $ \mathbb{C} $-linear monoidal supercategory generated by one object $ \go $, and two odd morphisms
\begin{equation*}
\begin{tikzpicture}[centerzero, thick]
        \draw[-] (0,-0.2) -- (0,0.2);
        \singdot{0,0};
\end{tikzpicture} 
\ \colon \go \to \go, \qquad 
\begin{tikzpicture}[centerzero, thick]
        \draw[-] (0,-0.2) -- (0,0.2);
        \xtoken{0,0};
    \end{tikzpicture} \colon \go \to \go \ ,
\end{equation*}
with defining relations given by
\begin{equation} \label{spinCl123}
\begin{tikzpicture}[anchorbase, thick]
        \draw[-] (0,0) -- (0,0.7);
        \xtoken{0,0.2};
        \xtoken{0,0.45};
\end{tikzpicture}
=
\begin{tikzpicture}[anchorbase, thick]
        \draw[-] (0,0) -- (0,0.7);
\end{tikzpicture} \ ,
\quad 
\begin{tikzpicture}[centerzero, thick]
        \draw[-] (0,-0.4) -- (0,0.4);
        \xtoken{0,0.15};
        \singdot{0,-0.15};
\end{tikzpicture}
= -
\begin{tikzpicture}[centerzero, thick]
        \draw[-] (0,-0.4) -- (0,0.4);
        \xtoken{0,-0.15};
        \singdot{0,0.15};
\end{tikzpicture} \ .
\end{equation}
We refer to the morphism $ \begin{tikzpicture}[centerzero, thick]
        \draw[-] (0,-0.2) -- (0,0.2);
        \xtoken{0,0};
\end{tikzpicture} $ as the \emph{Clifford token}. For $ n \in \mathbb{N} $, we define the superalgebra
\begin{equation*}
\operatorname{OPol}_{n}(\operatorname{Cl}) := \operatorname{End}_{\mathpzc{OPol}(\operatorname{Cl})}(\go^{\otimes n}).
\end{equation*}
\end{defn}

\begin{prop} \label{grassroots}
Let $ n \in \mathbb{N}_{+} $. Then $ \operatorname{OPol}_{n}(\operatorname{Cl}) $ is isomorphic to the free associative superalgebra on the odd generators $ c_{1}, \ldots, c_{n}, x_{1}, \ldots, x_{n} $, subject to the relations \cref{Clifford1}, \cref{Clifford2}, \cref{oddrel}, and
\begin{align}
c_{i}x_{j} &= -x_{j}c_{i}, & 1 \leq i,j \leq n. \label{grassroots2}
\end{align}
Under this isomorphism, $ c_{i} $ corresponds to the Clifford token on the $ i $-th strand, and $ x_{i} $ corresponds to the dot on the $ i $-th strand.
\end{prop}

\begin{proof}
This follows from Proposition \ref{zebra}.
\end{proof}

From here onwards, we identify $ \operatorname{OPol}_{n}(\operatorname{Cl}) $ with the algebra presented in Proposition \ref{grassroots}. It follows that we have an isomorphism of superalgebras
\begin{equation}
\operatorname{OPol}_{n}(\operatorname{Cl}) \cong \operatorname{Cl}_{n} \otimes \operatorname{OPol}_{n}.
\end{equation}

\begin{lem} \label{center_Odd(Cl)}
The even center of $ \operatorname{OPol}_{1}(\operatorname{Cl}) $ is equal to $ \mathbb{C}[x^{2}] $. Furthermore, 
\begin{equation} \label{center_Odd(Cl)2}
\mathcal{E}_{\operatorname{OPol}_{1}(\operatorname{Cl})} = \mathbb{C}[x^{2}] \setminus 0 = \mathcal{E}_{\operatorname{OPol}_{1}}.
\end{equation}
\end{lem}

\begin{proof}
The first statement is straightforward to check. The equation \cref{center_Odd(Cl)2} then follows from the first statement, Lemma \ref{sitting}, and the definition of $ \mathcal{E}_{\operatorname{OPol}_{1}(\operatorname{Cl})} $.
\end{proof}

\details{We prove the first statement here. Let $ f $ be a polynomial in $ x^{2} $ with coefficients in $ \mathbb{C} $. Then it is clear that \(f\) commutes with $ c $ and $ x $, from which we obtain that \(f\) lies in the even center of $ \operatorname{OPol}_{1}(\operatorname{Cl}) $. For the reverse inclusion, if \(f\) lies in the even center of $ \operatorname{OPol}_{1}(\operatorname{Cl}) $, then we may write that $ f = \sum_{k\geq 0} a_{(k)}x^{k} $ for some $ a_{(k)} \in \operatorname{Cl} $ of parity $ k $. If \(k\) is even, then $ a_{(k)} $ must be even, and so $ a_{(k)} \in \mathbb{C} $ in this case. Thus it remains to show that $ a_{(k)} = 0 $ for all odd $ k $. We compute that
\begin{equation} \label{vroom}
0 = fx-xf = \sum_{k\geq 0} (1-(-1)^{k})a_{(k)}x^{k+1}.
\end{equation}
If $ k $ is odd, then we obtain by \cref{vroom} that $ a_{(k)} = 0 $ as required.}

\subsection{The superalgebras $ \operatorname{SH}_{n}^{\operatorname{aff}}(\operatorname{Cl}) $}

\begin{defn} \label{S(Cl)}
We define the supercategory $ \mathcal{AS}(\operatorname{Cl}) $ to be the strict $ \mathbb{C} $-linear monoidal supercategory generated by one object $ \go $, and three odd morphisms
\begin{gather}
    \begin{tikzpicture}[centerzero, thick]
        \draw[-] (-0.2,-0.2) -- (0.2,0.2);
        \draw[-] (0.2,-0.2) -- (-0.2,0.2);
    \end{tikzpicture}
      \ \colon
    \go \otimes \go \to \go \otimes \go, \qquad
 \begin{tikzpicture}[anchorbase, thick]
        \draw[-] (0,-0.2) -- (0,0.2);
        \singdot{0,0};
\end{tikzpicture}
\ \colon
\go \to \go, \qquad
    \begin{tikzpicture}[centerzero, thick]
        \draw[-] (0,-0.2) -- (0,0.2);
        \xtoken{0,0};
    \end{tikzpicture} \colon \go \to \go \ .
\end{gather}
The relations on the morphisms are given by \cref{spin1}, \cref{spin2}, \cref{spinCl123}, and 
\begin{gather} \label{spinCl1}
\begin{tikzpicture}[centerzero, thick]
        \draw[-] (0.3,-0.4) -- (-0.3,0.4);
        \draw[-] (-0.3,-0.4) -- (0.3,0.4);
        \xtoken{-0.15,-0.2};
\end{tikzpicture}
= -
\begin{tikzpicture}[centerzero, thick]
        \draw[-] (0.3,-0.4) -- (-0.3,0.4);
        \draw[-] (-0.3,-0.4) -- (0.3,0.4);
        \xtoken{-0.15,0.2};
\end{tikzpicture} \ ,
\quad
\begin{tikzpicture}[centerzero, thick]
        \draw[-] (0.3,-0.4) -- (-0.3,0.4);
        \draw[-] (-0.3,-0.4) -- (0.3,0.4);
        \xtoken{0.15,-0.2};
\end{tikzpicture}
= -
\begin{tikzpicture}[centerzero, thick]
        \draw[-] (0.3,-0.4) -- (-0.3,0.4);
        \draw[-] (-0.3,-0.4) -- (0.3,0.4);
        \xtoken{0.15,0.2};
\end{tikzpicture} \ .
\end{gather}
For $ n \in \mathbb{N} $, we define the superalgebra
\begin{equation*}
\operatorname{SH}^{\operatorname{aff}}_{n}(\operatorname{Cl}) := \operatorname{End}_{\mathcal{AS}(\operatorname{Cl})}(\go^{\otimes n}).
\end{equation*}
\end{defn}

It follows from Proposition \ref{zebra} that there is a superalgebra isomorphism
\begin{equation*}
\operatorname{SH}^{\operatorname{aff}}_{n}(\operatorname{Cl}) \cong \operatorname{Cl}_{n} \otimes \operatorname{SH}^{\operatorname{aff}}_{n}.
\end{equation*}

\subsection{The superalgebras $ \operatorname{SH}^{\operatorname{aff}}_{d,\mathbf{Q}}(\operatorname{Cl}) $}

\begin{defn} \label{LAS(Cl)}
We set $ \mathcal{T\!AS}(\operatorname{Cl}) $ to be the strict $ \mathbb{C} $-linear monoidal supercategory defined as follows. The objects are generated by the elements in the set $ \widehat{\mathcal{E}}_{\operatorname{OPol}_{1}} $, and the morphisms are generated by
\begin{gather}
    \begin{tikzpicture}[centerzero, thick]
        \draw[-] (-0.2,-0.2) -- (0.2,0.2);
        \draw[-] (0.2,-0.2) -- (-0.2,0.2);
    \end{tikzpicture}
      \ \colon
    \go \otimes \go \to \go \otimes \go, \qquad
 \begin{tikzpicture}[anchorbase, thick]
        \draw[-] (0,-0.2) -- (0,0.2);
        \singdot{0,0};
\end{tikzpicture}
\ \colon
\go \to \go, \qquad
    \begin{tikzpicture}[centerzero, thick]
        \draw[-] (0,-0.2) -- (0,0.2);
        \xtoken{0,0};
    \end{tikzpicture} \colon \go \to \go \ ,
\label{tuxedo1} \\
\begin{tikzpicture}[centerzero, thick]
        \draw[-] (0.2,-0.2) -- (-0.2,0.2);
        \draw[-, wei] (-0.2,-0.2) -- (0.2,0.2);
        \node at (-0.2,-0.4) {\tiny $ Q $};
\end{tikzpicture}
\ \colon Q \otimes \go \to \go \otimes Q, \qquad
\begin{tikzpicture}[centerzero, thick]
        \draw[-] (-0.2,-0.2) -- (0.2,0.2);
        \draw[-,wei] (0.2,-0.2) -- (-0.2,0.2);
        \node at (0.2,-0.4) {\tiny $ Q $};
\end{tikzpicture}
\colon \go \otimes Q \to Q \otimes \go, \label{tuxedo2}
\end{gather}
for all $ Q \in \mathcal{E}_{\operatorname{OPol}_{1}} $. The generating morphisms in \cref{tuxedo1} are odd, and the generating morphisms in \cref{tuxedo2} are even. The relations on the morphisms are given by \cref{spin1}, \cref{spin2}, \cref{spin3,spin4,spin6,spin8}, \cref{spinCl123}, \cref{spinCl1}, and 
\begin{gather} 
\label{spinCl2}
\begin{tikzpicture}[centerzero, thick]
        \draw[-] (-0.3,-0.4) -- (0.3,0.4);
        \draw[-, wei] (0.3,-0.4) -- (-0.3,0.4);
        \xtoken{-0.15,-0.2};
        \node at (0.3,-0.6) {\tiny $ Q $};
\end{tikzpicture}
= 
\begin{tikzpicture}[centerzero, thick]
        \draw[-] (-0.3,-0.4) -- (0.3,0.4);
        \draw[-, wei] (0.3,-0.4) -- (-0.3,0.4);
        \xtoken{0.15,0.2};
        \node at (0.3,-0.6) {\tiny $ Q $};
\end{tikzpicture}
\ , \quad
\begin{tikzpicture}[centerzero, thick]
        \draw[-] (0.3,-0.4) -- (-0.3,0.4);
        \draw[-, wei] (-0.3,-0.4) -- (0.3,0.4);
        \xtoken{0.15,-0.2};
        \node at (-0.3,-0.6) {\tiny $ Q $};
\end{tikzpicture}
= 
\begin{tikzpicture}[centerzero, thick]
        \draw[-] (0.3,-0.4) -- (-0.3,0.4);
        \draw[-, wei] (-0.3,-0.4) -- (0.3,0.4);
        \xtoken{-0.15,0.2};
        \node at (-0.3,-0.6) {\tiny $ Q $};
\end{tikzpicture} \ , \qquad Q \in \mathcal{E}_{\operatorname{OPol}_{1}}.
\end{gather}
\end{defn}

The set of objects in $ \mathcal{T\!AS}(\operatorname{Cl}) $ is $ F(\widehat{\mathcal{E}}_{\operatorname{OPol}_{1}}) $. Recall that, for $ d \in \mathbb{N} $ and $ \mathbf{Q} \in F(\mathcal{E}_{\operatorname{OPol}_{1}}) $, we have the set $ \Gamma_{d,\mathbf{Q}} $ from Definition \ref{fluffy_spin}.

\begin{defn} \label{HLDSAHA_Cliff}
Let $ d \in \mathbb{N} $ and $ \mathbf{Q} \in F(\mathcal{E}_{\operatorname{OPol}_{1}}) $. Then we define the superalgebra
\begin{equation} \label{HLDSAHA22_Cliff}
\operatorname{SH}^{\operatorname{aff}}_{d,\mathbf{Q}}(\operatorname{Cl}) := \operatorname{End}_{\operatorname{Add}(\mathcal{T\!AS}(\operatorname{Cl}))}\left(\bigoplus_{\mathbf{i} \in \Gamma_{d,\mathbf{Q}}}\mathbf{i}\right).
\end{equation}
Here, $ \operatorname{Add}(\mathcal{T\!AS}(\operatorname{Cl})) $ is the additive envelope of $ \mathcal{T\!AS}(\operatorname{Cl}) $. 
\end{defn}

We next define cyclotomic quotients of the superalgebras $ \operatorname{SH}_{d,\mathbf{Q}}^{\operatorname{aff}}(\operatorname{Cl}) $.

\begin{defn} 
Let $ d \in \mathbb{N}_{+} $ and $ \mathbf{Q} \in F(\mathcal{E}_{\operatorname{OPol}_{1}}) $. Then we define the superalgebra $ \operatorname{SH}^{\operatorname{cyc}}_{d,\mathbf{Q}}(\operatorname{Cl}) $ to be the quotient of $ \operatorname{SH}^{\operatorname{aff}}_{d,\mathbf{Q}}(\operatorname{Cl}) $ by the two-sided ideal generated by the elements in the set $ \{1_{\mathbf{i}} : \mathbf{i} \in \Gamma_{d,\mathbf{Q}}, \ \mathbf{i}_{1} = \go \} $.
\end{defn}

\subsection{Basis result}

In this section, we fix an integer $ d \in \mathbb{N} $ and a word $ \mathbf{Q} \in F(\mathcal{E}_{\operatorname{OPol}_{1}}) $. We will state a basis of $ \operatorname{Hom}_{\mathcal{T\!AS}(\operatorname{Cl})}(\mathbf{i}, \mathbf{j}) $ for all $ \mathbf{i}, \mathbf{j} \in \Gamma_{d,\mathbf{Q}} $. There is a monoidal functor $ \mathcal{T\!AS} \rightarrow \mathcal{T\!AS}(\operatorname{Cl}) $ which sends the generating objects and morphisms in $ \mathcal{T\!AS} $ to the objects and morphisms in $ \mathcal{T\!AS}(\operatorname{Cl}) $ of the same name. This functor induces a $ \mathbb{C} $-linear map
\begin{equation*}
\operatorname{Hom}_{\mathcal{T\!AS}}(\mathbf{i},\mathbf{j}) \rightarrow \operatorname{Hom}_{\mathcal{T\!AS}(\operatorname{Cl})}(\mathbf{i},\mathbf{j})
\end{equation*}
for each $ \mathbf{i}, \mathbf{j} \in \Gamma_{d,\mathbf{Q}} $. We use this map to view the morphism $ T_{\mathbf{j},w,\mathbf{i}} $ (see Definition \ref{Generalized_crossings}) as a morphism in $ \operatorname{Hom}_{\mathcal{T\!AS}(\operatorname{Cl})}(\mathbf{i},\mathbf{j}) $. If $ \alpha = (\alpha_{1},\ldots,\alpha_{d}) \in \mathbb{N}^{d} $, then we set $ \mathbf{c}_{\mathbf{i}}^{\alpha} = (c_{1,\mathbf{i}})^{\alpha_{1}}\cdots (c_{d,\mathbf{i}})^{\alpha_{d}} $, where $ c_{i,\mathbf{i}} \in \operatorname{End}_{\mathcal{T\!AS}(\operatorname{Cl})}(\mathbf{i}) $, for $ 1 \leq i \leq d $, is the Clifford token on the \(i\)-th black strand. We now define the following subset of $ \operatorname{Hom}_{\mathcal{T\!AS}(\operatorname{Cl})}(\mathbf{i}, \mathbf{j}) $:
\begin{align*}
_{\mathbf{j}} \mathcal{C}_{\mathbf{i}} := \{\mathbf{c}_{\mathbf{j}}^{\alpha}\mathbf{x}_{\mathbf{j}}^{\beta}T_{\mathbf{j},w,\mathbf{i}} : \alpha \in (\{0,1\})^{d}, \ \beta \in \mathbb{N}^{d}, \ w \in S_{d} \}.
\end{align*}

\begin{prop} \label{HL-basis_Spin-Cl}
Let $ \mathbf{i}, \mathbf{j} \in \Gamma_{d,\mathbf{Q}} $. Then the set $ _{\mathbf{j}} \mathcal{C}_{\mathbf{i}} $ is a $ \mathbb{C} $-basis of $ \operatorname{Hom}_{\mathcal{T\!AS}(\operatorname{Cl})}(\mathbf{i}, \mathbf{j}) $.
\end{prop}

\begin{proof}
This proposition can be shown by applying similar arguments to Section~\ref{keeper}. We omit the details here.
\end{proof}

\details{We explain here how one obtains linear independence of the elements in $ _{\mathbf{j}} \mathcal{C}_{\mathbf{i}} $. There is a strict $ \mathbb{C} $-linear monoidal functor
\begin{equation*}
\Phi \colon \mathcal{T\!AS}(\operatorname{Cl}) \rightarrow \mathcal{AS}(\operatorname{Cl})
\end{equation*}
which sends $ \go $ to $ \go $ and $ Q \in \mathcal{E}_{\operatorname{OPol}_{1}} $ to $ \mathds{1} $ (the unit object of $ \mathcal{AS}(\operatorname{Cl}) $), and sends the generating morphisms as follows:
\begin{gather*}
\Phi \left(\begin{tikzpicture}[centerzero, thick]
      \draw[-] (0.6,-0.4) -- (0.6,0.4);
      \xtoken{0.6,0};
\end{tikzpicture}\right) = \begin{tikzpicture}[centerzero, thick]
      \draw[-] (0.6,-0.4) -- (0.6,0.4);
      \xtoken{0.6,0};
\end{tikzpicture} \ , 
\quad 
\Phi \left(\begin{tikzpicture}[centerzero, thick]
      \draw[-] (0.6,-0.4) -- (0.6,0.4);
      \singdot{0.6,0};
\end{tikzpicture}\right) = 
\begin{tikzpicture}[centerzero, thick]
      \draw[-] (0.6,-0.4) -- (0.6,0.4);
      \singdot{0.6,0};
\end{tikzpicture} \ ,
\quad 
\Phi \left(\begin{tikzpicture}[centerzero, thick]
      \draw[-] (0.9,-0.4) -- (0.3,0.4);
      \draw[-] (0.3,-0.4) -- (0.9,0.4);
\end{tikzpicture}\right) 
=
\begin{tikzpicture}[centerzero, thick]
      \draw[-] (0.9,-0.4) -- (0.3,0.4);
      \draw[-] (0.3,-0.4) -- (0.9,0.4);
\end{tikzpicture} \ ,
\\ \Phi\left(\begin{tikzpicture}[centerzero, thick]
      \draw[-] (0.9,-0.4) -- (0.3,0.4);
      \draw[-,wei] (0.3,-0.4) -- (0.9,0.4);
      \node at (0.3,-0.6) {\tiny $ Q $};
\end{tikzpicture}\right) 
=
\begin{tikzpicture}[centerzero, thick]
      \draw[-] (0.6,-0.4) -- (0.6,0.4);
\end{tikzpicture} \ ,
\quad 
\Phi \left(\begin{tikzpicture}[centerzero, thick]
      \draw[-] (0.3,-0.4) -- (0.9,0.4);
      \draw[-,wei] (0.9,-0.4) -- (0.3,0.4);
      \node at (0.9,-0.6) {\tiny $ Q $};
\end{tikzpicture}\right) 
= 
\begin{tikzpicture}[centerzero, thick]
        \draw[-] (0,-0.5) -- (0,0.5);
             \pin{0,0}{.7,0}{Q};
\end{tikzpicture} \ ,
\end{gather*}
for all $ Q \in \mathcal{E}_{\operatorname{OPol}_{1}} $. This functor restricts to a $ \mathbb{C} $-linear map 
\begin{equation*}
\Phi \colon \operatorname{Hom}_{\mathcal{T\!AS}(\operatorname{Cl})}(\mathbf{i}, \mathbf{j}) \rightarrow \operatorname{SH}_{d}^{\operatorname{aff}}(\operatorname{Cl}).
\end{equation*}
Then
\begin{equation} \label{cactus_quantum445}
\Phi(T_{\mathbf{j},w,\mathbf{i}}) = h_{w,w}T_{w} + \sum_{\substack{u \in S_{d} \\ L(u) < L(w)}} h_{u,w}T_{u}
\end{equation}
for some $ h_{u,w} \in \operatorname{OPol}_{d}(\operatorname{Cl}) $, where $ h_{w,w} $ is regular in $ \operatorname{OPol}_{d}(\operatorname{Cl}) $. Now assume that there is a linear dependence equation
\begin{equation} \label{depequation39756_quantum}
\sum_{w \in S_{d}} f_{w}T_{\mathbf{j},w,\mathbf{i}} = 0,
\end{equation}
where $ f_{w} \in \operatorname{OPol}_{d}(\operatorname{Cl}) $. Then, by using \cref{cactus_quantum445}, one can compute the image of \cref{depequation39756_quantum} under $ \Phi $ and then deduce $ f_{w} = 0 $ for all $ w \in S_{d} $.}

\details{We prove here that the set $ _{\mathbf{j}} \mathcal{C}_{\mathbf{i}} $ spans  $ \operatorname{Hom}_{\mathcal{T\!AS}(\operatorname{Cl})}(\mathbf{i}, \mathbf{j}) $. We prove this by induction on the number of crossings. If \(E\) is a diagram with no crossings, then \(E\) contains only Clifford tokens and dots, from which it clearly follows that \(E\) lies in the span of $ _{\mathbf{j}} \mathcal{C}_{\mathbf{i}} $.
\\ \indent Now let $ k > 0 $, and let \(E\) be a diagram with \(k\) crossings. First, by using the relations \cref{spin2}, \cref{spin2.5}, \cref{spin3}, \cref{spinCl1}, and \cref{spinCl2}, one can move the Clifford tokens and dots to the top of \(E\). Note that this process may create additional terms with fewer than \(k\) crossings, but these terms all lie in the span of $ _{\mathbf{j}} \mathcal{C}_{\mathbf{i}} $ by the induction hypothesis. If there is a pair of strands in \(E\) that cross more than once, then \(E\) can be written as a $ \mathbb{C} $-linear combination of diagrams with fewer than $ k $ crossings (the proof of this is analogous to the proof of \cite[Lem.~4.10(3)]{Webster}). So we may from here on assume that any two strands in \(E\) cross at most once. In this case, the sequence of black strands on \(E\) corresponds to some reduced expression for some element $ w \in S_{d} $. We can now apply the relations \cref{spin6}, \cref{spin8}, and the second relation in \cref{spin1} to get from \(E\) to $ T_{\mathbf{j},w,\mathbf{i}} $. Note again that this process may create additional terms with fewer crossings (due to \cref{spin8}), but these terms all lie in the span of $ _{\mathbf{j}} \mathcal{C}_{\mathbf{i}} $ by the induction hypothesis. Thus we have that \(E\) lies in the span of $ _{\mathbf{j}} \mathcal{C}_{\mathbf{i}} $.}

\subsection{Relation to the tensor product degenerate spin affine Hecke algebras}

\begin{prop} \label{greenleaves}
Let $ d \in \mathbb{N} $ and $ \mathbf{Q} \in F(\mathcal{E}_{\operatorname{OPol}_{1}}) $. Then we have superalgebra isomorphisms
\begin{align}
\operatorname{SH}^{\operatorname{aff}}_{d,\mathbf{Q}}(\operatorname{Cl}) \cong \operatorname{Cl}_{d} \otimes \operatorname{SH}^{\operatorname{aff}}_{d,\mathbf{Q}}, \label{isomorphism_1}
\\ \operatorname{SH}^{\operatorname{cyc}}_{d,\mathbf{Q}}(\operatorname{Cl}) \cong \operatorname{Cl}_{d} \otimes \operatorname{SH}^{\operatorname{cyc}}_{d,\mathbf{Q}}. \label{isomorphism_2}
\end{align}
\end{prop}

\begin{proof}
The functor $ \mathcal{T\!AS} \rightarrow \mathcal{T\!AS}(\operatorname{Cl}) $, which sends the generating objects and morphisms in $ \mathcal{T\!AS} $ to the objects and morphisms in $ \mathcal{T\!AS}(\operatorname{Cl}) $ of the same name, induces a superalgebra homomorphism $ \iota \colon \operatorname{SH}^{\operatorname{aff}}_{d,\mathbf{Q}} \rightarrow \operatorname{SH}^{\operatorname{aff}}_{d,\mathbf{Q}}(\operatorname{Cl}) $. Furthermore, we have a superalgebra homomorphism 
\begin{equation*}
\phi \colon \operatorname{Cl}_{d} \rightarrow \operatorname{SH}^{\operatorname{aff}}_{d,\mathbf{Q}}(\operatorname{Cl}), \quad c_{i} \mapsto \sum_{\mathbf{i} \in \Gamma_{d,\mathbf{Q}}} c_{i,\mathbf{i}}, \qquad 1 \leq i \leq d,
\end{equation*}
and so we obtain an induced $ \mathbb{C} $-linear map 
\begin{equation*}
\varphi \colon \operatorname{Cl}_{d} \otimes \operatorname{SH}^{\operatorname{aff}}_{d,\mathbf{Q}} \rightarrow \operatorname{SH}^{\operatorname{aff}}_{d,\mathbf{Q}}(\operatorname{Cl}), \quad y \otimes z \mapsto \phi(y)\iota(z).
\end{equation*}
It follows from Theorem \ref{HL-basis_Spin} and Proposition \ref{HL-basis_Spin-Cl} that $ \varphi $ is bijective. We next prove that $ \varphi $ is a superalgebra homomorphism. To do this, we must show that
\begin{equation} \label{commuteproof}
\phi(y)\iota(z) = (-1)^{\bar{y}\bar{z}}\iota(z)\phi(y)
\end{equation}
for all homogeneous $ y \in \operatorname{Cl}_{d}, \ z \in \operatorname{SH}^{\operatorname{aff}}_{d,\mathbf{Q}} $. We may assume that $ z $ is either an identity morphism, a dot, a red-black crossing, or a black-black crossing (since $ \operatorname{SH}^{\operatorname{aff}}_{d,\mathbf{Q}} $ is generated by these elements). Furthermore, we may assume that $ y = c_{i} $ for some $ 1 \leq i \leq d $. Then it is clear that \cref{commuteproof} holds when $ z $ is an identity morphism, and by \cref{spinCl123}, it also holds when \(z\) is a dot. Furthermore, it follows from \cref{spinCl1} that \cref{commuteproof} holds when \(z\) is a black-black crossing. Finally, we have by \cref{spinCl2} that \cref{commuteproof} holds when \(z\) is a red-black crossing. Thus we now obtain the isomorphism \cref{isomorphism_1}.

Set $ F $ to be the quotient of $ \operatorname{Cl}_{d} \otimes \operatorname{SH}^{\operatorname{aff}}_{d,\mathbf{Q}} $ by the two-sided ideal generated by the elements in the set $ \{1 \otimes 1_{\mathbf{i}} : \mathbf{i} \in \Gamma_{d,\mathbf{Q}}, \ \mathbf{i}_{1} = \go \} $. Then $ F $ is isomorphic to $ \operatorname{Cl}_{d} \otimes \operatorname{SH}^{\operatorname{cyc}}_{d,\mathbf{Q}} $. Furthermore, since $ \varphi(1 \otimes 1_{\mathbf{i}}) = 1_{\mathbf{i}} $ and $ \varphi^{-1}(1_{\mathbf{i}}) = 1 \otimes 1_{\mathbf{i}} $ for all $ \mathbf{i} \in \Gamma_{d,\mathbf{Q}} $, we have that $ \varphi $ and $ \varphi^{-1} $ induce an isomorphism $ \operatorname{SH}^{\operatorname{cyc}}_{d,\mathbf{Q}}(\operatorname{Cl}) \cong F $. We now obtain the isomorphism \cref{isomorphism_2}.
\end{proof}

\section{Tensor product degenerate affine Hecke--Clifford superalgebras $ H_{d,\mathbf{Q}}^{\operatorname{aff}}(\operatorname{Cl}) $} \label{HLDAHCA_Section}

In this section, we define the tensor product degenerate affine Hecke--Clifford superalgebras $ H_{d,\mathbf{Q}}^{\operatorname{aff}}(\operatorname{Cl}) $ and their cyclotomic quotients $ H_{d,\mathbf{Q}}^{\operatorname{cyc}}(\operatorname{Cl}) $. We start by providing a recap of the degenerate affine Hecke--Clifford superalgebras, which were first defined by Nazarov in \cite{Nazarov}.

\subsection{Clifford polynomial superalgebras}

\begin{defn} \label{Clifford_polynomial_category}
We define the \emph{Clifford polynomial supercategory} $ \mathpzc{Pol}(\operatorname{Cl}) $ to be the strict $ \mathbb{C} $-linear monoidal supercategory generated by one object $ \go $, and the morphisms
\begin{equation*}
\begin{tikzpicture}[anchorbase, thick]
        \draw[-] (0,0) -- (0,0.7);
        \bluetoken{0,0.35};    
\end{tikzpicture} \colon \go \rightarrow \go, 
\qquad 
\begin{tikzpicture}[anchorbase, thick]
        \draw[-] (0,0) -- (0,0.7);
        \singdot{0,0.35};    
\end{tikzpicture} \colon \go \rightarrow \go,
\end{equation*}
subject to the relations 
\begin{equation} \label{Pol(Cl)}
\begin{tikzpicture}[anchorbase, thick]
        \draw[-] (0,0) -- (0,0.7);
        \bluetoken{0,0.2};
        \bluetoken{0,0.45};
\end{tikzpicture}
=
\begin{tikzpicture}[anchorbase, thick]
        \draw[-] (0,0) -- (0,0.7);
\end{tikzpicture} \ ,
\qquad 
\begin{tikzpicture}[centerzero, thick]
        \draw[-] (0,-0.4) -- (0,0.4);
        \bluetoken{0,0.15};
        \singdot{0,-0.15};
\end{tikzpicture}
= -
\begin{tikzpicture}[centerzero, thick]
        \draw[-] (0,-0.4) -- (0,0.4);
        \bluetoken{0,-0.15};
        \singdot{0,0.15};
\end{tikzpicture} \ .
\end{equation}
The dot is even, and the morphism $ \begin{tikzpicture}[centerzero, thick]
        \draw[-] (0,-0.2) -- (0,0.2);
        \bluetoken{0,0};
\end{tikzpicture} $, which we refer to as the \emph{Clifford token}, is odd. Given $ n \in \mathbb{N} $, we define the \emph{Clifford polynomial superalgebra} to be 
\begin{equation*}
\operatorname{Pol}_{n}(\operatorname{Cl}) := \operatorname{End}_{\mathpzc{Pol}(\operatorname{Cl})}(\go^{\otimes n}).
\end{equation*}
\end{defn}

\begin{prop} \label{racing}
Let $ n \in \mathbb{N}_{+} $. Then the Clifford polynomial superalgebra $ \operatorname{Pol}_{n}(\operatorname{Cl}) $ is isomorphic to the free associative superalgebra generated by the elements $ c_{1},\ldots, c_{n} $, $ x_{1},\ldots,x_{n} $, subject to the relations 
\begin{align}
c_{i}^{2} &= 1, \quad c_{i}x_{i} = -x_{i}c_{i}, & 1 \leq i \leq n, \label{tophat1}
\\ c_{i}c_{j} &= -c_{j}c_{i}, \quad c_{i}x_{j} = x_{j}c_{i}, & 1 \leq i,j \leq n, \ i \neq j, \label{tophat3}
\\ x_{i}x_{j} &= x_{j}x_{i}, & 1 \leq i,j \leq n. \label{tophat5}
\end{align}
Here, the $ c_{i} $ are odd and the $ x_{i} $ are even. Under this isomorphism, $ c_{i} $ corresponds to the Clifford token on the \(i\)-th strand, and $ x_{i} $ corresponds to the dot on the \(i\)-th strand.
\end{prop}

\begin{proof}
This follows from Proposition \ref{zebra}.
\end{proof}

In what follows, we identify $ \operatorname{Pol}_{n}(\operatorname{Cl}) $ with the superalgebra presented in Proposition~\ref{racing}. The superalgebra $ \operatorname{Pol}_{n}(\operatorname{Cl}) $ has basis
\begin{equation*}
\{c_{1}^{k_{1}}\cdots c_{n}^{k_{n}}x_{1}^{r_{1}}\cdots x_{n}^{r_{n}} : k_{i} \in \{0,1\}, \ r_{i} \in \mathbb{N}\}.
\end{equation*}

\begin{lem} \label{Pol1(Cl)_even_center}
The even center of $ \operatorname{Pol}_{1}(\operatorname{Cl}) $ is equal to $ \mathbb{C}[x^{2}] $. Furthermore, we have $ \mathcal{E}_{\operatorname{Pol}_{1}(\operatorname{Cl})} = \mathbb{C}[x^{2}] \setminus 0 $.
\end{lem}

\begin{proof}
The first statement is straightforward to show. The second statement then follows from the first statement and the definition of $ \mathcal{E}_{\operatorname{Pol}_{1}(\operatorname{Cl})} $ (see Definition~\ref{opportunity}).
\end{proof}

\details{We show the first statement here. It is clear that any polynomial in $ x^{2} $ with coefficients in $ \mathbb{C} $ will commute with $ c $ and $ x $, and so each such polynomial lies in the even center of $ \operatorname{Pol}_{1}(\operatorname{Cl}) $. Conversely, let \(f\) be an element in the even center of $ \operatorname{Pol}_{1}(\operatorname{Cl}) $. Then we may write that $ f = \sum_{k \geq 0} \lambda_{k}x^{k} $ for some $ \lambda_{k} \in \mathbb{C} $. We compute that
\begin{equation*}
0 = cf - fc = \sum_{k \geq 0} (1-(-1)^{k})\lambda_{k}cx^{k},
\end{equation*}
from which we obtain that $ \lambda_{k} = 0 $ for all odd $ k \geq 0 $.}

We have an action of the symmetric group $ S_{n} $ on $ \operatorname{Pol}_{n}(\operatorname{Cl}) $, given by $ w(c_{i}) = c_{w(i)} $ and $ w(x_{i}) = x_{w(i)} $ for $ w \in S_{n} $, $ 1 \leq i \leq n $. We denote the action of \(w\) on $ f \in \operatorname{Pol}_{n}(\operatorname{Cl}) $ by $ w(f) $. We next define the Clifford Demazure operators. These are special cases of Savage's deformed divided difference operators \cite[§4.1]{Savage}.

\begin{defn} \label{Cliff_Demazure}
We define the \emph{Clifford Demazure operators} to be the $ \mathbb{C} $-linear operators $ \partial_{i} \colon \operatorname{Pol}_{n}(\operatorname{Cl}) \rightarrow \operatorname{Pol}_{n}(\operatorname{Cl}) $, $ 1 \leq i \leq n-1 $, determined by the condition
\begin{equation} \label{Cliff_leibniz33}
\partial_{i}(\mathbf{a}) = 0, \quad \partial_{i}(x_{j}) = 
\begin{cases}
1+c_{i}c_{i+1} &\text{if} \ j = i,
\\ -1+c_{i}c_{i+1}  &\text{if} \ j = i+1,
\\ 0 &\text{otherwise},
\end{cases}
\end{equation}
for $ \mathbf{a} \in \operatorname{Cl}_{n} $, and the twisted Leibniz rule
\begin{equation} \label{Cliff_leibniz}
\partial_{i}(fg) = \partial_{i}(f)g + s_{i}(f)\partial_{i}(g), \quad f,g \in \operatorname{Pol}_{n}(\operatorname{Cl}).
\end{equation}
We set $ \partial = \partial_{1} \colon \operatorname{Pol}_{2}(\operatorname{Cl}) \rightarrow \operatorname{Pol}_{2}(\operatorname{Cl}) $.
\end{defn}

\subsection{Degenerate affine Hecke--Clifford algebras} 

\begin{defn} \label{DSAHCSdef}
We define the \emph{degenerate affine Hecke--Clifford supercategory} (or \emph{degenerate affine Sergeev supercategory}) $ \mathcal{AH}(\operatorname{Cl}) $ to be the strict $ \mathbb{C} $-linear monoidal supercategory generated by one object $ \go $, and the morphisms
\begin{gather}
    \begin{tikzpicture}[centerzero, thick]
        \draw[-] (-0.2,-0.2) -- (0.2,0.2);
        \draw[-] (0.2,-0.2) -- (-0.2,0.2);
    \end{tikzpicture}
      \ \colon
    \go \otimes \go \to \go \otimes \go,
    \quad
    \begin{tikzpicture}[anchorbase, thick]
        \draw[-] (0,-0.3) -- (0,0.3);
        \bluetoken{0,0};
    \end{tikzpicture}
    \colon \go \to \go, \quad
 \begin{tikzpicture}[anchorbase, thick]
        \draw[-] (0,0) -- (0,0.6);
        \singdot{0,0.3};
\end{tikzpicture}
\ \colon
\go \to \go.
\end{gather}
The crossing and dot are even, and the Clifford token is odd. The relations on the morphisms are given by \cref{Pol(Cl)} and
\begin{gather}
\label{Cliff_hw-1}
\begin{tikzpicture}[anchorbase, thick]
        \draw[-] (0.2,-0.5) to[out=up,in=down] (-0.2,0) to[out=up,in=down] (0.2,0.5);
        \draw[-] (-0.2,-0.5) to[out=up,in=down] (0.2,0) to[out=up,in=down] (-0.2,0.5);
\end{tikzpicture}
\ =\
\begin{tikzpicture}[anchorbase, thick]
        \draw[-] (-0.2,-0.5) -- (-0.2,0.5);
        \draw[-] (0.2,-0.5) -- (0.2,0.5);
\end{tikzpicture}
\ ,\quad
\begin{tikzpicture}[anchorbase, thick]
        \draw[-] (0.4,-0.5) -- (-0.4,0.5);
        \draw[-] (0,-0.5) to[out=up, in=down] (-0.4,0) to[out=up,in=down] (0,0.5);
        \draw[-] (-0.4,-0.5) -- (0.4,0.5);
\end{tikzpicture}
\ =\
\begin{tikzpicture}[anchorbase, thick]
        \draw[-] (0.4,-0.5) -- (-0.4,0.5);
        \draw[-] (0,-0.5) to[out=up, in=down] (0.4,0) to[out=up,in=down] (0,0.5);
        \draw[-] (-0.4,-0.5) -- (0.4,0.5);
\end{tikzpicture}
\ ,\quad
\begin{tikzpicture}[centerzero, thick]
        \draw[-] (0.3,-0.4) -- (-0.3,0.4);
        \draw[-] (-0.3,-0.4) -- (0.3,0.4);
        \bluetoken{-0.15,-0.2};
\end{tikzpicture}
\ =\
\begin{tikzpicture}[centerzero, thick]
        \draw[-] (0.3,-0.4) -- (-0.3,0.4);
        \draw[-] (-0.3,-0.4) -- (0.3,0.4);
        \bluetoken{0.15,0.2};
\end{tikzpicture} \ ,
\\ \label{Cliff_hw0}
\begin{tikzpicture}[centerzero, thick]
        \draw[-] (0.3,-0.4) -- (-0.3,0.4);
        \draw[-] (-0.3,-0.4) -- (0.3,0.4);
        \singdot{-0.15,-0.2};
\end{tikzpicture}
\ =\
\begin{tikzpicture}[centerzero, thick]
        \draw[-] (-0.3,-0.4) -- (0.3,0.4);
        \draw[-] (0.3,-0.4) -- (-0.3,0.4);
        \singdot{0.171,0.228};
\end{tikzpicture}
\ - \ 
\begin{tikzpicture}[centerzero]
        \draw (-0.2,-0.4) -- (-0.2,0.4);
        \draw (0.2,-0.4) -- (0.2,0.4);
\end{tikzpicture} 
-
\begin{tikzpicture}[centerzero]
        \draw (-0.2,-0.4) -- (-0.2,0.4);
        \draw (0.2,-0.4) -- (0.2,0.4);
        \bluetoken{-0.2,0.1};
        \bluetoken{0.2,-0.1};
\end{tikzpicture} \ .
\end{gather}
For $ n \in \mathbb{N} $, we define the \emph{degenerate affine Hecke--Clifford superalgebra} (or \emph{degenerate affine Sergeev superalgebra}) to be 
\begin{equation*}
H_{n}^{\operatorname{aff}}(\operatorname{Cl}) := \operatorname{End}_{\mathcal{AH}(\operatorname{Cl})}(\go^{\otimes n}).
\end{equation*}
\end{defn}

It follows from \cref{Cliff_hw-1} and \cref{Cliff_hw0} that
\begin{equation} \label{Cliff_hw-5}
\begin{tikzpicture}[centerzero, thick]
        \draw[-] (0.3,-0.4) -- (-0.3,0.4);
        \draw[-] (-0.3,-0.4) -- (0.3,0.4);
        \bluetoken{0.15,-0.2};
\end{tikzpicture}
\ =\
\begin{tikzpicture}[centerzero, thick]
        \draw[-] (0.3,-0.4) -- (-0.3,0.4);
        \draw[-] (-0.3,-0.4) -- (0.3,0.4);
        \bluetoken{-0.15,0.2};
\end{tikzpicture} \ ,
\quad 
\begin{tikzpicture}[centerzero, thick]
        \draw[-] (0.3,-0.4) -- (-0.3,0.4);
        \draw[-] (-0.3,-0.4) -- (0.3,0.4);
        \singdot{-0.15,0.2};
\end{tikzpicture}
\ =\
\begin{tikzpicture}[centerzero, thick]
        \draw[-] (-0.3,-0.4) -- (0.3,0.4);
        \draw[-] (0.3,-0.4) -- (-0.3,0.4);
        \singdot{0.171,-0.228};
\end{tikzpicture}
\ - \ 
\begin{tikzpicture}[centerzero]
        \draw (-0.2,-0.4) -- (-0.2,0.4);
        \draw (0.2,-0.4) -- (0.2,0.4);
\end{tikzpicture} 
\ + \
\begin{tikzpicture}[centerzero]
        \draw (-0.2,-0.4) -- (-0.2,0.4);
        \draw (0.2,-0.4) -- (0.2,0.4);
        \bluetoken{-0.2,0.1};
        \bluetoken{0.2,-0.1};
\end{tikzpicture}
\end{equation}

%

\subsection{Tensor product degenerate affine Hecke--Clifford superalgebras}

\begin{defn} \label{HLDAHCC}
We define the \emph{tensor product degenerate affine Hecke--Clifford supercategory} (or the \emph{tensor product degenerate affine Sergeev supercategory}) $ \mathcal{T\!AH}(\operatorname{Cl}) $ to be the strict $ \mathbb{C} $-linear monoidal supercategory defined as follows. The objects are generated by the elements in the set $ \widehat{\mathcal{E}}_{\operatorname{Pol}_{1}(\operatorname{Cl})} $. The morphisms are generated by
\begin{gather}
    \begin{tikzpicture}[centerzero, thick]
        \draw[-] (-0.2,-0.2) -- (0.2,0.2);
        \draw[-] (0.2,-0.2) -- (-0.2,0.2);
    \end{tikzpicture}
      \ \colon
    \go \otimes \go \to \go \otimes \go,
    \quad
    \begin{tikzpicture}[anchorbase, thick]
        \draw[-] (0,-0.3) -- (0,0.3);
        \bluetoken{0,0};
    \end{tikzpicture}
    \colon \go \to \go, \quad
 \begin{tikzpicture}[anchorbase, thick]
        \draw[-] (0,0) -- (0,0.6);
        \singdot{0,0.3};
\end{tikzpicture}
\ \colon
\go \to \go, \\
\begin{tikzpicture}[centerzero, thick]
        \draw[-] (0.2,-0.2) -- (-0.2,0.2);
        \draw[-, wei] (-0.2,-0.2) -- (0.2,0.2);
        \node at (-0.2,-0.4) {\tiny $ Q $};
\end{tikzpicture}
\ \colon Q \otimes \go \to \go \otimes Q, \qquad
\begin{tikzpicture}[centerzero, thick]
        \draw[-] (-0.2,-0.2) -- (0.2,0.2);
        \draw[-,wei] (0.2,-0.2) -- (-0.2,0.2);
        \node at (0.2,-0.4) {\tiny $ Q $};
\end{tikzpicture}
\colon \go \otimes Q \to Q \otimes \go, 
\end{gather}
for all $ Q \in \mathcal{E}_{\operatorname{Pol}_{1}(\operatorname{Cl})} $. The crossings are even, the dot is even, and the Clifford token is odd. The relations on the morphisms are given by \cref{Pol(Cl)}, \cref{Cliff_hw-1}, \cref{Cliff_hw0}, and
\begin{gather}
\label{Cliff_hw1}
\begin{tikzpicture}[centerzero, thick]
        \draw[-] (-0.3,-0.4) -- (0.3,0.4);
        \draw[-, wei] (0.3,-0.4) -- (-0.3,0.4);
        \bluetoken{-0.15,-0.2};
        \node at (0.3,-0.6) {\tiny $ Q $};
\end{tikzpicture}
= 
\begin{tikzpicture}[centerzero, thick]
        \draw[-] (-0.3,-0.4) -- (0.3,0.4);
        \draw[-, wei] (0.3,-0.4) -- (-0.3,0.4);
        \bluetoken{0.15,0.2};
        \node at (0.3,-0.6) {\tiny $ Q $};
\end{tikzpicture}
\ , \qquad
\begin{tikzpicture}[centerzero, thick]
        \draw[-] (0.3,-0.4) -- (-0.3,0.4);
        \draw[-, wei] (-0.3,-0.4) -- (0.3,0.4);
        \bluetoken{0.15,-0.2};
        \node at (-0.3,-0.6) {\tiny $ Q $};
\end{tikzpicture}
= 
\begin{tikzpicture}[centerzero, thick]
        \draw[-] (0.3,-0.4) -- (-0.3,0.4);
        \draw[-, wei] (-0.3,-0.4) -- (0.3,0.4);
        \bluetoken{-0.15,0.2};
        \node at (-0.3,-0.6) {\tiny $ Q $};
\end{tikzpicture} \ , 
\\ \label{Cliff_hw2}
\begin{tikzpicture}[centerzero, thick]
        \draw[-] (-0.3,-0.4) -- (0.3,0.4);
        \draw[-, wei] (0.3,-0.4) -- (-0.3,0.4);
        \singdot{-0.15,-0.2};
        \node at (0.3,-0.6) {\tiny $ Q $};
\end{tikzpicture}
= 
\begin{tikzpicture}[centerzero, thick]
        \draw[-] (-0.3,-0.4) -- (0.3,0.4);
        \draw[-, wei] (0.3,-0.4) -- (-0.3,0.4);
        \singdot{0.15,0.2};
        \node at (0.3,-0.6) {\tiny $ Q $};
\end{tikzpicture}
\ , \qquad
\begin{tikzpicture}[centerzero, thick]
        \draw[-] (0.3,-0.4) -- (-0.3,0.4);
        \draw[-, wei] (-0.3,-0.4) -- (0.3,0.4);
        \singdot{0.15,-0.2};
        \node at (-0.3,-0.6) {\tiny $ Q $};
\end{tikzpicture}
= 
\begin{tikzpicture}[centerzero, thick]
        \draw[-] (0.3,-0.4) -- (-0.3,0.4);
        \draw[-, wei] (-0.3,-0.4) -- (0.3,0.4);
        \singdot{-0.15,0.2};
        \node at (-0.3,-0.6) {\tiny $ Q $};
\end{tikzpicture} \ , 
\\ \label{Cliff_hw3}
\begin{tikzpicture}[centerzero, thick]
        \draw[-] (-0.2,-0.5) to[out=up,in=down] (0.2,0) to[out=up,in=down] (-0.2,0.5);
        \draw[-, wei] (0.2,-0.5) to[out=up,in=down] (-0.2,0) to[out=up,in=down] (0.2,0.5);
        \node at (0.2,-0.7) {\tiny $ Q $};
\end{tikzpicture}
=
\begin{tikzpicture}[centerzero, thick]
        \pin{-0.2,0}{-1,0}{Q};
        \draw[-] (-0.2,-0.5) -- (-0.2,0.5);
        \draw[-, wei] (0.2,-0.5) -- (0.2,0.5);
        \node at (0.2,-0.7) {\tiny $ Q $};
\end{tikzpicture} \ , 
\qquad
\begin{tikzpicture}[centerzero, thick]
        \draw[-] (0.2,-0.5) to[out=up,in=down] (-0.2,0) to[out=up,in=down] (0.2,0.5);
        \draw[-, wei] (-0.2,-0.5) to[out=up,in=down] (0.2,0) to[out=up,in=down] (-0.2,0.5);
        \node at (-0.2,-0.7) {\tiny $ Q $};
\end{tikzpicture}
\ = \
\begin{tikzpicture}[centerzero, thick]
        \draw[-, wei] (-0.2,-0.5) -- (-0.2,0.5);
        \draw[-] (0.2,-0.5) -- (0.2,0.5);
        \node at (-0.2,-0.7) {\tiny $ Q $};
        \pin{0.2,0}{1,0}{Q};
\end{tikzpicture} \ , 
\\ \label{Cliff_hw5}
\begin{tikzpicture}[centerzero, thick]
        \draw[-] (0.4,-0.5) -- (-0.4,0.5);
        \draw[-] (0,-0.5) to[out=up, in=down] (-0.4,0) to[out=up,in=down] (0,0.5);
        \draw[-, wei] (-0.4,-0.5) -- (0.4,0.5);
        \node at (-0.4,-0.7) {\tiny $ Q $};
\end{tikzpicture}
\ =\
\begin{tikzpicture}[centerzero, thick]
        \draw[-] (0.4,-0.5) -- (-0.4,0.5);
        \draw[-] (0,-0.5) to[out=up, in=down] (0.4,0) to[out=up,in=down] (0,0.5);
        \draw[-, wei] (-0.4,-0.5) -- (0.4,0.5);
        \node at (-0.4,-0.7) {\tiny $ Q $};
\end{tikzpicture} \ , 
\qquad
\begin{tikzpicture}[centerzero, thick]
        \draw[-] (0,-0.5) to[out=up, in=down] (-0.4,0) to[out=up,in=down] (0,0.5);
        \draw[-] (-0.4,-0.5) -- (0.4,0.5);
        \draw[-, wei] (0.4,-0.5) -- (-0.4,0.5);
        \node at (0.4,-0.7) {\tiny $ Q $};
\end{tikzpicture}
\ =\
\begin{tikzpicture}[centerzero, thick]
        \draw[wipe] (0,-0.5) to[out=up, in=down] (0.4,0) to[out=up,in=down] (0,0.5);
        \draw[-] (0,-0.5) to[out=up, in=down] (0.4,0) to[out=up,in=down] (0,0.5);
        \draw[-] (-0.4,-0.5) -- (0.4,0.5);
        \draw[-, wei] (0.4,-0.5) -- (-0.4,0.5);
        \node at (0.4,-0.7) {\tiny $ Q $};
\end{tikzpicture}
\ , 
\\ \label{Cliff_green}
\begin{tikzpicture}[centerzero, thick]
        \draw[-] (0.4,-0.5) -- (-0.4,0.5);
        \draw[-] (-0.4,-0.5) -- (0.4,0.5);
        \node at (0,-0.7) {\tiny $ Q $};
        \draw[-, wei] (0,-0.5) to[out=up, in=down] (-0.4,0) to[out=up,in=down] (0,0.5);
\end{tikzpicture}
\ =\
\begin{tikzpicture}[centerzero, thick]
        \draw[-] (2,-0.5) -- (1.2,0.5);
        \draw[-] (1.2,-0.5) -- (2,0.5);
        \node at (1.6,-0.7) {\tiny $ Q $};
        \draw[-, wei] (1.6,-0.5) to[out=up, in=down] (2,0) to[out=up,in=down] (1.6,0.5);
\end{tikzpicture}
\ - \
\begin{tikzpicture}[centerzero, thick]
        \node at (0,-0.7) {\tiny $ Q $};
        \draw[-] (-0.6,-0.5) -- (-0.6,0.5);
        \draw[-, wei] (0,-0.5) -- (0,0.5);
        \draw[-] (0.6,-0.5) -- (0.6,0.5);
        \pinpin{0.6,0}{-0.6,0}{1.8,0}{\partial(Q_{1})};
\end{tikzpicture} \ ,
\end{gather}
for all $ Q \in \mathcal{E}_{\operatorname{Pol}_{1}(\operatorname{Cl})} $. This concludes the definition of $ \mathcal{T\!AH}(\operatorname{Cl}) $.
\end{defn}

Recall again that, for a given set \(X\), $ F(X) $ is the free monoid on \(X\). Then the set of objects in $ \mathcal{T\!AH}(\operatorname{Cl}) $ is $ F(\widehat{\mathcal{E}}_{\operatorname{Pol}_{1}(\operatorname{Cl})}) $. 

\begin{defn} 
Given an integer $ d \in \mathbb{N} $ and a word $ \mathbf{Q} \in F(\mathcal{E}_{\operatorname{Pol}_{1}(\operatorname{Cl})}) $, we define $ \Lambda_{d,\mathbf{Q}} $ to be the set of $ (\go^{d},\mathbf{Q}) $-shuffles in $ F(\widehat{\mathcal{E}}_{\operatorname{Pol}_{1}(\operatorname{Cl})}) $. In other words, 
\begin{equation} \label{Lambda_def}
\Lambda_{d,\mathbf{Q}} = F(\widehat{\mathcal{E}}_{\operatorname{Pol}_{1}(\operatorname{Cl})})_{\go^{d},\mathbf{Q}},
\end{equation}
where $ F(\widehat{\mathcal{E}}_{\operatorname{Pol}_{1}(\operatorname{Cl})})_{\go^{d},\mathbf{Q}} $ is given as in Definition~\ref{snowing}.
\end{defn}

\begin{defn} 
Let $ d \in \mathbb{N} $ and $ \mathbf{Q} \in F(\mathcal{E}_{\operatorname{Pol}_{1}(\operatorname{Cl})}) $. Then we define the $ (d,\mathbf{Q}) $-\emph{degenerate affine Hecke--Clifford superalgebra} (or the $ (d,\mathbf{Q}) $-\emph{degenerate affine Sergeev superalgebra}) to be 
\begin{equation} \label{higher_Clifford_Hecke_def}
H_{d,\mathbf{Q}}^{\operatorname{aff}}(\operatorname{Cl}) := \operatorname{End}_{\operatorname{Add}(\mathcal{T\!AH}(\operatorname{Cl}))}\left(\bigoplus_{\mathbf{i} \in \Lambda_{d,\mathbf{Q}}} \mathbf{i}\right).
\end{equation}
Here, $ \operatorname{Add}(\mathcal{T\!AH}(\operatorname{Cl})) $ is the additive envelope of $ \mathcal{T\!AH}(\operatorname{Cl}) $. 
\end{defn}

We will often refer to the superalgebra $ H_{d,\mathbf{Q}}^{\operatorname{aff}}(\operatorname{Cl}) $ as a \emph{tensor product degenerate affine Hecke--Clifford superalgebra}. This superalgebra is a special case of the higher-level affine wreath product algebra introduced in \cite[Def.~5.5]{Moran} taking the Frobenius superalgebra, denoted \(A\) there, to be the Clifford superalgebra $ \operatorname{Cl} $ with trace map $ \operatorname{tr} \colon \operatorname{Cl} \rightarrow \mathbb{C} $, $ 1 \mapsto 1 $, $ c \mapsto 0 $. A basis of $ H_{d,\mathbf{Q}}^{\operatorname{aff}}(\operatorname{Cl}) $ is given via \cite[Thm.~5.18]{Moran}, but we will not state this result here, since it will not be needed in the current paper. 

\subsection{Cyclotomic quotients}

We next define cyclotomic quotients of the tensor product degenerate affine Hecke--Clifford superalgebra.

\begin{defn} \label{spin_cyclotomic_quotient_Clifford}
Let $ d \in \mathbb{N}_{+} $ and $ \mathbf{Q} \in F(\mathcal{E}_{\operatorname{Pol}_{1}(\operatorname{Cl})}) $. Then we define the \emph{cyclotomic $ (d,\mathbf{Q}) $-Hecke--Clifford superalgebra} $ H_{d,\mathbf{Q}}^{\operatorname{cyc}}(\operatorname{Cl}) $ to be the quotient of $ H_{d,\mathbf{Q}}^{\operatorname{aff}}(\operatorname{Cl}) $ by the two-sided ideal generated by the elements in the set $ \{1_{\mathbf{i}} : \mathbf{i} \in \Lambda_{d,\mathbf{Q}}, \ \mathbf{i}_{1} = \go \} $.
\end{defn}

We can also define cyclotomic quotients of the degenerate affine Hecke--Clifford superalgebras $ H_{n}^{\operatorname{aff}}(\operatorname{Cl}) $.

\begin{defn}
Let $ n \in \mathbb{N}_{+} $ and $ Q \in \mathcal{E}_{\operatorname{Pol}_{1}(\operatorname{Cl})} $. Then we define the \emph{cyclotomic Hecke--Clifford superalgebra} $ H_{n}^{Q}(\operatorname{Cl}) $ to be the quotient of $ H_{n}^{\operatorname{aff}}(\operatorname{Cl}) $ by the two-sided ideal generated by the element $ Q_{1} \in \operatorname{Pol}_{n}(\operatorname{Cl}) $.
\end{defn}

The following proposition says that, when $ \mathbf{Q} \in F(\mathcal{E}_{\operatorname{Pol}_{1}(\operatorname{Cl})}) $ is a word of length one, the cyclotomic $ (d,\mathbf{Q}) $-Hecke--Clifford superalgebra is isomorphic to a cyclotomic Hecke--Clifford superalgebra. It is an analogue of Proposition~\ref{frontlines}.

\begin{prop} \label{Cliff_cyclotomic_result}
Let $ d \in \mathbb{N}_{+} $ and $ Q \in \mathcal{E}_{\operatorname{Pol}_{1}(\operatorname{Cl})} $. Then we have an isomorphism of superalgebras
\begin{equation*}
H_{d,Q}^{\operatorname{cyc}}(\operatorname{Cl}) \cong H_{d}^{Q}(\operatorname{Cl}).
\end{equation*}
\end{prop}

\begin{proof}
This is a special case of \cite[Thm.~5.34]{Moran}.
\end{proof}

\section{The isomorphisms} \label{Morita_superequivalence123}

In this section, we prove the main result of this paper. We show that tensor product degenerate spin affine Hecke superalgebras are Morita superequivalent to tensor product degenerate affine Hecke--Clifford superalgebras. This result is given in Theorem~\ref{Morita_result}.

\subsection{Isomorphism of $ \mathpzc{OPol}(\operatorname{Cl}) $ and $ \mathpzc{Pol}(\operatorname{Cl}) $} \label{Isomorphism1}

Recall that we have the supercategories $ \mathpzc{OPol}(\operatorname{Cl}) $ and $ \mathpzc{Pol}(\operatorname{Cl}) $ from Definition \ref{Odd_category_Cl} and Definition \ref{Clifford_polynomial_category}.

\begin{prop} \label{Next}
We have an isomorphism of monoidal supercategories $ \Psi \colon \mathpzc{OPol}(\operatorname{Cl}) \rightarrow \mathpzc{Pol}(\operatorname{Cl}) $ that sends $ \go $ to $ \go $, and sends the generating morphisms as follows:
\begin{equation} \label{joker_1}
\Psi\left( \begin{tikzpicture}[anchorbase, thick]
        \draw[-] (0,0) -- (0,0.7);
        \xtoken{0,0.35};
\end{tikzpicture}\right)
=
\begin{tikzpicture}[anchorbase, thick]
        \draw[-] (0,0) -- (0,0.7);
         \bluetoken{0,0.35};   
\end{tikzpicture} \ , 
\qquad 
\Psi\left( \begin{tikzpicture}[anchorbase, thick]
        \draw[-] (0,0) -- (0,0.7);
        \singdot{0,0.35};
\end{tikzpicture}\right)
= \frac{1}{\sqrt{-2}} \
\begin{tikzpicture}[anchorbase, thick]
        \draw[-] (0,0) -- (0,0.7);
        \singdot{0,0.2};
        \bluetoken{0,0.5}; 
\end{tikzpicture} \ .
\end{equation}
Its inverse $ \Psi^{-1} \colon \mathpzc{Pol}(\operatorname{Cl}) \rightarrow \mathpzc{OPol}(\operatorname{Cl}) $ sends $ \go $ to $ \go $, and sends the generating morphisms as follows:
\begin{equation} \label{joker_2}
\Psi^{-1}\left( \begin{tikzpicture}[anchorbase, thick]
        \draw[-] (0,0) -- (0,0.7);
        \bluetoken{0,0.35}; 
\end{tikzpicture}\right)
=
\begin{tikzpicture}[anchorbase, thick]
        \draw[-] (0,0) -- (0,0.7);
        \xtoken{0,0.35};
\end{tikzpicture} \ , 
\qquad 
\Psi^{-1}\left(\begin{tikzpicture}[anchorbase, thick]
        \draw[-] (0,0) -- (0,0.7);
        \singdot{0,0.35};
\end{tikzpicture}\right)
= \sqrt{-2} \
\begin{tikzpicture}[anchorbase, thick]
        \draw[-] (0,0) -- (0,0.7);
        \singdot{0,0.2};
        \xtoken{0,0.5};
\end{tikzpicture} \ .
\end{equation}
\end{prop}

\begin{proof}
It is straightforward to check that the given maps respect the defining relations and are mutually inverse. We omit the details here.
\end{proof}

\details{For the existence of $ \Psi $, it suffices just to check the two relations in \cref{spinCl123}. The first relation in \cref{spinCl123} follows immediately from \cref{Pol(Cl)}, and for the second relation, we compute that
\begin{equation*}
\Psi \left(\begin{tikzpicture}[centerzero, thick]
        \draw[-] (0,-0.4) -- (0,0.4);
        \xtoken{0,0.15};
        \singdot{0,-0.15};
\end{tikzpicture}\right)
= \frac{1}{\sqrt{-2}} 
\ 
\begin{tikzpicture}[centerzero, thick]
        \draw[-] (0,-0.4) -- (0,0.4);
        \bluetoken{0,0.2};
        \bluetoken{0,0};
        \singdot{0,-0.2};
\end{tikzpicture}
\ \stackrel{\mathclap{\cref{Pol(Cl)}}}{=} \ -\frac{1}{\sqrt{-2}} 
\ 
\begin{tikzpicture}[centerzero, thick]
        \draw[-] (0,-0.4) -- (0,0.4);
        \bluetoken{0,0.2};
        \bluetoken{0,-0.2};
        \singdot{0,0};
\end{tikzpicture}
=
\Psi \left(- \ \begin{tikzpicture}[centerzero, thick]
        \draw[-] (0,-0.4) -- (0,0.4);
        \xtoken{0,-0.15};
        \singdot{0,0.15};
\end{tikzpicture}\right).
\end{equation*}
For the existence of $ \Psi^{-1} $, it suffices just to check the two relations in \cref{Pol(Cl)}. The first relation in \cref{Pol(Cl)} follows immediately from \cref{spinCl123}, and for the second relation, we compute that 
\begin{equation*}
\Psi^{-1} \left(\begin{tikzpicture}[centerzero, thick]
        \draw[-] (0,-0.4) -- (0,0.4);
        \bluetoken{0,0.15};
        \singdot{0,-0.15};
\end{tikzpicture}\right)
= \sqrt{-2}
\ 
\begin{tikzpicture}[centerzero, thick]
        \draw[-] (0,-0.4) -- (0,0.4);
        \xtoken{0,0.2};
        \xtoken{0,0};
        \singdot{0,-0.2};
\end{tikzpicture}
\ \stackrel{\mathclap{\cref{spinCl123}}}{=} \ -\sqrt{-2}
\ 
\begin{tikzpicture}[centerzero, thick]
        \draw[-] (0,-0.4) -- (0,0.4);
        \xtoken{0,0.2};
        \xtoken{0,-0.2};
        \singdot{0,0};
\end{tikzpicture}
=
\Psi^{-1} \left(- \ \begin{tikzpicture}[centerzero, thick]
        \draw[-] (0,-0.4) -- (0,0.4);
        \bluetoken{0,-0.15};
        \singdot{0,0.15};
\end{tikzpicture}\right).
\end{equation*}
Finally, it is straightforward to check that $ \Psi $ and $ \Psi^{-1} $ are indeed mutual inverses.}

The functor $ \Psi $ induces a superalgebra isomorphism $ \psi_{n} \colon \operatorname{OPol}_{n}(\operatorname{Cl}) \rightarrow \operatorname{Pol}_{n}(\operatorname{Cl}) $ for each $ n \in \mathbb{N}_{+} $. We set $ \psi = \psi_{1} $.

\begin{lem} \label{Psi(H)}
We have that
\begin{equation}
\psi(\mathcal{E}_{\operatorname{OPol}_{1}}) = \mathcal{E}_{\operatorname{Pol}_{1}(\operatorname{Cl})}.
\end{equation}
\end{lem}

\begin{proof}
Since $ \psi \colon \operatorname{OPol}_{1}(\operatorname{Cl}) \rightarrow \operatorname{Pol}_{1}(\operatorname{Cl}) $ is a superalgebra isomorphism, we immediately obtain that $ \psi(\mathcal{E}_{\operatorname{OPol}_{1}(\operatorname{Cl})}) = \mathcal{E}_{\operatorname{Pol}_{1}(\operatorname{Cl})} $. The result now follows by Lemma \ref{center_Odd(Cl)}.
\end{proof}

\begin{lem}
Let $ n \in \mathbb{N}_{+} $ and $ 1 \leq i \leq n $. Then we have in $ \operatorname{Pol}_{n}(\operatorname{Cl}) $ that
\begin{equation} \label{fan_switch}
\psi_{n}(f_{i}) = \psi(f)_{i} \qquad f \in \operatorname{OPol}_{1}(\operatorname{Cl}).
\end{equation}
Furthermore, we have in $ \operatorname{OPol}_{n}(\operatorname{Cl}) $ that
\begin{equation} \label{fan_switch_2}
\psi_{n}^{-1}(f_{i}) = \psi^{-1}(f)_{i} \qquad f \in \operatorname{Pol}_{1}(\operatorname{Cl}).
\end{equation}
\end{lem}

\begin{proof}
These are easy to see, and so we omit the proofs here.
\end{proof}

\details{It suffices to check that \cref{fan_switch} holds for $ f = c $ and $ f = x $. If $ f = c $, then both sides of \cref{fan_switch} are equal to $ c_{i} $. If $ f = x $, then we have
\begin{equation*}
\psi_{n}(f_{i}) = \psi_{n}(x_{i}) = \frac{1}{\sqrt{-2}}c_{i}x_{i} = \left( \frac{1}{\sqrt{-2}}cx \right)_{i} = \psi(x)_{i}  = \psi(f)_{i}.
\end{equation*}
Next, it suffices to check that \cref{fan_switch_2} holds for $ f = c $ and $ f = x $. If $ f = c $, then both sides of \cref{fan_switch_2} are equal to $ c_{i} $. If $ f = x $, then we have
\begin{equation*}
\psi_{n}^{-1}(f_{i}) = \psi_{n}^{-1}(x_{i}) = \sqrt{-2}c_{i}x_{i} = \left( \sqrt{-2}cx \right)_{i} = \psi^{-1}(x)_{i}  = \psi^{-1}(f)_{i}.
\end{equation*}}

\begin{lem}
Let $ n \in \mathbb{N}_{+} $ and $ 1 \leq i \leq n-1 $. Then 
\begin{equation} \label{frown}
\psi_{n}(f)(c_{i}-c_{i+1}) = (-1)^{\bar{f}}(c_{i}-c_{i+1})\psi_{n}(f), \qquad f \in \operatorname{OPol}_{n}.
\end{equation}
\end{lem}

\begin{proof}
It suffices to check that \cref{frown} holds for $ f = x_{j} $, $ 1 \leq j \leq n $. These computations are straightforward. For example, if $ j = i $, then
\begin{equation*}
\psi_{n}(f)(c_{i}-c_{i+1}) = \frac{1}{\sqrt{-2}}c_{i}x_{i}(c_{i}-c_{i+1}) \ \overset{\mathclap{\cref{tophat1}}}{\underset{\mathclap{\cref{tophat3}}}{=}} \ -\frac{1}{\sqrt{-2}}(c_{i}-c_{i+1})c_{i}x_{i} = -(c_{i}-c_{i+1})\psi_{n}(f). \qedhere
\end{equation*}
\end{proof}

\details{If $ j \neq i,i+1 $, then we have 
\begin{equation*}
\psi_{n}(f)(c_{i}-c_{i+1}) = \frac{1}{\sqrt{-2}}c_{j}x_{j}(c_{i}-c_{i+1}) \ \stackrel{\mathclap{\cref{tophat3}}}{=} \ -\frac{1}{\sqrt{-2}}(c_{i}-c_{i+1})c_{j}x_{j} = -(c_{i}-c_{i+1})\psi_{n}(f).
\end{equation*}
Next, if $ j = i $, then
\begin{equation*}
\psi_{n}(f)(c_{i}-c_{i+1}) = \frac{1}{\sqrt{-2}}c_{i}x_{i}(c_{i}-c_{i+1}) \ \overset{\mathclap{\cref{tophat1}}}{\underset{\mathclap{\cref{tophat3}}}{=}} \ -\frac{1}{\sqrt{-2}}(c_{i}-c_{i+1})c_{i}x_{i} = -(c_{i}-c_{i+1})\psi_{n}(f).
\end{equation*}
Finally, if $ j = i+1 $, then we have
\begin{equation*}
\psi_{n}(f)(c_{i}-c_{i+1}) = \frac{1}{\sqrt{-2}}c_{i+1}x_{i+1}(c_{i}-c_{i+1}) \ \overset{\mathclap{\cref{tophat1}}}{\underset{\mathclap{\cref{tophat3}}}{=}} \ -\frac{1}{\sqrt{-2}}(c_{i}-c_{i+1})c_{i+1}x_{i+1} = -(c_{i}-c_{i+1})\psi_{n}(f).
\end{equation*}}

Recall that, for each $ 1 \leq i \leq n-1 $, we have the odd Demazure operators $ D_{i} \colon \operatorname{OPol}_{n} \rightarrow \operatorname{OPol}_{n} $ from Definition \ref{oddDemazure}, and we have the Clifford Demazure operators $ \partial_{i} \colon \operatorname{Pol}_{n}(\operatorname{Cl}) \rightarrow \operatorname{Pol}_{n}(\operatorname{Cl}) $ from Definition \ref{Cliff_Demazure}.

\begin{lem}
Let $ n \in \mathbb{N}_{+} $ and $ 1 \leq i \leq n-1 $. Then
\begin{equation} \label{seeing}
\psi_{n}(D_{i}(f)) = \frac{1}{\sqrt{-2}}(c_{i}-c_{i+1})\partial_{i}(\psi_{n}(f)), \qquad f \in \operatorname{OPol}_{n}.
\end{equation}
\end{lem}

\begin{proof}
If $ f \in \mathbb{C} $ or $ f = x_{j} $ for some $ j \neq i,i+1 $, then both sides of \cref{seeing} are equal to zero. Furthermore, it is straightforward to check that \cref{seeing} holds when $ f = x_{i} $ or $ f = x_{i+1} $. Now assume that \cref{seeing} holds for some $ f,g \in \operatorname{OPol}_{n} $. Then we compute that
\begin{align*}
\psi_{n}(D_{i}(fg)) &\stackrel{\mathclap{\cref{oddLeibnizrule}}}{=} \ \psi_{n}(D_{i}(f))\psi_{n}(g) + (-1)^{\bar{f}}\psi_{n}(s_{i}(f))\psi_{n}(D_{i}(g))
\\ &\stackrel{\mathclap{\cref{frown}}}{=} \ \frac{1}{\sqrt{-2}}(c_{i}-c_{i+1})\partial_{i}(\psi_{n}(f))\psi_{n}(g) + \frac{1}{\sqrt{-2}}(c_{i}-c_{i+1})s_{i}(\psi_{n}(f))\partial_{i}(\psi_{n}(g))
\\ &\stackrel{\mathclap{\cref{Cliff_leibniz}}}{=} \ \frac{1}{\sqrt{-2}}(c_{i}-c_{i+1})\partial_{i}(\psi_{n}(fg)).
\end{align*}
In the second equality, we used that $ \psi_{n}(s_{i}(f)) = s_{i}(\psi_{n}(f)) $, which is straightforward to see. Thus \cref{seeing} now follows by induction.
\end{proof}

\details{We check here that \cref{seeing} holds for $ f = x_{i} $ and $ f = x_{i+1} $. We compute that
\begin{multline*}
\frac{1}{\sqrt{-2}}(c_{i}-c_{i+1})\partial_{i}(\psi_{n}(x_{i})) = -\frac{1}{2}(c_{i}-c_{i+1})\partial_{i}(c_{i}x_{i}) 
\\ = -\frac{1}{2}(c_{i}-c_{i+1})c_{i+1}(1+c_{i}c_{i+1}) = \frac{1}{2}(c_{i}-c_{i+1})^{2} = 1 = \psi_{n}(1) = \psi_{n}(D_{i}(x_{i})).
\end{multline*}
Similarly, 
\begin{multline*}
\frac{1}{\sqrt{-2}}(c_{i}-c_{i+1})\partial_{i}(\psi_{n}(x_{i+1})) = -\frac{1}{2}(c_{i}-c_{i+1})\partial_{i}(c_{i+1}x_{i+1}) 
\\ = -\frac{1}{2}(c_{i}-c_{i+1})c_{i}(-1+c_{i}c_{i+1}) = \frac{1}{2}(c_{i}-c_{i+1})^{2} = 1 = \psi_{n}(1) = \psi_{n}(D_{i}(x_{i+1})).
\end{multline*}}

\begin{lem} \label{working_656}
Let $ n \in \mathbb{N}_{+} $, $ 1 \leq i \leq n-1 $, and $ f \in \operatorname{Pol}_{n}(\operatorname{Cl}) $. Assume that $ \psi_{n}^{-1}(f) \in \operatorname{OPol}_{n} $. Then 
\begin{equation} \label{working_655}
\psi_{n}^{-1}(\partial_{i}(f)) = -\frac{1}{\sqrt{-2}}(c_{i}-c_{i+1})D_{i}(\psi_{n}^{-1}(f)).
\end{equation}
\end{lem}

\begin{proof}
Since $ \psi_{n}^{-1}(f) \in \operatorname{OPol}_{n} $, it follows by \cref{seeing} that
\begin{equation} \label{working_123}
\psi_{n}(D_{i}(\psi_{n}^{-1}(f))) = \frac{1}{\sqrt{-2}}(c_{i}-c_{i+1})\partial_{i}(f).
\end{equation}
Applying $ \psi_{n}^{-1} $ to both sides of \cref{working_123} gives
\begin{equation} \label{working_124}
D_{i}(\psi_{n}^{-1}(f)) = \frac{1}{\sqrt{-2}}(c_{i}-c_{i+1})\psi_{n}^{-1}(\partial_{i}(f)).
\end{equation}
Finally, multiplying both sides of \cref{working_124} by $ -\frac{1}{\sqrt{-2}}(c_{i}-c_{i+1}) $ gives \cref{working_655}.
\end{proof}

\subsection{Isomorphism of $ \mathcal{AS}(\operatorname{Cl}) $ and $ \mathcal{AH}(\operatorname{Cl}) $}

Recall that we have the supercategory $ \mathcal{AS}(\operatorname{Cl}) $ from Definition \ref{S(Cl)} and the supercategory $ \mathcal{AH}(\operatorname{Cl}) $ from Definition \ref{DSAHCSdef}. We define the following morphism in $ \mathcal{AS}(\operatorname{Cl}) $:
\begin{equation} \label{green_tele_def}
\begin{tikzpicture}[centerzero]
        \draw (-0.2,-0.3) -- (-0.2,0.3);
        \draw (0.2,-0.3) -- (0.2,0.3);
        \xteleport{-0.2,0}{0.2,0};
    \end{tikzpicture}
:=
\begin{tikzpicture}[centerzero]
        \draw (-0.2,-0.3) -- (-0.2,0.3);
        \draw (0.2,-0.3) -- (0.2,0.3);
        \xtoken{-0.2,0};
    \end{tikzpicture}
\ - \
\begin{tikzpicture}[centerzero]
        \draw (-0.2,-0.3) -- (-0.2,0.3);
        \draw (0.2,-0.3) -- (0.2,0.3);
        \xtoken{0.2,0};
    \end{tikzpicture}\ .
\end{equation}
It is straightforward to check using \cref{spinCl123} that
\begin{equation} \label{cards}
\begin{tikzpicture}[centerzero]
        \draw (-0.2,-0.3) -- (-0.2,0.3);
        \draw (0.2,-0.3) -- (0.2,0.3);
        \xteleport{-0.2,0.15}{0.2,0.15};
        \xteleport{-0.2,-0.15}{0.2,-0.15};
    \end{tikzpicture}
= 2 \
\begin{tikzpicture}[centerzero]
        \draw (-0.2,-0.3) -- (-0.2,0.3);
        \draw (0.2,-0.3) -- (0.2,0.3);
    \end{tikzpicture} \ .
\end{equation}
Furthermore, we define the following morphism in $ \mathcal{AH}(\operatorname{Cl}) $:
\begin{equation} \label{blue_tele}
\begin{tikzpicture}[centerzero]
        \draw (-0.2,-0.3) -- (-0.2,0.3);
        \draw (0.2,-0.3) -- (0.2,0.3);
        \teleport{-0.2,0}{0.2,0};
    \end{tikzpicture}
:=
\begin{tikzpicture}[centerzero]
        \draw (-0.2,-0.3) -- (-0.2,0.3);
        \draw (0.2,-0.3) -- (0.2,0.3);
        \bluetoken{-0.2,0};
    \end{tikzpicture}
\ - \
\begin{tikzpicture}[centerzero]
        \draw (-0.2,-0.3) -- (-0.2,0.3);
        \draw (0.2,-0.3) -- (0.2,0.3);
        \bluetoken{0.2,0};
    \end{tikzpicture} \ .
\end{equation}
It is straightforward to check using \cref{Pol(Cl)} that
\begin{equation} \label{tele_squared}
\begin{tikzpicture}[centerzero]
        \draw (-0.2,-0.3) -- (-0.2,0.3);
        \draw (0.2,-0.3) -- (0.2,0.3);
        \teleport{-0.2,0.15}{0.2,0.15};
        \teleport{-0.2,-0.15}{0.2,-0.15};
    \end{tikzpicture}
= 2 \
\begin{tikzpicture}[centerzero]
        \draw (-0.2,-0.3) -- (-0.2,0.3);
        \draw (0.2,-0.3) -- (0.2,0.3);
    \end{tikzpicture} \ .
\end{equation}

\begin{prop} \label{AS-AH-iso}
We have an isomorphism of monoidal supercategories $ \Phi \colon \mathcal{AS}(\operatorname{Cl}) \rightarrow \mathcal{AH}(\operatorname{Cl}) $ which sends $ \go $ to $ \go $, and sends the generating morphisms as follows:
\begin{gather*}
\Phi \left(\begin{tikzpicture}[centerzero, thick]
      \draw[-] (0.6,-0.4) -- (0.6,0.4);
        \xtoken{0.6,0};
\end{tikzpicture}\right) = \begin{tikzpicture}[centerzero, thick]
      \draw[-] (0.6,-0.4) -- (0.6,0.4);
      \bluetoken{0.6,0};
\end{tikzpicture} \ , 
\quad 
\Phi \left(\begin{tikzpicture}[centerzero, thick]
      \draw[-] (0.6,-0.4) -- (0.6,0.4);
      \singdot{0.6,0};
\end{tikzpicture}\right) = \frac{1}{\sqrt{-2}} \
\begin{tikzpicture}[centerzero, thick]
      \draw[-] (0.6,-0.4) -- (0.6,0.4);
      \singdot{0.6,-0.15};
      \bluetoken{0.6,0.15};
\end{tikzpicture} \ ,
\quad 
\Phi \left(\begin{tikzpicture}[centerzero, thick]
      \draw[-] (0.9,-0.4) -- (0.3,0.4);
      \draw[-] (0.3,-0.4) -- (0.9,0.4);
\end{tikzpicture}\right) 
= - \frac{1}{\sqrt{-2}} \
\begin{tikzpicture}[centerzero, thick]
      \draw[-] (0.9,-0.4) -- (0.3,0.4);
      \draw[-] (0.3,-0.4) -- (0.9,0.4);
      \teleport{0.45,0.2}{0.75,0.2};
\end{tikzpicture} \ .
\end{gather*}
Its inverse $ \Phi^{-1} \colon \mathcal{AH}(\operatorname{Cl}) \rightarrow \mathcal{AS}(\operatorname{Cl}) $ sends $ \go $ to $ \go $, and sends the generating morphisms as follows:
\begin{gather*}
\Phi^{-1} \left(\begin{tikzpicture}[centerzero, thick]
      \draw[-] (0.6,-0.4) -- (0.6,0.4);
      \bluetoken{0.6,0};
\end{tikzpicture}\right) = \begin{tikzpicture}[centerzero, thick]
      \draw[-] (0.6,-0.4) -- (0.6,0.4);
      \xtoken{0.6,0};
\end{tikzpicture} \ , 
\quad 
\Phi^{-1} \left(\begin{tikzpicture}[centerzero, thick]
      \draw[-] (0.6,-0.4) -- (0.6,0.4);
      \singdot{0.6,0};
\end{tikzpicture}\right) = \sqrt{-2} \
\begin{tikzpicture}[centerzero, thick]
      \draw[-] (0.6,-0.4) -- (0.6,0.4);
      \singdot{0.6,-0.15};
      \xtoken{0.6,0.15};
\end{tikzpicture} \ ,
\quad 
\Phi^{-1} \left(\begin{tikzpicture}[centerzero, thick]
      \draw[-] (0.9,-0.4) -- (0.3,0.4);
      \draw[-] (0.3,-0.4) -- (0.9,0.4);
\end{tikzpicture}\right) 
= \frac{1}{\sqrt{-2}} \
\begin{tikzpicture}[centerzero, thick]
      \draw[-] (0.9,-0.4) -- (0.3,0.4);
      \draw[-] (0.3,-0.4) -- (0.9,0.4);
      \xteleport{0.45,0.2}{0.75,0.2};
\end{tikzpicture} \ .
\end{gather*}
\end{prop}

\begin{proof}
It is straightforward to verify that the given maps respect the defining relations and are mutually inverse. Alternatively, this proposition is a consequence of \cite[Thm.~4.1]{Wang}.
\end{proof}

\details{We first prove the existence of $ \Phi $, which requires checking the relations \cref{spin1}, \cref{spin2}, \cref{spinCl123}, and \cref{spinCl1}. For \cref{spin1}, we first compute that
\begin{equation*}
\Phi \left(\begin{tikzpicture}[anchorbase, thick]
        \draw[-] (0.2,-0.5) to[out=up,in=down] (-0.2,0) to[out=up,in=down] (0.2,0.5);
        \draw[-] (-0.2,-0.5) to[out=up,in=down] (0.2,0) to[out=up,in=down] (-0.2,0.5);
\end{tikzpicture}\right)
= - \frac{1}{2} \
\begin{tikzpicture}[anchorbase, thick]
        \draw[-] (0.2,-0.5) to[out=up,in=down] (-0.2,0) to[out=up,in=down] (0.2,0.5);
        \draw[-] (-0.2,-0.5) to[out=up,in=down] (0.2,0) to[out=up,in=down] (-0.2,0.5);
        \teleport{-0.2,0}{0.2,0};
        \teleport{-0.15,0.35}{0.15,0.35};
\end{tikzpicture}
\quad \overset{\mathclap{\cref{Cliff_hw-1}}}{\underset{\mathclap{\cref{Cliff_hw-5}}}{=}} \ \quad
\frac{1}{2} \
\begin{tikzpicture}[anchorbase, thick]
        \draw[-] (0.2,-0.5) to[out=up,in=down] (-0.2,0) to[out=up,in=down] (0.2,0.5);
        \draw[-] (-0.2,-0.5) to[out=up,in=down] (0.2,0) to[out=up,in=down] (-0.2,0.5);
        \teleport{-0.2,-0.05}{0.2,-0.05};
        \teleport{-0.15,0.15}{0.15,0.15};
\end{tikzpicture}
\quad \stackrel{\mathclap{\cref{tele_squared}}}{=} \quad 
\begin{tikzpicture}[anchorbase, thick]
        \draw[-] (0.2,-0.5) to[out=up,in=down] (-0.2,0) to[out=up,in=down] (0.2,0.5);
        \draw[-] (-0.2,-0.5) to[out=up,in=down] (0.2,0) to[out=up,in=down] (-0.2,0.5);
\end{tikzpicture}
\quad \stackrel{\mathclap{\cref{Cliff_hw-1}}}{=} \quad
\begin{tikzpicture}[anchorbase, thick]
        \draw[-] (0.2,-0.5) -- (0.2,0.5);
        \draw[-] (-0.2,-0.5) -- (-0.2,0.5);
\end{tikzpicture} 
=
\Phi \left( \ \begin{tikzpicture}[anchorbase, thick]
        \draw[-] (0.2,-0.5) -- (0.2,0.5);
        \draw[-] (-0.2,-0.5) -- (-0.2,0.5);
\end{tikzpicture} 
\ \right).
\end{equation*}
Next, for the second relation of \cref{spin1}, we compute that
\begin{align}
\Phi \left(\begin{tikzpicture}[anchorbase, thick]
        \draw[-] (0.8,-1) -- (-0.8,1);
        \draw[-] (0,-1) to[out=up, in=down] (-0.8,0) to[out=up,in=down] (0,1);
        \draw[-] (-0.8,-1) -- (0.8,1);
\end{tikzpicture}\right)
&= \left(\frac{-1}{\sqrt{-2}}\right)^{3} \
\begin{tikzpicture}[anchorbase, thick]
        \draw[-] (0.8,-1) -- (-0.8,1);
        \draw[-] (0,-1) to[out=up, in=down] (-0.8,0) to[out=up,in=down] (0,1);
        \draw[-] (-0.8,-1) -- (0.8,1);
        \teleport{-0.5,0.6}{-0.2,0.6};
        \teleport{-0.15,0.2}{0.15,0.2};
        \teleport{-0.7,-0.3}{-0.2,-0.3};
\end{tikzpicture}
\quad \overset{\mathclap{\cref{Cliff_hw-1}}}{\underset{\mathclap{\cref{Cliff_hw-5}}}{=}} \quad
\left(\frac{-1}{\sqrt{-2}}\right)^{3} \
\begin{tikzpicture}[anchorbase, thick]
        \draw[-] (0.8,-1) -- (-0.8,1);
        \draw[-] (0,-1) to[out=up, in=down] (-0.8,0) to[out=up,in=down] (0,1);
        \draw[-] (-0.8,-1) -- (0.8,1);
        \teleport{-0.75,0.9}{0,0.9};
        \teleport{-0.6,0.75}{0.6,0.75};
        \teleport{0.5,0.6}{-0.2,0.6};
\end{tikzpicture} \ , \label{braidimage456}
\\ \Phi \left(\begin{tikzpicture}[anchorbase, thick]
        \draw[-] (0.8,-1) -- (-0.8,1);
        \draw[-] (0,-1) to[out=up, in=down] (0.8,0) to[out=up,in=down] (0,1);
        \draw[-] (-0.8,-1) -- (0.8,1);
\end{tikzpicture}\right)
&= \left(\frac{-1}{\sqrt{-2}}\right)^{3} \
\begin{tikzpicture}[anchorbase, thick]
        \draw[-] (0.8,-1) -- (-0.8,1);
        \draw[-] (0,-1) to[out=up, in=down] (0.8,0) to[out=up,in=down] (0,1);
        \draw[-] (-0.8,-1) -- (0.8,1);
        \teleport{0.5,0.6}{0.2,0.6};
        \teleport{-0.15,0.2}{0.15,0.2};
        \teleport{0.7,-0.3}{0.2,-0.3};
\end{tikzpicture}
\quad \overset{\mathclap{\cref{Cliff_hw-1}}}{\underset{\mathclap{\cref{Cliff_hw-5}}}{=}} \quad
\left(\frac{-1}{\sqrt{-2}}\right)^{3} \
\begin{tikzpicture}[anchorbase, thick]
        \draw[-] (0.8,-1) -- (-0.8,1);
        \draw[-] (0,-1) to[out=up, in=down] (0.8,0) to[out=up,in=down] (0,1);
        \draw[-] (-0.8,-1) -- (0.8,1);
        \teleport{0.75,0.9}{0,0.9};
        \teleport{-0.6,0.75}{0.6,0.75};
        \teleport{-0.5,0.6}{0.2,0.6};
\end{tikzpicture} \ . \label{braidimage457}
\end{align}
Thus we have that this second relation of \cref{spin1} will follow if we are able to show that
\begin{equation} \label{Stanley2}
\begin{tikzpicture}[anchorbase, thick]
        \draw[-] (-0.4,-0.5) -- (-0.4,0.5);
        \draw[-] (0,-0.5) -- (0,0.5);
        \draw[-] (0.4,-0.5) -- (0.4,0.5);
        \teleport{-0.4,0.25}{0,0.25};
        \teleport{0.4,0}{-0.4,0};
        \teleport{0.4,-0.25}{0,-0.25};
\end{tikzpicture} 
=
\begin{tikzpicture}[anchorbase, thick]
        \draw[-] (-0.4,-0.5) -- (-0.4,0.5);
        \draw[-] (0,-0.5) -- (0,0.5);
        \draw[-] (0.4,-0.5) -- (0.4,0.5);
        \teleport{0,0.25}{0.4,0.25};
        \teleport{-0.4,0}{0.4,0};
        \teleport{0,-0.25}{-0.4,-0.25};
\end{tikzpicture} \ .
\end{equation}
The equation \cref{Stanley2} is straightforward to check. For \cref{spin2}, we compute that
\begin{align*}
\Phi \left(\begin{tikzpicture}[centerzero, thick]
        \draw[-] (-0.5,-0.7) -- (0.5,0.7);
        \draw[-] (0.5,-0.7) -- (-0.5,0.7);
        \singdot{-0.25,-0.35};
\end{tikzpicture}
\ + \
\begin{tikzpicture}[centerzero, thick]
        \draw[-] (-0.5,-0.7) -- (0.5,0.7);
        \draw[-] (0.5,-0.7) -- (-0.5,0.7);
        \singdot{0.25,0.35};
\end{tikzpicture}\right)
&= \frac{1}{2}\left(\begin{tikzpicture}[centerzero, thick]
        \draw[-] (0.5,-0.7) -- (-0.5,0.7);
        \draw[-] (-0.5,-0.7) -- (0.5,0.7);
        \singdot{-0.25,-0.35};
        \teleport{-0.25,0.35}{0.25,0.35};
        \bluetoken{-0.1,-0.15};
\end{tikzpicture}
\ + \
\begin{tikzpicture}[centerzero, thick]
        \draw[-] (0.5,-0.7) -- (-0.5,0.7);
        \draw[-] (-0.5,-0.7) -- (0.5,0.7);
        \singdot{0.23,0.3};
        \teleport{-0.1,0.13}{0.1,0.13};
        \bluetoken{0.35,0.5};
\end{tikzpicture}\right)
\\ &\stackrel{\mathclap{\cref{Pol(Cl)}}}{=} \ \frac{1}{2}\left(\begin{tikzpicture}[centerzero, thick]
        \draw[-] (-0.5,-0.7) -- (0.5,0.7);
        \draw[-] (0.5,-0.7) -- (-0.5,0.7);
        \singdot{-0.25,-0.35};
        \bluetoken{-0.25,0.35};
        \bluetoken{0.25,0.35};
\end{tikzpicture}
\ - \
\begin{tikzpicture}[centerzero, thick]
        \draw[-] (-0.5,-0.7) -- (0.5,0.7);
        \draw[-] (0.5,-0.7) -- (-0.5,0.7);
        \singdot{0.25,0.35};
        \bluetoken{0.35,0.5};
        \bluetoken{-0.35,0.5};
\end{tikzpicture}\right)
- \frac{1}{2}\left(\begin{tikzpicture}[centerzero, thick]
        \draw[-] (-0.5,-0.7) -- (0.5,0.7);
        \draw[-] (0.5,-0.7) -- (-0.5,0.7);
        \singdot{-0.25,-0.35};
\end{tikzpicture} 
\ - \
\begin{tikzpicture}[centerzero, thick]
        \draw[-] (-0.5,-0.7) -- (0.5,0.7);
        \draw[-] (0.5,-0.7) -- (-0.5,0.7);
        \singdot{0.25,0.35};
\end{tikzpicture}\right)
\\ &\stackrel{\mathclap{\cref{Cliff_hw0}}}{=} \ \frac{1}{2}\left( - \ \begin{tikzpicture}[centerzero, thick]
        \draw[-] (-0.3,-0.7) -- (-0.3,0.7);
        \draw[-] (0.3,-0.7) -- (0.3,0.7);
        \bluetoken{-0.3,0};
        \bluetoken{0.3,0};
\end{tikzpicture}
\ + \
\begin{tikzpicture}[centerzero, thick]
        \draw[-] (-0.3,-0.7) -- (-0.3,0.7);
        \draw[-] (0.3,-0.7) -- (0.3,0.7);
\end{tikzpicture} \right)
-\frac{1}{2}\left(- \ \begin{tikzpicture}[centerzero, thick]
        \draw[-] (-0.3,-0.7) -- (-0.3,0.7);
        \draw[-] (0.3,-0.7) -- (0.3,0.7);
\end{tikzpicture}
\ - \
\begin{tikzpicture}[centerzero, thick]
        \draw[-] (-0.3,-0.7) -- (-0.3,0.7);
        \draw[-] (0.3,-0.7) -- (0.3,0.7);
        \bluetoken{-0.3,0};
        \bluetoken{0.3,0};
\end{tikzpicture} \right)
\\ &= \Phi \left( \ \begin{tikzpicture}[centerzero, thick]
        \draw[-] (-0.3,-0.7) -- (-0.3,0.7);
        \draw[-] (0.3,-0.7) -- (0.3,0.7);
\end{tikzpicture} \ \right).
\end{align*}
For \cref{spinCl123}, we compute that
\begin{align*}
\Phi \left(\begin{tikzpicture}[anchorbase, thick]
        \draw[-] (0,0) -- (0,0.7);
        \xtoken{0,0.2};
        \xtoken{0,0.45};
\end{tikzpicture}\right)
&= 
\begin{tikzpicture}[anchorbase, thick]
        \draw[-] (0,0) -- (0,0.7);
        \bluetoken{0,0.2};
        \bluetoken{0,0.45};
\end{tikzpicture}
\quad \stackrel{\mathclap{\cref{Pol(Cl)}}}{=} \quad
\begin{tikzpicture}[anchorbase, thick]
        \draw[-] (0,0) -- (0,0.7);
\end{tikzpicture} \ ,
\\ \Phi \left(\begin{tikzpicture}[centerzero, thick]
        \draw[-] (0,-0.4) -- (0,0.4);
        \xtoken{0,0.15};
        \singdot{0,-0.15};
\end{tikzpicture}\right)
&= \frac{1}{\sqrt{-2}} \ 
\begin{tikzpicture}[centerzero, thick]
        \draw[-] (0,-0.4) -- (0,0.4);
        \bluetoken{0,0.25};
        \bluetoken{0,0};
        \singdot{0,-0.25};
\end{tikzpicture}
\quad \stackrel{\mathclap{\cref{Pol(Cl)}}}{=} \quad
-\frac{1}{\sqrt{-2}} \ \begin{tikzpicture}[centerzero, thick]
        \draw[-] (0,-0.4) -- (0,0.4);
        \bluetoken{0,0.25};
        \bluetoken{0,-0.25};
        \singdot{0,0};
\end{tikzpicture}
= \Phi \left(-
\begin{tikzpicture}[centerzero, thick]
        \draw[-] (0,-0.4) -- (0,0.4);
        \xtoken{0,-0.15};
        \singdot{0,0.15};
\end{tikzpicture}\right).
\end{align*}
For \cref{spinCl1}, we compute that
\begin{gather*}
\Phi \left(\begin{tikzpicture}[centerzero, thick]
        \draw[-] (0.3,-0.4) -- (-0.3,0.4);
        \draw[-] (-0.3,-0.4) -- (0.3,0.4);
        \xtoken{-0.15,-0.2};
\end{tikzpicture}\right)
= -\frac{1}{\sqrt{-2}} \ 
\begin{tikzpicture}[centerzero, thick]
        \draw[-] (0.3,-0.4) -- (-0.3,0.4);
        \draw[-] (-0.3,-0.4) -- (0.3,0.4);
        \bluetoken{-0.15,-0.2};
        \teleport{-0.15,0.2}{0.15,0.2};
\end{tikzpicture}
\quad \overset{\mathclap{\cref{Pol(Cl)}}}{\underset{\mathclap{\cref{Cliff_hw-1}}}{=}} \quad
\frac{1}{\sqrt{-2}} \ 
\begin{tikzpicture}[centerzero, thick]
        \draw[-] (0.3,-0.4) -- (-0.3,0.4);
        \draw[-] (-0.3,-0.4) -- (0.3,0.4);
        \bluetoken{-0.23,0.32};
        \teleport{-0.13,0.15}{0.13,0.15};
\end{tikzpicture}
= \Phi \left(-
\begin{tikzpicture}[centerzero, thick]
        \draw[-] (0.3,-0.4) -- (-0.3,0.4);
        \draw[-] (-0.3,-0.4) -- (0.3,0.4);
        \xtoken{-0.15,0.2};
\end{tikzpicture}\right),
\\ 
\Phi \left(\begin{tikzpicture}[centerzero, thick]
        \draw[-] (0.3,-0.4) -- (-0.3,0.4);
        \draw[-] (-0.3,-0.4) -- (0.3,0.4);
        \xtoken{0.15,-0.2};
\end{tikzpicture}\right)
= -\frac{1}{\sqrt{-2}} \ 
\begin{tikzpicture}[centerzero, thick]
        \draw[-] (0.3,-0.4) -- (-0.3,0.4);
        \draw[-] (-0.3,-0.4) -- (0.3,0.4);
        \bluetoken{0.15,-0.2};
        \teleport{-0.15,0.2}{0.15,0.2};
\end{tikzpicture}
\quad \overset{\mathclap{\cref{Pol(Cl)}}}{\underset{\mathclap{\cref{Cliff_hw-5}}}{=}} \quad
\frac{1}{\sqrt{-2}} \ 
\begin{tikzpicture}[centerzero, thick]
        \draw[-] (0.3,-0.4) -- (-0.3,0.4);
        \draw[-] (-0.3,-0.4) -- (0.3,0.4);
        \bluetoken{0.23,0.32};
        \teleport{-0.13,0.15}{0.13,0.15};
\end{tikzpicture}
= \Phi
\left(-
\begin{tikzpicture}[centerzero, thick]
        \draw[-] (0.3,-0.4) -- (-0.3,0.4);
        \draw[-] (-0.3,-0.4) -- (0.3,0.4);
        \xtoken{0.15,0.2};
\end{tikzpicture}\right).
\end{gather*}
We next prove the existence of $ \Phi^{-1} $, which requires checking the relations \cref{Pol(Cl)}, \cref{Cliff_hw-1}, and \cref{Cliff_hw0}. For \cref{Pol(Cl)}, we compute that 
\begin{gather*}
\Phi^{-1} \left(\begin{tikzpicture}[centerzero, thick]
        \draw[-] (0,-0.4) -- (0,0.4);
        \bluetoken{0,0.15};
        \singdot{0,-0.15};
\end{tikzpicture}\right)
= \sqrt{-2} \
\begin{tikzpicture}[centerzero, thick]
        \draw[-] (0,-0.4) -- (0,0.4);
        \xtoken{0,0.25};
        \xtoken{0,0};
        \singdot{0,-0.25};
\end{tikzpicture}
\quad \stackrel{\mathclap{\cref{spinCl123}}}{=} \quad 
-\sqrt{-2} \
\begin{tikzpicture}[centerzero, thick]
        \draw[-] (0,-0.4) -- (0,0.4);
        \xtoken{0,0.25};
        \xtoken{0,-0.25};
        \singdot{0,0};
\end{tikzpicture}
=
\Phi^{-1} \left(- \begin{tikzpicture}[centerzero, thick]
        \draw[-] (0,-0.4) -- (0,0.4);
        \bluetoken{0,-0.15};
        \singdot{0,0.15};
\end{tikzpicture}\right),
\\ \Phi^{-1} \left(\begin{tikzpicture}[anchorbase, thick]
        \draw[-] (0,0) -- (0,0.7);
        \bluetoken{0,0.2};
        \bluetoken{0,0.45};
\end{tikzpicture}\right)
= 
\begin{tikzpicture}[anchorbase, thick]
        \draw[-] (0,0) -- (0,0.7);
        \xtoken{0,0.2};
        \xtoken{0,0.45};
\end{tikzpicture}
\quad \stackrel{\mathclap{\cref{spinCl123}}}{=} \quad
\begin{tikzpicture}[anchorbase, thick]
        \draw[-] (0,0) -- (0,0.7);
\end{tikzpicture} \ .
\end{gather*}
For the first relation in \cref{Cliff_hw-1}, we have
\begin{equation*}
\Phi^{-1} \left(\begin{tikzpicture}[anchorbase, thick]
        \draw[-] (0.2,-0.5) to[out=up,in=down] (-0.2,0) to[out=up,in=down] (0.2,0.5);
        \draw[-] (-0.2,-0.5) to[out=up,in=down] (0.2,0) to[out=up,in=down] (-0.2,0.5);
\end{tikzpicture}\right)
= - \frac{1}{2} \
\begin{tikzpicture}[anchorbase, thick]
        \draw[-] (0.2,-0.5) to[out=up,in=down] (-0.2,0) to[out=up,in=down] (0.2,0.5);
        \draw[-] (-0.2,-0.5) to[out=up,in=down] (0.2,0) to[out=up,in=down] (-0.2,0.5);
        \xteleport{-0.2,0}{0.2,0};
        \xteleport{-0.15,0.35}{0.15,0.35};
\end{tikzpicture}
\quad \stackrel{\mathclap{\cref{spinCl1}}}{=} \quad
\frac{1}{2} \
\begin{tikzpicture}[anchorbase, thick]
        \draw[-] (0.2,-0.5) to[out=up,in=down] (-0.2,0) to[out=up,in=down] (0.2,0.5);
        \draw[-] (-0.2,-0.5) to[out=up,in=down] (0.2,0) to[out=up,in=down] (-0.2,0.5);
        \xteleport{-0.2,-0.05}{0.2,-0.05};
        \xteleport{-0.15,0.15}{0.15,0.15};
\end{tikzpicture}
\quad \stackrel{\mathclap{\cref{cards}}}{=} \quad 
\begin{tikzpicture}[anchorbase, thick]
        \draw[-] (0.2,-0.5) to[out=up,in=down] (-0.2,0) to[out=up,in=down] (0.2,0.5);
        \draw[-] (-0.2,-0.5) to[out=up,in=down] (0.2,0) to[out=up,in=down] (-0.2,0.5);
\end{tikzpicture}
\quad \stackrel{\mathclap{\cref{spin1}}}{=} \quad
\begin{tikzpicture}[anchorbase, thick]
        \draw[-] (0.2,-0.5) -- (0.2,0.5);
        \draw[-] (-0.2,-0.5) -- (-0.2,0.5);
\end{tikzpicture} 
= 
\Phi^{-1} \left( \ \begin{tikzpicture}[anchorbase, thick]
        \draw[-] (0.2,-0.5) -- (0.2,0.5);
        \draw[-] (-0.2,-0.5) -- (-0.2,0.5);
\end{tikzpicture} \
\right).
\end{equation*}
Next, for the second relation in \cref{Cliff_hw-1}, we have that
\begin{align}
\Phi^{-1} \left(\begin{tikzpicture}[anchorbase, thick]
        \draw[-] (0.8,-1) -- (-0.8,1);
        \draw[-] (0,-1) to[out=up, in=down] (-0.8,0) to[out=up,in=down] (0,1);
        \draw[-] (-0.8,-1) -- (0.8,1);
\end{tikzpicture}\right)
&= \left(\frac{1}{\sqrt{-2}}\right)^{3} \
\begin{tikzpicture}[anchorbase, thick]
        \draw[-] (0.8,-1) -- (-0.8,1);
        \draw[-] (0,-1) to[out=up, in=down] (-0.8,0) to[out=up,in=down] (0,1);
        \draw[-] (-0.8,-1) -- (0.8,1);
        \xteleport{-0.5,0.6}{-0.2,0.6};
        \xteleport{-0.15,0.2}{0.15,0.2};
        \xteleport{-0.7,-0.3}{-0.2,-0.3};
\end{tikzpicture}
\quad \stackrel{\mathclap{\cref{spinCl1}}}{=} \quad
\left(\frac{1}{\sqrt{-2}}\right)^{3} \
\begin{tikzpicture}[anchorbase, thick]
        \draw[-] (0.8,-1) -- (-0.8,1);
        \draw[-] (0,-1) to[out=up, in=down] (-0.8,0) to[out=up,in=down] (0,1);
        \draw[-] (-0.8,-1) -- (0.8,1);
        \xteleport{-0.75,0.9}{0,0.9};
        \xteleport{-0.1,0.75}{0.6,0.75};
        \xteleport{-0.5,0.6}{-0.2,0.6};
\end{tikzpicture} \ , \label{braidimage}
\\ \Phi^{-1}  \left(\begin{tikzpicture}[anchorbase, thick]
        \draw[-] (0.8,-1) -- (-0.8,1);
        \draw[-] (0,-1) to[out=up, in=down] (0.8,0) to[out=up,in=down] (0,1);
        \draw[-] (-0.8,-1) -- (0.8,1);
\end{tikzpicture}\right)
&= \left(\frac{1}{\sqrt{-2}}\right)^{3} \
\begin{tikzpicture}[anchorbase, thick]
        \draw[-] (0.8,-1) -- (-0.8,1);
        \draw[-] (0,-1) to[out=up, in=down] (0.8,0) to[out=up,in=down] (0,1);
        \draw[-] (-0.8,-1) -- (0.8,1);
        \xteleport{0.5,0.6}{0.2,0.6};
        \xteleport{-0.15,0.2}{0.15,0.2};
        \xteleport{0.7,-0.3}{0.2,-0.3};
\end{tikzpicture}
\quad \stackrel{\mathclap{\cref{spinCl1}}}{=} \quad
\left(\frac{1}{\sqrt{-2}}\right)^{3} \
\begin{tikzpicture}[anchorbase, thick]
        \draw[-] (0.8,-1) -- (-0.8,1);
        \draw[-] (0,-1) to[out=up, in=down] (0.8,0) to[out=up,in=down] (0,1);
        \draw[-] (-0.8,-1) -- (0.8,1);
        \xteleport{0.75,0.9}{0,0.9};
        \xteleport{0.1,0.75}{-0.6,0.75};
        \xteleport{0.5,0.6}{0.2,0.6};
\end{tikzpicture} \ . \label{braidimage2}
\end{align}
Thus the second relation of \cref{Cliff_hw-1} will follow if we are able to show that
\begin{equation} \label{Stanley}
\begin{tikzpicture}[anchorbase, thick]
        \draw[-] (-0.4,-0.5) -- (-0.4,0.5);
        \draw[-] (0,-0.5) -- (0,0.5);
        \draw[-] (0.4,-0.5) -- (0.4,0.5);
        \xteleport{-0.4,0.25}{0,0.25};
        \xteleport{0.4,0}{0,0};
        \xteleport{-0.4,-0.25}{0,-0.25};
\end{tikzpicture} 
=
\begin{tikzpicture}[anchorbase, thick]
        \draw[-] (-0.4,-0.5) -- (-0.4,0.5);
        \draw[-] (0,-0.5) -- (0,0.5);
        \draw[-] (0.4,-0.5) -- (0.4,0.5);
        \xteleport{0,0.25}{0.4,0.25};
        \xteleport{-0.4,0}{0,0};
        \xteleport{0,-0.25}{0.4,-0.25};
\end{tikzpicture} \ .
\end{equation}
The equation \cref{Stanley} is straightforward to check. Finally, for the third relation in \cref{Cliff_hw-1}, we have
\begin{equation*}
\Phi^{-1}  \left(\begin{tikzpicture}[centerzero, thick]
        \draw[-] (0.3,-0.4) -- (-0.3,0.4);
        \draw[-] (-0.3,-0.4) -- (0.3,0.4);
        \bluetoken{-0.15,-0.2};
\end{tikzpicture}\right)
= \frac{1}{\sqrt{-2}} \
\begin{tikzpicture}[centerzero, thick]
        \draw[-] (0.3,-0.4) -- (-0.3,0.4);
        \draw[-] (-0.3,-0.4) -- (0.3,0.4);
        \xtoken{-0.15,-0.2};
        \xteleport{-0.13,0.15}{0.13,0.15};
\end{tikzpicture}
\quad \stackrel{\mathclap{\cref{spinCl1}}}{=} \quad \frac{1}{\sqrt{-2}} \
\begin{tikzpicture}[centerzero, thick]
        \draw[-] (0.3,-0.4) -- (-0.3,0.4);
        \draw[-] (-0.3,-0.4) -- (0.3,0.4);
        \xtoken{0.23,0.32};
        \xteleport{-0.13,0.15}{0.13,0.15};
\end{tikzpicture}
=
\Phi^{-1} \left(\begin{tikzpicture}[centerzero, thick]
        \draw[-] (0.3,-0.4) -- (-0.3,0.4);
        \draw[-] (-0.3,-0.4) -- (0.3,0.4);
        \bluetoken{0.15,0.2};
\end{tikzpicture}\right).
\end{equation*}
For \cref{Cliff_hw0}, we compute using \cref{spin2}, \cref{spinCl1} and \cref{green_tele_def} that
\begin{align*}
\Phi^{-1} \left(\begin{tikzpicture}[centerzero, thick]
        \draw[-] (-0.5,-0.7) -- (0.5,0.7);
        \draw[-] (0.5,-0.7) -- (-0.5,0.7);
        \singdot{-0.25,-0.35};
\end{tikzpicture}
\ - \
\begin{tikzpicture}[centerzero, thick]
        \draw[-] (-0.5,-0.7) -- (0.5,0.7);
        \draw[-] (0.5,-0.7) -- (-0.5,0.7);
        \singdot{0.25,0.35};
\end{tikzpicture}\right)
&= \begin{tikzpicture}[centerzero, thick]
        \draw[-] (0.5,-0.7) -- (-0.5,0.7);
        \draw[-] (-0.5,-0.7) -- (0.5,0.7);
        \singdot{-0.25,-0.35};
        \xteleport{-0.25,0.35}{0.25,0.35};
        \xtoken{-0.1,-0.15};
\end{tikzpicture}
\ - \
\begin{tikzpicture}[centerzero, thick]
        \draw[-] (0.5,-0.7) -- (-0.5,0.7);
        \draw[-] (-0.5,-0.7) -- (0.5,0.7);
        \singdot{0.23,0.3};
        \xteleport{-0.1,0.13}{0.1,0.13};
        \xtoken{0.35,0.5};
\end{tikzpicture}
\\ &\overset{\mathclap{\cref{spinCl123}}}{\underset{\mathclap{\cref{spinCl1}}}{=}} \quad - \left(\begin{tikzpicture}[centerzero, thick]
        \draw[-] (-0.5,-0.7) -- (0.5,0.7);
        \draw[-] (0.5,-0.7) -- (-0.5,0.7);
        \singdot{-0.25,-0.35};
\end{tikzpicture}
\ + \
\begin{tikzpicture}[centerzero, thick]
        \draw[-] (-0.5,-0.7) -- (0.5,0.7);
        \draw[-] (0.5,-0.7) -- (-0.5,0.7);
        \singdot{0.25,0.35};
\end{tikzpicture}\right)
- \left(\begin{tikzpicture}[centerzero, thick]
        \draw[-] (-0.5,-0.7) -- (0.5,0.7);
        \draw[-] (0.5,-0.7) -- (-0.5,0.7);
        \singdot{-0.25,-0.35};
        \xtoken{-0.25,0.35};
        \xtoken{0.25,0.35};
\end{tikzpicture} 
\ + \
\begin{tikzpicture}[centerzero, thick]
        \draw[-] (-0.5,-0.7) -- (0.5,0.7);
        \draw[-] (0.5,-0.7) -- (-0.5,0.7);
        \singdot{0.25,0.35};
        \xtoken{-0.35,0.5};
        \xtoken{0.35,0.5};
\end{tikzpicture}\right)
\\ &\stackrel{\mathclap{\cref{spin2}}}{=} \quad - \ 
\begin{tikzpicture}[centerzero, thick]
        \draw[-] (-0.3,-0.7) -- (-0.3,0.7);
        \draw[-] (0.3,-0.7) -- (0.3,0.7);
\end{tikzpicture} 
\ - \
\begin{tikzpicture}[centerzero, thick]
        \draw[-] (-0.3,-0.7) -- (-0.3,0.7);
        \draw[-] (0.3,-0.7) -- (0.3,0.7);
        \xtoken{-0.3,0};
        \xtoken{0.3,0};
\end{tikzpicture} 
\\ &= \Phi^{-1} \left( - \ 
\begin{tikzpicture}[centerzero, thick]
        \draw[-] (-0.3,-0.7) -- (-0.3,0.7);
        \draw[-] (0.3,-0.7) -- (0.3,0.7);
\end{tikzpicture} 
\ - \
\begin{tikzpicture}[centerzero, thick]
        \draw[-] (-0.3,-0.7) -- (-0.3,0.7);
        \draw[-] (0.3,-0.7) -- (0.3,0.7);
        \bluetoken{-0.3,0};
        \bluetoken{0.3,0};
\end{tikzpicture} \right).
\end{align*}
It remains to check that $ \Phi $ and $ \Phi^{-1} $ are indeed mutually inverse. This step is straightforward, and so we omit the details here.}

\subsection{Isomorphism of $ \mathcal{T\!AS}(\operatorname{Cl}) $ and $ \mathcal{T\!AH}(\operatorname{Cl}) $}

It follows from Lemma \ref{Psi(H)} that $ \psi(Q) \in \mathcal{E}_{\operatorname{Pol}_{1}(\operatorname{Cl})} $ for all $ Q \in \mathcal{E}_{\operatorname{OPol}_{1}} $. Also, by Lemma \ref{Psi(H)}, $ \psi^{-1}(Q) \in \mathcal{E}_{\operatorname{OPol}_{1}} $ for all $ Q \in \mathcal{E}_{\operatorname{Pol}_{1}(\operatorname{Cl})} $. We use these facts in the statement of the following theorem. 

\begin{theo} \label{Morita_equiv2}
We have an isomorphism of monoidal supercategories $ \mathcal{F} \colon \mathcal{T\!AS}(\operatorname{Cl}) \rightarrow \mathcal{T\!AH}(\operatorname{Cl}) $ which sends $ \go $ to $ \go $ and $ Q \in \mathcal{E}_{\operatorname{OPol}_{1}} $ to $ \psi(Q) \in \mathcal{E}_{\operatorname{Pol}_{1}(\operatorname{Cl})} $, and sends the generating morphisms as follows:
\begin{gather*}
\mathcal{F} \left(\begin{tikzpicture}[centerzero, thick]
      \draw[-] (0.6,-0.4) -- (0.6,0.4);
        \xtoken{0.6,0};
\end{tikzpicture}\right) = \begin{tikzpicture}[centerzero, thick]
      \draw[-] (0.6,-0.4) -- (0.6,0.4);
      \bluetoken{0.6,0};
\end{tikzpicture} \ , 
\quad 
\mathcal{F} \left(\begin{tikzpicture}[centerzero, thick]
      \draw[-] (0.6,-0.4) -- (0.6,0.4);
      \singdot{0.6,0};
\end{tikzpicture}\right) = \frac{1}{\sqrt{-2}} \
\begin{tikzpicture}[centerzero, thick]
      \draw[-] (0.6,-0.4) -- (0.6,0.4);
      \singdot{0.6,-0.15};
      \bluetoken{0.6,0.15};
\end{tikzpicture} \ ,
\quad 
\mathcal{F} \left(\begin{tikzpicture}[centerzero, thick]
      \draw[-] (0.9,-0.4) -- (0.3,0.4);
      \draw[-] (0.3,-0.4) -- (0.9,0.4);
\end{tikzpicture}\right) 
= - \frac{1}{\sqrt{-2}} \
\begin{tikzpicture}[centerzero, thick]
      \draw[-] (0.9,-0.4) -- (0.3,0.4);
      \draw[-] (0.3,-0.4) -- (0.9,0.4);
      \teleport{0.45,0.2}{0.75,0.2};
\end{tikzpicture} \ ,
\\ \mathcal{F} \left(\begin{tikzpicture}[centerzero, thick]
      \draw[-] (0.9,-0.4) -- (0.3,0.4);
      \draw[-,wei] (0.3,-0.4) -- (0.9,0.4);
      \node at (0.3,-0.6) {\tiny $ Q $};
\end{tikzpicture}\right) 
=
\begin{tikzpicture}[centerzero, thick]
      \draw[-] (0.9,-0.4) -- (0.3,0.4);
      \draw[-,wei] (0.3,-0.4) -- (0.9,0.4);
      \node at (0.3,-0.6) {\tiny $ \psi(Q) $};
\end{tikzpicture} \ ,
\quad 
\mathcal{F}\left(\begin{tikzpicture}[centerzero, thick]
      \draw[-] (0.3,-0.4) -- (0.9,0.4);
      \draw[-,wei] (0.9,-0.4) -- (0.3,0.4);
      \node at (0.9,-0.6) {\tiny $ Q $};
\end{tikzpicture}\right) 
= 
\begin{tikzpicture}[centerzero, thick]
      \draw[-] (0.3,-0.4) -- (0.9,0.4);
      \draw[-,wei] (0.9,-0.4) -- (0.3,0.4);
      \node at (0.9,-0.6) {\tiny $ \psi(Q) $};
\end{tikzpicture} \ .
\end{gather*}
Its inverse $ \mathcal{G} \colon \mathcal{T\!AH}(\operatorname{Cl}) \rightarrow \mathcal{T\!AS}(\operatorname{Cl}) $ sends $ \go $ to $ \go $ and $ Q \in \mathcal{E}_{\operatorname{Pol}_{1}(\operatorname{Cl})} $ to $ \psi^{-1}(Q) \in \mathcal{E}_{\operatorname{OPol}_{1}} $, and sends the generating morphisms as follows:
\begin{gather*}
\mathcal{G} \left(\begin{tikzpicture}[centerzero, thick]
      \draw[-] (0.6,-0.4) -- (0.6,0.4);
      \bluetoken{0.6,0};
\end{tikzpicture}\right) = \begin{tikzpicture}[centerzero, thick]
      \draw[-] (0.6,-0.4) -- (0.6,0.4);
      \xtoken{0.6,0};
\end{tikzpicture} \ , 
\quad 
\mathcal{G} \left(\begin{tikzpicture}[centerzero, thick]
      \draw[-] (0.6,-0.4) -- (0.6,0.4);
      \singdot{0.6,0};
\end{tikzpicture}\right) = \sqrt{-2} \
\begin{tikzpicture}[centerzero, thick]
      \draw[-] (0.6,-0.4) -- (0.6,0.4);
      \singdot{0.6,-0.15};
      \xtoken{0.6,0.15};
\end{tikzpicture} \ ,
\quad 
\mathcal{G} \left(\begin{tikzpicture}[centerzero, thick]
      \draw[-] (0.9,-0.4) -- (0.3,0.4);
      \draw[-] (0.3,-0.4) -- (0.9,0.4);
\end{tikzpicture}\right) 
= \frac{1}{\sqrt{-2}} \
\begin{tikzpicture}[centerzero, thick]
      \draw[-] (0.9,-0.4) -- (0.3,0.4);
      \draw[-] (0.3,-0.4) -- (0.9,0.4);
      \xteleport{0.45,0.2}{0.75,0.2};
\end{tikzpicture} \ ,
\\ \mathcal{G} \left(\begin{tikzpicture}[centerzero, thick]
      \draw[-] (0.9,-0.4) -- (0.3,0.4);
      \draw[-,wei] (0.3,-0.4) -- (0.9,0.4);
      \node at (0.3,-0.6) {\tiny $ Q $};
\end{tikzpicture}\right) 
=
\begin{tikzpicture}[centerzero, thick]
      \draw[-] (0.9,-0.4) -- (0.3,0.4);
      \draw[-,wei] (0.3,-0.4) -- (0.9,0.4);
      \node at (0.3,-0.6) {\tiny $ \psi^{-1}(Q) $};
\end{tikzpicture} \ ,
\quad 
\mathcal{G} \left(\begin{tikzpicture}[centerzero, thick]
      \draw[-] (0.3,-0.4) -- (0.9,0.4);
      \draw[-,wei] (0.9,-0.4) -- (0.3,0.4);
      \node at (0.9,-0.6) {\tiny $ Q $};
\end{tikzpicture}\right) 
= 
\begin{tikzpicture}[centerzero, thick]
      \draw[-] (0.3,-0.4) -- (0.9,0.4);
      \draw[-,wei] (0.9,-0.4) -- (0.3,0.4);
      \node at (0.9,-0.6) {\tiny $ \psi^{-1}(Q) $};
\end{tikzpicture} \ .
\end{gather*}
\end{theo}

The rest of this subsection is dedicated to proving this theorem. We will prove the theorem by checking the defining relations of $ \mathcal{T\!AS}(\operatorname{Cl}) $ and $ \mathcal{T\!AH}(\operatorname{Cl}) $.

\begin{lem} \label{drip}
If $ Q \in \mathcal{E}_{\operatorname{OPol}_{1}} $, then
\begin{align} \label{wheel_77}
\mathcal{F}\left(\begin{tikzpicture}[centerzero, thick]
        \pin{-0.2,0}{-1,0}{f};
        \draw[-] (-0.2,-0.5) -- (-0.2,0.5);
        \draw[-, wei] (0.2,-0.5) -- (0.2,0.5);
        \node at (0.2,-0.7) {\tiny $ Q $};
\end{tikzpicture}\right)
&= 
\begin{tikzpicture}[centerzero, thick]
        \pin{-0.2,0}{-1,0}{\psi(f)};
        \draw[-] (-0.2,-0.5) -- (-0.2,0.5);
        \draw[-, wei] (0.2,-0.5) -- (0.2,0.5);
        \node at (0.2,-0.7) {\tiny $ \psi(Q) $};
\end{tikzpicture} , \quad f \in \operatorname{OPol}_{1}(\operatorname{Cl}),
\\ \mathcal{F}\left(\begin{tikzpicture}[centerzero, thick]
        \node at (0,-0.7) {\tiny $ Q $};
        \draw[-] (-0.4,-0.5) -- (-0.4,0.5);
        \draw[-, wei] (0,-0.5) -- (0,0.5);
        \draw[-] (0.4,-0.5) -- (0.4,0.5);
        \pinpin{0.4,0}{-0.4,0}{1.2,0}{f};
\end{tikzpicture}\right)
&= 
\begin{tikzpicture}[centerzero, thick]
        \node at (0,-0.7) {\tiny $ \psi(Q) $};
        \draw[-] (-0.4,-0.5) -- (-0.4,0.5);
        \draw[-, wei] (0,-0.5) -- (0,0.5);
        \draw[-] (0.4,-0.5) -- (0.4,0.5);
        \pinpin{0.4,0}{-0.4,0}{1.4,0}{\psi_{2}(f)};
\end{tikzpicture} \ , \qquad f \in \operatorname{OPol}_{2}(\operatorname{Cl}). \label{white_paper_79}
\end{align}
\end{lem}

\begin{proof}
We only explain how one obtains \cref{wheel_77}, since \cref{white_paper_79} is similar. It is straightforward to check that $ \mathcal{F} $ satisfies the relation \cref{spinCl123}. This fact together with Proposition \ref{grassroots} implies that we have a superalgebra homomorphism
\begin{equation*}
\operatorname{OPol}_{1}(\operatorname{Cl}) \rightarrow \operatorname{End}_{\mathcal{T\!AH}(\operatorname{Cl})}(\go \otimes \psi(Q)), \quad f \mapsto \mathcal{F}\left(\begin{tikzpicture}[centerzero, thick]
        \pin{-0.2,0}{-1,0}{f};
        \draw[-] (-0.2,-0.5) -- (-0.2,0.5);
        \draw[-, wei] (0.2,-0.5) -- (0.2,0.5);
        \node at (0.2,-0.7) {\tiny $ Q $};
\end{tikzpicture}\right).
\end{equation*}
In particular, it suffices to show that \cref{wheel_77} holds for $ f = c $ and $ f = x $. These are straightforward to see.
\end{proof}

\details{We have
\begin{gather*}
\mathcal{F}\left(\begin{tikzpicture}[centerzero, thick]
        \pin{-0.2,0}{-1,0}{x};
        \draw[-] (-0.2,-0.5) -- (-0.2,0.5);
        \draw[-, wei] (0.2,-0.5) -- (0.2,0.5);
        \node at (0.2,-0.7) {\tiny $ Q $};
\end{tikzpicture}\right)
=
\mathcal{F}\left(\begin{tikzpicture}[centerzero, thick]
        \draw[-] (-0.2,-0.5) -- (-0.2,0.5);
        \draw[-, wei] (0.2,-0.5) -- (0.2,0.5);
        \node at (0.2,-0.7) {\tiny $ Q $};
        \singdot{-0.2,0};
\end{tikzpicture}\right)
= \frac{1}{\sqrt{-2}} \ 
\begin{tikzpicture}[centerzero, thick]
        \draw[-] (-0.2,-0.5) -- (-0.2,0.5);
        \draw[-, wei] (0.2,-0.5) -- (0.2,0.5);
        \node at (0.2,-0.7) {\tiny $ \psi(Q) $};
        \bluetoken{-0.2,0.2};
        \singdot{-0.2,-0.2};
\end{tikzpicture}
\quad \stackrel{\mathclap{\cref{joker_1}}}{=} \quad
\begin{tikzpicture}[centerzero, thick]
        \pin{-0.2,0}{-1,0}{\psi(x)};
        \draw[-] (-0.2,-0.5) -- (-0.2,0.5);
        \draw[-, wei] (0.2,-0.5) -- (0.2,0.5);
        \node at (0.2,-0.7) {\tiny $ \psi(Q) $};
\end{tikzpicture} \ ,
\\ \mathcal{F}\left(\begin{tikzpicture}[centerzero, thick]
        \pin{-0.2,0}{-1,0}{c};
        \draw[-] (-0.2,-0.5) -- (-0.2,0.5);
        \draw[-, wei] (0.2,-0.5) -- (0.2,0.5);
        \node at (0.2,-0.7) {\tiny $ Q $};
\end{tikzpicture}\right)
=
\mathcal{F}\left(\begin{tikzpicture}[centerzero, thick]
        \draw[-] (-0.2,-0.5) -- (-0.2,0.5);
        \draw[-, wei] (0.2,-0.5) -- (0.2,0.5);
        \node at (0.2,-0.7) {\tiny $ Q $};
        \xtoken{-0.2,0};
\end{tikzpicture}\right)
= \ 
\begin{tikzpicture}[centerzero, thick]
        \draw[-] (-0.2,-0.5) -- (-0.2,0.5);
        \draw[-, wei] (0.2,-0.5) -- (0.2,0.5);
        \node at (0.2,-0.7) {\tiny $ \psi(Q) $};
        \bluetoken{-0.2,0};
\end{tikzpicture}
\quad \stackrel{\mathclap{\cref{joker_1}}}{=} \quad
\begin{tikzpicture}[centerzero, thick]
        \pin{-0.2,0}{-1,0}{\psi(c)};
        \draw[-] (-0.2,-0.5) -- (-0.2,0.5);
        \draw[-, wei] (0.2,-0.5) -- (0.2,0.5);
        \node at (0.2,-0.7) {\tiny $ \psi(Q) $};
\end{tikzpicture} \ .
\end{gather*}}

\begin{lem} \label{Yoneda_346}
If $ Q \in \mathcal{E}_{\operatorname{OPol}_{1}} $, then 
\begin{align} \label{Yoneda_0.5}
\mathcal{F}\left(\begin{tikzpicture}[centerzero, thick]
        \draw[-] (-0.2,-0.5) to[out=up,in=down] (0.2,0) to[out=up,in=down] (-0.2,0.5);
        \draw[-, wei] (0.2,-0.5) to[out=up,in=down] (-0.2,0) to[out=up,in=down] (0.2,0.5);
        \node at (0.2,-0.7) {\tiny $ Q $};
\end{tikzpicture}\right)
=
\mathcal{F}\left(\begin{tikzpicture}[centerzero, thick]
        \pin{-0.2,0}{-1,0}{Q};
        \draw[-] (-0.2,-0.5) -- (-0.2,0.5);
        \draw[-, wei] (0.2,-0.5) -- (0.2,0.5);
        \node at (0.2,-0.7) {\tiny $ Q $};
\end{tikzpicture}\right),
\qquad
\mathcal{F}\left(\begin{tikzpicture}[centerzero, thick]
        \draw[-] (0.2,-0.5) to[out=up,in=down] (-0.2,0) to[out=up,in=down] (0.2,0.5);
        \draw[-, wei] (-0.2,-0.5) to[out=up,in=down] (0.2,0) to[out=up,in=down] (-0.2,0.5);
        \node at (-0.2,-0.7) {\tiny $ Q $};
\end{tikzpicture}\right)
\ = \
\mathcal{F}\left(\begin{tikzpicture}[centerzero, thick]
        \draw[-, wei] (-0.2,-0.5) -- (-0.2,0.5);
        \draw[-] (0.2,-0.5) -- (0.2,0.5);
        \node at (-0.2,-0.7) {\tiny $ Q $};
        \pin{0.2,0}{1,0}{Q};
\end{tikzpicture}\right). 
\end{align}
\end{lem}

\begin{proof}
We compute that
\begin{equation*}
\mathcal{F}\left(\begin{tikzpicture}[centerzero, thick]
        \draw[-] (-0.2,-0.5) to[out=up,in=down] (0.2,0) to[out=up,in=down] (-0.2,0.5);
        \draw[-, wei] (0.2,-0.5) to[out=up,in=down] (-0.2,0) to[out=up,in=down] (0.2,0.5);
        \node at (0.2,-0.7) {\tiny $ Q $};
\end{tikzpicture}\right)
= \begin{tikzpicture}[centerzero, thick]
        \draw[-] (-0.2,-0.5) to[out=up,in=down] (0.2,0) to[out=up,in=down] (-0.2,0.5);
        \draw[-, wei] (0.2,-0.5) to[out=up,in=down] (-0.2,0) to[out=up,in=down] (0.2,0.5);
        \node at (0.2,-0.7) {\tiny $ \psi(Q) $};
\end{tikzpicture}
 \stackrel{\mathclap{\cref{Cliff_hw3}}}{=} \
\begin{tikzpicture}[centerzero, thick]
        \pin{-0.2,0}{-1,0}{\psi(Q)};
        \draw[-] (-0.2,-0.5) -- (-0.2,0.5);
        \draw[-, wei] (0.2,-0.5) -- (0.2,0.5);
        \node at (0.2,-0.7) {\tiny $ \psi(Q) $};
\end{tikzpicture} 
\ \stackrel{\mathclap{\cref{wheel_77}}}{=} \
\mathcal{F}\left(\begin{tikzpicture}[centerzero, thick]
        \pin{-0.2,0}{-1,0}{Q};
        \draw[-] (-0.2,-0.5) -- (-0.2,0.5);
        \draw[-, wei] (0.2,-0.5) -- (0.2,0.5);
        \node at (0.2,-0.7) {\tiny $ Q $};
\end{tikzpicture}\right). 
\end{equation*}
The proof of the second equation in \cref{Yoneda_0.5} is similar.
\end{proof}

\begin{lem}
If $ Q \in \mathcal{E}_{\operatorname{OPol}_{1}} $, then 
\begin{equation} \label{Toronto}
\mathcal{F}\left(\begin{tikzpicture}[centerzero, thick]
        \draw[-] (0.4,-0.5) -- (-0.4,0.5);
        \draw[-] (-0.4,-0.5) -- (0.4,0.5);
        \node at (0,-0.7) {\tiny $ Q $};
        \draw[-, wei] (0,-0.5) to[out=up, in=down] (-0.4,0) to[out=up,in=down] (0,0.5);
\end{tikzpicture}\right)
= 
\mathcal{F}\left(\begin{tikzpicture}[centerzero, thick]
        \draw[-] (2,-0.5) -- (1.2,0.5);
        \draw[-] (1.2,-0.5) -- (2,0.5);
        \node at (1.6,-0.7) {\tiny $ Q $};
        \draw[-, wei] (1.6,-0.5) to[out=up, in=down] (2,0) to[out=up,in=down] (1.6,0.5);
\end{tikzpicture}
\ + \
\begin{tikzpicture}[centerzero, thick]
        \node at (0,-0.7) {\tiny $ Q $};
        \draw[-] (-0.4,-0.5) -- (-0.4,0.5);
        \draw[-, wei] (0,-0.5) -- (0,0.5);
        \draw[-] (0.4,-0.5) -- (0.4,0.5);
        \pinpin{0.4,0}{-0.4,0}{1.5,0}{D(Q_{1})};
\end{tikzpicture}\right).
\end{equation}
\end{lem}

\begin{proof}
We compute that
\begin{equation*}
\mathcal{F}\left(\begin{tikzpicture}[centerzero, thick]
        \draw[-] (0.4,-0.5) -- (-0.4,0.5);
        \draw[-] (-0.4,-0.5) -- (0.4,0.5);
        \node at (0,-0.7) {\tiny $ Q $};
        \draw[-, wei] (0,-0.5) to[out=up, in=down] (-0.4,0) to[out=up,in=down] (0,0.5);
\end{tikzpicture}\right)
= -\frac{1}{\sqrt{-2}} \
\begin{tikzpicture}[centerzero, thick]
        \draw[-] (0.4,-0.5) -- (-0.4,0.5);
        \draw[-] (-0.4,-0.5) -- (0.4,0.5);
        \node at (0,-0.7) {\tiny $ \psi(Q) $};
        \draw[-, wei] (0,-0.5) to[out=up, in=down] (-0.4,0) to[out=up,in=down] (0,0.5);
        \teleport{-0.13,0.13}{0.13,0.13};
\end{tikzpicture}
\ \overset{\mathclap{\cref{Cliff_hw1}}}{\underset{\mathclap{\cref{Cliff_green}}}{=}} \ 
- \frac{1}{\sqrt{-2}} \
\begin{tikzpicture}[centerzero, thick]
        \draw[-] (2,-0.5) -- (1.2,0.5);
        \draw[-] (1.2,-0.5) -- (2,0.5);
        \node at (1.6,-0.7) {\tiny $ \psi(Q) $};
        \draw[-, wei] (1.6,-0.5) to[out=up, in=down] (2,0) to[out=up,in=down] (1.6,0.5);
        \teleport{1.47,0.13}{1.73,0.13};
\end{tikzpicture}
+ \frac{1}{\sqrt{-2}} \
\begin{tikzpicture}[centerzero, thick]
        \node at (0,-0.7) {\tiny $ \psi(Q) $};
        \draw[-] (-0.4,-0.5) -- (-0.4,0.5);
        \draw[-, wei] (0,-0.5) -- (0,0.5);
        \draw[-] (0.4,-0.5) -- (0.4,0.5);
        \pinpin{0.4,-0.2}{-0.4,-0.2}{1.9,-0.2}{\partial(\psi(Q)_{1})};
        \teleport{-0.4,0.2}{0.4,0.2};
\end{tikzpicture}
\end{equation*}
and
\begin{equation*}
\mathcal{F}\left(\begin{tikzpicture}[centerzero, thick]
        \draw[-] (2,-0.5) -- (1.2,0.5);
        \draw[-] (1.2,-0.5) -- (2,0.5);
        \node at (1.6,-0.7) {\tiny $ Q $};
        \draw[-, wei] (1.6,-0.5) to[out=up, in=down] (2,0) to[out=up,in=down] (1.6,0.5);
\end{tikzpicture}
\ + \
\begin{tikzpicture}[centerzero, thick]
        \node at (0,-0.7) {\tiny $ Q $};
        \draw[-] (-0.4,-0.5) -- (-0.4,0.5);
        \draw[-, wei] (0,-0.5) -- (0,0.5);
        \draw[-] (0.4,-0.5) -- (0.4,0.5);
        \pinpin{0.4,0}{-0.4,0}{1.5,0}{D(Q_{1})};
\end{tikzpicture}\right)
= 
-\frac{1}{\sqrt{-2}} \
\begin{tikzpicture}[centerzero, thick]
        \draw[-] (2,-0.5) -- (1.2,0.5);
        \draw[-] (1.2,-0.5) -- (2,0.5);
        \node at (1.6,-0.7) {\tiny $ \psi(Q) $};
        \draw[-, wei] (1.6,-0.5) to[out=up, in=down] (2,0) to[out=up,in=down] (1.6,0.5);
        \teleport{1.47,0.13}{1.73,0.13};
\end{tikzpicture}
+ \mathcal{F}\left(\begin{tikzpicture}[centerzero, thick]
        \node at (0,-0.7) {\tiny $ Q $};
        \draw[-] (-0.4,-0.5) -- (-0.4,0.5);
        \draw[-, wei] (0,-0.5) -- (0,0.5);
        \draw[-] (0.4,-0.5) -- (0.4,0.5);
        \pinpin{0.4,0}{-0.4,0}{1.5,0}{D(Q_{1})};
\end{tikzpicture}\right).
\end{equation*}
Furthermore,
\begin{equation*}
\mathcal{F}\left(\begin{tikzpicture}[centerzero, thick]
        \node at (0,-0.7) {\tiny $ Q $};
        \draw[-] (-0.4,-0.5) -- (-0.4,0.5);
        \draw[-, wei] (0,-0.5) -- (0,0.5);
        \draw[-] (0.4,-0.5) -- (0.4,0.5);
        \pinpin{0.4,0}{-0.4,0}{1.5,0}{D(Q_{1})};
\end{tikzpicture}\right)
\ \stackrel{\mathclap{\cref{white_paper_79}}}{=} \
\begin{tikzpicture}[centerzero, thick]
        \node at (0,-0.7) {\tiny $ \psi(Q) $};
        \draw[-] (-0.4,-0.5) -- (-0.4,0.5);
        \draw[-, wei] (0,-0.5) -- (0,0.5);
        \draw[-] (0.4,-0.5) -- (0.4,0.5);
        \pinpin{0.4,0}{-0.4,0}{1.8,0}{\psi_{2}(D(Q_{1}))};
\end{tikzpicture}
\ \overset{\mathclap{\cref{fan_switch}}}{\underset{\mathclap{\cref{seeing}}}{=}} \ \frac{1}{\sqrt{-2}} \
\begin{tikzpicture}[centerzero, thick]
        \node at (0,-0.7) {\tiny $ \psi(Q) $};
        \draw[-] (-0.4,-0.5) -- (-0.4,0.5);
        \draw[-, wei] (0,-0.5) -- (0,0.5);
        \draw[-] (0.4,-0.5) -- (0.4,0.5);
        \pinpin{0.4,-0.2}{-0.4,-0.2}{1.9,-0.2}{\partial(\psi(Q)_{1})};
        \teleport{-0.4,0.2}{0.4,0.2};
\end{tikzpicture} \ ,
\end{equation*}
and so the result follows. 
\end{proof}

\begin{lem}
If $ Q \in \mathcal{E}_{\operatorname{Pol}_{1}(\operatorname{Cl})} $, then 
\begin{equation} \label{coffee}
\mathcal{G}\left(\begin{tikzpicture}[centerzero, thick]
        \draw[-] (0.4,-0.5) -- (-0.4,0.5);
        \draw[-] (-0.4,-0.5) -- (0.4,0.5);
        \node at (0,-0.7) {\tiny $ Q $};
        \draw[-, wei] (0,-0.5) to[out=up, in=down] (-0.4,0) to[out=up,in=down] (0,0.5);
\end{tikzpicture}\right)
=
\mathcal{G}\left(\begin{tikzpicture}[centerzero, thick]
        \draw[-] (2,-0.5) -- (1.2,0.5);
        \draw[-] (1.2,-0.5) -- (2,0.5);
        \node at (1.6,-0.7) {\tiny $ Q $};
        \draw[-, wei] (1.6,-0.5) to[out=up, in=down] (2,0) to[out=up,in=down] (1.6,0.5);
\end{tikzpicture}
\ - \
\begin{tikzpicture}[centerzero, thick]
        \node at (0,-0.7) {\tiny $ Q $};
        \draw[-] (-0.4,-0.5) -- (-0.4,0.5);
        \draw[-, wei] (0,-0.5) -- (0,0.5);
        \draw[-] (0.4,-0.5) -- (0.4,0.5);
        \pinpin{0.4,0}{-0.4,0}{1.5,0}{\partial(Q_{1})};
\end{tikzpicture}\right).
\end{equation}
\end{lem}

\begin{proof}
We compute that
\begin{equation*}
\mathcal{G}\left(\begin{tikzpicture}[centerzero, thick]
        \draw[-] (0.4,-0.5) -- (-0.4,0.5);
        \draw[-] (-0.4,-0.5) -- (0.4,0.5);
        \node at (0,-0.7) {\tiny $ Q $};
        \draw[-, wei] (0,-0.5) to[out=up, in=down] (-0.4,0) to[out=up,in=down] (0,0.5);
\end{tikzpicture}\right)
= \frac{1}{\sqrt{-2}} \
\begin{tikzpicture}[centerzero, thick]
        \draw[-] (0.4,-0.5) -- (-0.4,0.5);
        \draw[-] (-0.4,-0.5) -- (0.4,0.5);
        \node at (0,-0.7) {\tiny $ \psi^{-1}(Q) $};
        \draw[-, wei] (0,-0.5) to[out=up, in=down] (-0.4,0) to[out=up,in=down] (0,0.5);
        \xteleport{-0.13,0.13}{0.13,0.13};
\end{tikzpicture}
\ \overset{\mathclap{\cref{spin8}}}{\underset{\mathclap{\cref{spinCl2}}}{=}} \
\frac{1}{\sqrt{-2}} 
\begin{tikzpicture}[centerzero, thick]
        \draw[-] (2,-0.5) -- (1.2,0.5);
        \draw[-] (1.2,-0.5) -- (2,0.5);
        \node at (1.6,-0.7) {\tiny $ \psi^{-1}(Q) $};
        \draw[-, wei] (1.6,-0.5) to[out=up, in=down] (2,0) to[out=up,in=down] (1.6,0.5);
        \xteleport{1.47,0.13}{1.73,0.13};
\end{tikzpicture}
+ \frac{1}{\sqrt{-2}} 
\begin{tikzpicture}[centerzero, thick]
        \node at (0,-0.7) {\tiny $ \psi^{-1}(Q) $};
        \draw[-] (-0.4,-0.5) -- (-0.4,0.5);
        \draw[-, wei] (0,-0.5) -- (0,0.5);
        \draw[-] (0.4,-0.5) -- (0.4,0.5);
        \pinpin{0.4,-0.2}{-0.4,-0.2}{1.9,-0.2}{D(\psi^{-1}(Q)_{1})};
        \xteleport{-0.4,0.2}{0.4,0.2};
\end{tikzpicture}
\end{equation*}
and
\begin{equation*}
\mathcal{G}\left(\begin{tikzpicture}[centerzero, thick]
        \draw[-] (2,-0.5) -- (1.2,0.5);
        \draw[-] (1.2,-0.5) -- (2,0.5);
        \node at (1.6,-0.7) {\tiny $ Q $};
        \draw[-, wei] (1.6,-0.5) to[out=up, in=down] (2,0) to[out=up,in=down] (1.6,0.5);
\end{tikzpicture}
\ - \
\begin{tikzpicture}[centerzero, thick]
        \node at (0,-0.7) {\tiny $ Q $};
        \draw[-] (-0.4,-0.5) -- (-0.4,0.5);
        \draw[-, wei] (0,-0.5) -- (0,0.5);
        \draw[-] (0.4,-0.5) -- (0.4,0.5);
        \pinpin{0.4,0}{-0.4,0}{1.5,0}{\partial(Q_{1})};
\end{tikzpicture}\right)
= 
\frac{1}{\sqrt{-2}} 
\begin{tikzpicture}[centerzero, thick]
        \draw[-] (2,-0.5) -- (1.2,0.5);
        \draw[-] (1.2,-0.5) -- (2,0.5);
        \node at (1.6,-0.7) {\tiny $ \psi^{-1}(Q) $};
        \draw[-, wei] (1.6,-0.5) to[out=up, in=down] (2,0) to[out=up,in=down] (1.6,0.5);
        \xteleport{1.47,0.13}{1.73,0.13};
\end{tikzpicture}
- \mathcal{G}\left(\begin{tikzpicture}[centerzero, thick]
        \node at (0,-0.7) {\tiny $ Q $};
        \draw[-] (-0.4,-0.5) -- (-0.4,0.5);
        \draw[-, wei] (0,-0.5) -- (0,0.5);
        \draw[-] (0.4,-0.5) -- (0.4,0.5);
        \pinpin{0.4,0}{-0.4,0}{1.5,0}{\partial(Q_{1})};
\end{tikzpicture}\right).
\end{equation*}
Furthermore, we have
\begin{equation*} 
\mathcal{G}\left(\begin{tikzpicture}[centerzero, thick]
        \node at (0,-0.7) {\tiny $ Q $};
        \draw[-] (-0.4,-0.5) -- (-0.4,0.5);
        \draw[-, wei] (0,-0.5) -- (0,0.5);
        \draw[-] (0.4,-0.5) -- (0.4,0.5);
        \pinpin{0.4,0}{-0.4,0}{1.5,0}{\partial(Q_{1})};
\end{tikzpicture}\right)
\ = \ 
\begin{tikzpicture}[centerzero, thick]
        \node at (0,-0.7) {\tiny $ \psi^{-1}(Q) $};
        \draw[-] (-0.4,-0.5) -- (-0.4,0.5);
        \draw[-, wei] (0,-0.5) -- (0,0.5);
        \draw[-] (0.4,-0.5) -- (0.4,0.5);
        \pinpin{0.4,0}{-0.4,0}{2,0}{\psi_{2}^{-1}(\partial(Q_{1}))};
\end{tikzpicture}
\ \overset{\mathclap{\cref{fan_switch_2}}}{\underset{\mathclap{\cref{working_655}}}{=}} \ 
-\frac{1}{\sqrt{-2}} 
\begin{tikzpicture}[centerzero, thick]
        \node at (0,-0.7) {\tiny $ \psi^{-1}(Q) $};
        \draw[-] (-0.4,-0.5) -- (-0.4,0.5);
        \draw[-, wei] (0,-0.5) -- (0,0.5);
        \draw[-] (0.4,-0.5) -- (0.4,0.5);
        \pinpin{0.4,-0.2}{-0.4,-0.2}{1.9,-0.2}{D(\psi^{-1}(Q)_{1})};
        \xteleport{-0.4,0.2}{0.4,0.2};
\end{tikzpicture} \ ,
\end{equation*}
and so the equation \cref{coffee} follows.
\end{proof}

\details{Note that, in the first equality of final equation in the above proof, we used the fact that
\begin{equation*}
\mathcal{G}\left(\begin{tikzpicture}[centerzero, thick]
        \node at (0,-0.7) {\tiny $ Q $};
        \draw[-] (-0.4,-0.5) -- (-0.4,0.5);
        \draw[-, wei] (0,-0.5) -- (0,0.5);
        \draw[-] (0.4,-0.5) -- (0.4,0.5);
        \pinpin{0.4,0}{-0.4,0}{1.2,0}{f};
\end{tikzpicture}\right)
= 
\begin{tikzpicture}[centerzero, thick]
        \node at (0,-0.7) {\tiny $ \psi^{-1}(Q) $};
        \draw[-] (-0.4,-0.5) -- (-0.4,0.5);
        \draw[-, wei] (0,-0.5) -- (0,0.5);
        \draw[-] (0.4,-0.5) -- (0.4,0.5);
        \pinpin{0.4,0}{-0.4,0}{1.4,0}{\psi_{2}^{-1}(f)};
\end{tikzpicture} \ , \qquad f \in \operatorname{Pol}_{2}(\operatorname{Cl}),
\end{equation*}
which can be proven by applying a similar argument to Lemma~\ref{drip}. Furthermore, in the second equality, we implicitly used that $ \psi_{2}^{-1}(Q_{1}) \in \operatorname{OPol}_{2} $ (since, when applying \cref{working_655}, one requires that the assumptions of Lemma \ref{working_656} hold). This can be obtained from the fact that $ \psi^{-1}(Q) \in \operatorname{OPol}_{1} $ (see Lemma \ref{Psi(H)}), and the fact that $ \psi_{2}^{-1}(Q_{1}) = \psi^{-1}(Q)_{1} $ (see \cref{fan_switch_2}).}

\begin{proof}[Proof of Theorem~\ref{Morita_equiv2}]
Verifying the existence of $ \mathcal{F} $ requires checking the relations \cref{spin1}, \cref{spin2}, \cref{spin3,spin4,spin6,spin8}, \cref{spinCl123}, \cref{spinCl1}, and \cref{spinCl2}. We have that $ \mathcal{F} $ satisfies \cref{spin4} by \cref{Yoneda_0.5}, and $ \mathcal{F} $ also satisfies \cref{spin8} by \cref{Toronto}. Furthermore, $ \mathcal{F} $ satisfies the relations \cref{spin1}, \cref{spin2}, \cref{spinCl123}, and \cref{spinCl1} by Proposition~\ref{AS-AH-iso}. We omit the proofs of the remaining relations, since these are straightforward.
\\ \indent Next, verifying the existence of $ \mathcal{G} $ requires checking the relations \cref{Pol(Cl)}, \cref{Cliff_hw-1}, \cref{Cliff_hw0}, and \cref{Cliff_hw1,Cliff_hw2,Cliff_hw3,Cliff_hw5,Cliff_green}. We have that $ \mathcal{G} $ satisfies \cref{Cliff_green} by \cref{coffee}, and one can prove using a similar argument to Lemma~\ref{Yoneda_346} that $ \mathcal{G} $ satisfies \cref{Cliff_hw3}. Furthermore, $ \mathcal{G} $ satisfies \cref{Pol(Cl)}, \cref{Cliff_hw-1} and \cref{Cliff_hw0} by Proposition~\ref{AS-AH-iso}. We omit the proof of the remaining relations, since these are straightforward. 
\\ \indent Finally, it is straightforward to verify that $ \mathcal{F} $ and $ \mathcal{G} $ are inverses to each other.
\end{proof}

\subsection{Morita superequivalences}

Restricting to objects, the functor $ \mathcal{F} $ induces a monoid isomorphism $ \mathcal{F} \colon F(\widehat{\mathcal{E}}_{\operatorname{OPol}_{1}}) \rightarrow F(\widehat{\mathcal{E}}_{\operatorname{Pol}_{1}(\operatorname{Cl})}) $. This isomorphism satisfies $ \mathcal{F}(\mathbf{Q}) \in F(\mathcal{E}_{\operatorname{Pol}_{1}(\operatorname{Cl})}) $ for all $ \mathbf{Q} \in F(\mathcal{E}_{\operatorname{OPol}_{1}}) $.

\begin{lem}
Let $ d \in \mathbb{N}_{+} $ and $ \mathbf{Q} \in F(\mathcal{E}_{\operatorname{OPol}_{1}}) $. Then 
\begin{gather} \label{warm_weather}
\mathcal{F}(\Gamma_{d,\mathbf{Q}}) = \Lambda_{d,\mathcal{F}(\mathbf{Q})}.
\end{gather}
\end{lem}

\begin{proof}
We compute that 
\begin{gather*} 
\mathcal{F}(\Gamma_{d,\mathbf{Q}}) \ \stackrel{\mathclap{\cref{fluffy_5678}}}{=} \ \mathcal{F}(F(\widehat{\mathcal{E}}_{\operatorname{OPol}_{1}})_{\go^{d},\mathbf{Q}}) = F(\widehat{\mathcal{E}}_{\operatorname{Pol}_{1}(\operatorname{Cl})})_{\mathcal{F}(\go^{d}),\mathcal{F}(\mathbf{Q})} = F(\widehat{\mathcal{E}}_{\operatorname{Pol}_{1}(\operatorname{Cl})})_{\go^{d},\mathcal{F}(\mathbf{Q})} \ \stackrel{\mathclap{\cref{Lambda_def}}}{=} \ \Lambda_{d,\mathcal{F}(\mathbf{Q})}.
\end{gather*}
Here, the second equality follows from Lemma \ref{pencil_sharp_2}, and the third equality follows from the fact that $ \mathcal{F} $ fixes $ \go $.
\end{proof}

\begin{theo} \label{Morita_result}
Let $ d \in \mathbb{N} $ and $ \mathbf{Q} \in F(\mathcal{E}_{\operatorname{OPol}_{1}}) $. Then we have superalgebra isomorphisms
\begin{align}
H_{d,\mathcal{F}(\mathbf{Q})}^{\operatorname{aff}}(\operatorname{Cl}) \cong \operatorname{SH}^{\operatorname{aff}}_{d,\mathbf{Q}}(\operatorname{Cl}) \cong \operatorname{Cl}_{d} \otimes \operatorname{SH}^{\operatorname{aff}}_{d,\mathbf{Q}}, \label{isomorphism_123}
\\ H_{d,\mathcal{F}(\mathbf{Q})}^{\operatorname{cyc}}(\operatorname{Cl}) \cong \operatorname{SH}^{\operatorname{cyc}}_{d,\mathbf{Q}}(\operatorname{Cl}) \cong \operatorname{Cl}_{d} \otimes \operatorname{SH}^{\operatorname{cyc}}_{d,\mathbf{Q}}. \label{isomorphism_223}
\end{align}
In particular, the $ (d,\mathbf{Q}) $-degenerate spin affine Hecke superalgebra $ \operatorname{SH}^{\operatorname{aff}}_{d,\mathbf{Q}} $ is Morita superequivalent to the $ (d,\mathcal{F}(\mathbf{Q})) $-degenerate affine Hecke--Clifford superalgebra $ H_{d,\mathcal{F}(\mathbf{Q})}^{\operatorname{aff}}(\operatorname{Cl}) $, and the cyclotomic $ (d,\mathbf{Q}) $-spin Hecke superalgebra $ \operatorname{SH}^{\operatorname{cyc}}_{d,\mathbf{Q}} $ is Morita superequivalent to the cyclotomic $ (d,\mathcal{F}(\mathbf{Q})) $-Hecke--Clifford superalgebra $ H_{d,\mathcal{F}(\mathbf{Q})}^{\operatorname{cyc}}(\operatorname{Cl}) $.
\end{theo}

\begin{proof}
The final assertion follows from \cref{isomorphism_123}, \cref{isomorphism_223}, and Corollary~\ref{baguette345}. The functor $ \mathcal{F} $  induces an isomorphism between the additive envelopes $ \operatorname{Add}(\mathcal{T\!AS}(\operatorname{Cl})) $ and $ \operatorname{Add}(\mathcal{T\!AH}(\operatorname{Cl})) $. This isomorphism in turn induces a superalgebra isomorphism 
\begin{multline*}
\operatorname{SH}^{\operatorname{aff}}_{d,\mathbf{Q}}(\operatorname{Cl}) \quad \stackrel{\mathclap{\cref{HLDSAHA22_Cliff}}}{=} \quad \operatorname{End}_{\operatorname{Add}(\mathcal{T\!AS}(\operatorname{Cl}))}\left(\bigoplus_{\mathbf{i} \in \Gamma_{d,\mathbf{Q}}} \mathbf{i} \right) \cong \operatorname{End}_{\operatorname{Add}(\mathcal{T\!AH}(\operatorname{Cl}))}\left(\bigoplus_{\mathbf{i} \in \Gamma_{d,\mathbf{Q}}} \mathcal{F}(\mathbf{i}) \right)
\\ \stackrel{\mathclap{\cref{warm_weather}}}{=} \quad \operatorname{End}_{\operatorname{Add}(\mathcal{T\!AH}(\operatorname{Cl}))}\left(\bigoplus_{\mathbf{j} \in \Lambda_{d,\mathcal{F}(\mathbf{Q})}} \mathbf{j}\right) \quad \stackrel{\mathclap{\cref{higher_Clifford_Hecke_def}}}{=} \quad H_{d,\mathcal{F}(\mathbf{Q})}^{\operatorname{aff}}(\operatorname{Cl}). 
\end{multline*}
Then this isomorphism together with \cref{isomorphism_1} yields \cref{isomorphism_123}. Next, suppose that $ \mathbf{i} \in \Gamma_{d,\mathbf{Q}} $ is a word whose first entry is equal to $ \go $. Then $ \mathcal{F}(1_{\mathbf{i}}) = 1_{\mathcal{F}(\mathbf{i})} $, where the 
first entry in the word $ \mathcal{F}(\mathbf{i}) \in \Lambda_{d,\mathcal{F}(\mathbf{Q})} $ is also equal to $ \go $. Thus the isomorphism $ H_{d,\mathcal{F}(\mathbf{Q})}^{\operatorname{aff}}(\operatorname{Cl}) \cong \operatorname{SH}^{\operatorname{aff}}_{d,\mathbf{Q}}(\operatorname{Cl}) $ induces an isomorphism $ H_{d,\mathcal{F}(\mathbf{Q})}^{\operatorname{cyc}}(\operatorname{Cl}) \cong \operatorname{SH}^{\operatorname{cyc}}_{d,\mathbf{Q}}(\operatorname{Cl}) $ on the cyclotomic quotients. This fact together with \cref{isomorphism_2} yields \cref{isomorphism_223}. 
\end{proof}

Specializing $ \mathbf{Q} $, we obtain the following two results. The first of these results is Wang's isomorphism \cite[Thm.~4.1]{Wang}.

\begin{cor}
If $ d \in \mathbb{N} $, then we have a superalgebra isomorphism
\begin{equation*}
H_{d}^{\operatorname{aff}}(\operatorname{Cl}) \cong \operatorname{Cl}_{d} \otimes \operatorname{SH}^{\operatorname{aff}}_{d}.
\end{equation*}
In particular, the degenerate spin affine Hecke superalgebra $ \operatorname{SH}^{\operatorname{aff}}_{d} $ is Morita superequivalent to the degenerate affine Hecke--Clifford superalgebra $ H_{d}^{\operatorname{aff}}(\operatorname{Cl}) $.
\end{cor}

\begin{proof}
This follows by taking $ \mathbf{Q} $ to be the empty word in Theorem~\ref{Morita_result}.
\end{proof}

\begin{cor} \label{rooftop}
Let $ d \in \mathbb{N}_{+} $ and $ Q \in \mathcal{E}_{\operatorname{OPol}_{1}} $. Then we have a superalgebra isomorphism
\begin{equation*}
H_{d}^{\psi(Q)}(\operatorname{Cl}) \cong \operatorname{Cl}_{d} \otimes \operatorname{SH}^{Q}_{d}.
\end{equation*}
In particular, the cyclotomic spin Hecke superalgebra $ \operatorname{SH}^{Q}_{d} $ is Morita superequivalent to the cyclotomic Hecke--Clifford superalgebra $ H_{d}^{\psi(Q)}(\operatorname{Cl}) $.
\end{cor}

\begin{proof}
This follows by taking $ \mathbf{Q} = Q $ in Theorem~\ref{Morita_result}, and then applying Proposition~\ref{frontlines} and Proposition~\ref{Cliff_cyclotomic_result}.
\end{proof}



\bibliographystyle{alphaurl}
\bibliography{references_4}

@article {Maksimau-Stroppel,
    AUTHOR = {Maksimau, Ruslan and Stroppel, Catharina},
     TITLE = {Higher level affine {S}chur and {H}ecke algebras},
   JOURNAL = {J. Pure Appl. Algebra},
  FJOURNAL = {Journal of Pure and Applied Algebra},
    VOLUME = {225},
      YEAR = {2021},
    NUMBER = {8},
     PAGES = {Paper No. 106442, 44},
      ISSN = {0022-4049,1873-1376},
   MRCLASS = {20C08 (18M30 20B30)},
  MRNUMBER = {4182993},
MRREVIEWER = {Mahir\ Bilen\ Can},
       DOI = {10.1016/j.jpaa.2020.106442},
       URL = {https://doi.org/10.1016/j.jpaa.2020.106442},
       ARCHIVEPREFIX = {arXiv},
       EPRINT = {1805.02425},
}

@article {Savage,
    AUTHOR = {Savage, Alistair},
     TITLE = {Affine wreath product algebras},
   JOURNAL = {Int. Math. Res. Not. IMRN},
  FJOURNAL = {International Mathematics Research Notices. IMRN},
      YEAR = {2020},
    NUMBER = {10},
     PAGES = {2977--3041},
      ISSN = {1073-7928,1687-0247},
   MRCLASS = {20C08 (16W55 18N25)},
  MRNUMBER = {4098632},
MRREVIEWER = {Shoumin\ Liu},
       DOI = {10.1093/imrn/rny092},
       URL = {https://doi.org/10.1093/imrn/rny092},
       ARCHIVEPREFIX = {arXiv},
       EPRINT = {1709.02998},
}

@article {Webster,
    AUTHOR = {Webster, Ben},
     TITLE = {Knot invariants and higher representation theory},
   JOURNAL = {Mem. Amer. Math. Soc.},
  FJOURNAL = {Memoirs of the American Mathematical Society},
    VOLUME = {250},
      YEAR = {2017},
    NUMBER = {1191},
     PAGES = {v+141},
      ISSN = {0065-9266,1947-6221},
      ISBN = {978-1-4704-2650-7; 978-1-4704-4206-4},
   MRCLASS = {57M27 (17B10 18D05 57M25)},
  MRNUMBER = {3709726},
MRREVIEWER = {Stefan\ K.\ Friedl},
       DOI = {10.1090/memo/1191},
       URL = {https://doi.org/10.1090/memo/1191},
       ARCHIVEPREFIX = {arXiv},
       EPRINT = {1309.3796},
}

@article {Wang,
    AUTHOR = {Wang, Weiqiang},
     TITLE = {Double affine {H}ecke algebras for the spin symmetric group},
   JOURNAL = {Math. Res. Lett.},
  FJOURNAL = {Mathematical Research Letters},
    VOLUME = {16},
      YEAR = {2009},
    NUMBER = {6},
     PAGES = {1071--1085},
      ISSN = {1073-2780},
   MRCLASS = {20C08},
  MRNUMBER = {2576694},
MRREVIEWER = {Maarten\ Sander\ Solleveld},
       DOI = {10.4310/MRL.2009.v16.n6.a14},
       URL = {https://doi.org/10.4310/MRL.2009.v16.n6.a14},
       ARCHIVEPREFIX = {arXiv},
       EPRINT = {math/0608074},
}

@article {Ellis-Khovanov-Lauda,
    AUTHOR = {Ellis, Alexander P. and Khovanov, Mikhail and Lauda, Aaron D.},
     TITLE = {The odd nil{H}ecke algebra and its diagrammatics},
   JOURNAL = {Int. Math. Res. Not. IMRN},
  FJOURNAL = {International Mathematics Research Notices. IMRN},
      YEAR = {2014},
    NUMBER = {4},
     PAGES = {991--1062},
      ISSN = {1073-7928,1687-0247},
   MRCLASS = {20C08},
  MRNUMBER = {3168401},
MRREVIEWER = {Anton\ Cox},
       DOI = {10.1093/imrn/rns240},
       URL = {https://doi.org/10.1093/imrn/rns240},
       ARCHIVEPREFIX = {arXiv},
       EPRINT = {1111.1320},
}

@article {Webster2,
    AUTHOR = {Webster, Ben},
     TITLE = {On graded presentations of {H}ecke algebras and their
              generalizations},
   JOURNAL = {Algebr. Comb.},
  FJOURNAL = {Algebraic Combinatorics},
    VOLUME = {3},
      YEAR = {2020},
    NUMBER = {1},
     PAGES = {1--38},
      ISSN = {2589-5486},
   MRCLASS = {20C08 (20G43)},
  MRNUMBER = {4068742},
MRREVIEWER = {Li\ Luo},
       DOI = {10.5802/alco.84},
       URL = {https://doi.org/10.5802/alco.84},
       ARCHIVEPREFIX = {arXiv},
       EPRINT = {1305.0599},
}

@article {Kang-Kashiwara-Tsuchioka,
    AUTHOR = {Kang, Seok-Jin and Kashiwara, Masaki and Tsuchioka, Shunsuke},
     TITLE = {Quiver {H}ecke superalgebras},
   JOURNAL = {J. Reine Angew. Math.},
  FJOURNAL = {Journal f\"ur die Reine und Angewandte Mathematik. [Crelle's
              Journal]},
    VOLUME = {711},
      YEAR = {2016},
     PAGES = {1--54},
      ISSN = {0075-4102,1435-5345},
   MRCLASS = {81R50 (17B10 17B37 20C08)},
  MRNUMBER = {3456757},
MRREVIEWER = {Se-jin\ Oh},
       DOI = {10.1515/crelle-2013-0089},
       URL = {https://doi.org/10.1515/crelle-2013-0089},
       ARCHIVEPREFIX = {arXiv},
       EPRINT = {1107.1039},
}

@article {Savage-Stuart,
    AUTHOR = {Savage, Alistair and Stuart, John},
     TITLE = {Frobenius nil-{H}ecke algebras},
   JOURNAL = {Pacific J. Math.},
  FJOURNAL = {Pacific Journal of Mathematics},
    VOLUME = {311},
      YEAR = {2021},
    NUMBER = {2},
     PAGES = {455--473},
      ISSN = {0030-8730,1945-5844},
   MRCLASS = {20C08},
  MRNUMBER = {4292937},
MRREVIEWER = {Xun\ Xie},
       DOI = {10.2140/pjm.2021.311.455},
       URL = {https://doi.org/10.2140/pjm.2021.311.455},
       ARCHIVEPREFIX = {arXiv},
       EPRINT = {2008.07977},
}

@article {Wang2,
    AUTHOR = {Wang, Weiqiang},
     TITLE = {Spin {H}ecke algebras of finite and affine types},
   JOURNAL = {Adv. Math.},
  FJOURNAL = {Advances in Mathematics},
    VOLUME = {212},
      YEAR = {2007},
    NUMBER = {2},
     PAGES = {723--748},
      ISSN = {0001-8708,1090-2082},
   MRCLASS = {20C08},
  MRNUMBER = {2329318},
MRREVIEWER = {Pramod\ N.\ Achar},
       DOI = {10.1016/j.aim.2006.11.007},
       URL = {https://doi.org/10.1016/j.aim.2006.11.007},
       ARCHIVEPREFIX = {arXiv},
       EPRINT = {math/0611950},
}

@book {Kleshchev,
    AUTHOR = {Kleshchev, Alexander},
     TITLE = {Linear and projective representations of symmetric groups},
    SERIES = {Cambridge Tracts in Mathematics},
    VOLUME = {163},
 PUBLISHER = {Cambridge University Press, Cambridge},
      YEAR = {2005},
     PAGES = {xiv+277},
      ISBN = {0-521-83703-0},
   MRCLASS = {20C30 (20C08)},
  MRNUMBER = {2165457},
MRREVIEWER = {Christine\ Bessenrodt},
       DOI = {10.1017/CBO9780511542800},
       URL = {https://doi.org/10.1017/CBO9780511542800},
}

@misc{Song-XWang,
  author       = {Linliang Song and Xingyu Wang},
  title        = {Affine Web of type {$Q$}},
  year         = {2025},
       ARCHIVEPREFIX = {arXiv},
       EPRINT = {2506.09729},
}

@article {Nazarov,
    AUTHOR = {Nazarov, Maxim},
     TITLE = {Young's symmetrizers for projective representations of the
              symmetric group},
   JOURNAL = {Adv. Math.},
  FJOURNAL = {Advances in Mathematics},
    VOLUME = {127},
      YEAR = {1997},
    NUMBER = {2},
     PAGES = {190--257},
      ISSN = {0001-8708,1090-2082},
   MRCLASS = {20C25},
  MRNUMBER = {1448714},
MRREVIEWER = {Jorn\ B.\ Olsson},
       DOI = {10.1006/aima.1997.1621},
       URL = {https://doi.org/10.1006/aima.1997.1621},
}

@article {Brundan-Ellis,
    AUTHOR = {Brundan, Jonathan and Ellis, Alexander P.},
     TITLE = {Monoidal supercategories},
   JOURNAL = {Comm. Math. Phys.},
  FJOURNAL = {Communications in Mathematical Physics},
    VOLUME = {351},
      YEAR = {2017},
    NUMBER = {3},
     PAGES = {1045--1089},
      ISSN = {0010-3616,1432-0916},
   MRCLASS = {18D10 (17B37 18D05)},
  MRNUMBER = {3623246},
MRREVIEWER = {Robert\ Harold\ McRae},
       DOI = {10.1007/s00220-017-2850-9},
       URL = {https://doi.org/10.1007/s00220-017-2850-9},
       ARCHIVEPREFIX = {arXiv},
       EPRINT = {1603.05928},
}

@misc{Liu,
  author       = {Bingyan Liu},
  title        = {Presentations of linear monoidal categories and their endomorphism algebras},
  year         = {2018},
       ARCHIVEPREFIX = {arXiv},
       EPRINT = {1810.10988},
}

@article {Brundan-Kleshchev,
    AUTHOR = {Brundan, Jonathan and Kleshchev, Alexander},
     TITLE = {Projective representations of symmetric groups via {S}ergeev
              duality},
   JOURNAL = {Math. Z.},
  FJOURNAL = {Mathematische Zeitschrift},
    VOLUME = {239},
      YEAR = {2002},
    NUMBER = {1},
     PAGES = {27--68},
      ISSN = {0025-5874,1432-1823},
   MRCLASS = {20C25 (17B10 20C30)},
  MRNUMBER = {1879328},
MRREVIEWER = {Christine\ Bessenrodt},
       DOI = {10.1007/s002090100282},
       URL = {https://doi.org/10.1007/s002090100282},
}

@article {Brundan-Savage-Webster2,
    AUTHOR = {Brundan, Jonathan and Savage, Alistair and Webster, Ben},
     TITLE = {Foundations of {F}robenius {H}eisenberg categories},
   JOURNAL = {J. Algebra},
  FJOURNAL = {Journal of Algebra},
    VOLUME = {578},
      YEAR = {2021},
     PAGES = {115--185},
      ISSN = {0021-8693,1090-266X},
   MRCLASS = {18M05 (17B10 17B65)},
  MRNUMBER = {4234799},
       DOI = {10.1016/j.jalgebra.2021.02.025},
       URL = {https://doi.org/10.1016/j.jalgebra.2021.02.025},
       ARCHIVEPREFIX = {arXiv},
       EPRINT = {2007.01642},
}

@article {Brundan-Kleshchev2,
    AUTHOR = {Brundan, Jonathan and Kleshchev, Alexander},
     TITLE = {Hecke-{C}lifford superalgebras, crystals of type
              {$A_{2l}^{(2)}$} and modular branching rules for {$\hat S_n$}},
   JOURNAL = {Represent. Theory},
  FJOURNAL = {Representation Theory. An Electronic Journal of the American
              Mathematical Society},
    VOLUME = {5},
      YEAR = {2001},
     PAGES = {317--403},
      ISSN = {1088-4165},
   MRCLASS = {17B67 (17B10 20C08 20C20)},
  MRNUMBER = {1870595},
MRREVIEWER = {Dmitry\ A.\ Timash\"ev},
       DOI = {10.1090/S1088-4165-01-00123-6},
       URL = {https://doi.org/10.1090/S1088-4165-01-00123-6},
       ARCHIVEPREFIX = {arXiv},
       EPRINT = {math/0103060},
}

@article{Schur,
author = {Schur, J.},
journal = {Journal für die reine und angewandte Mathematik},
pages = {155-250},
title = {Über die Darstellung der symmetrischen und der alternierenden Gruppe durch gebrochene lineare Substitutionen.},
url = {http://eudml.org/doc/149348},
volume = {139},
year = {1911},
}

@article {Jones-Nazarov,
    AUTHOR = {Jones, A. R. and Nazarov, M. L.},
     TITLE = {Affine {S}ergeev algebra and {$q$}-analogues of the {Y}oung
              symmetrizers for projective representations of the symmetric
              group},
   JOURNAL = {Proc. London Math. Soc. (3)},
  FJOURNAL = {Proceedings of the London Mathematical Society. Third Series},
    VOLUME = {78},
      YEAR = {1999},
    NUMBER = {3},
     PAGES = {481--512},
      ISSN = {0024-6115,1460-244X},
   MRCLASS = {20C25 (05E10)},
  MRNUMBER = {1674836},
MRREVIEWER = {A.\ S.\ Kondrat\cprime ev},
       DOI = {10.1112/S002461159900177X},
       URL = {https://doi.org/10.1112/S002461159900177X},
       ARCHIVEPREFIX = {arXiv},
       EPRINT = {q-alg/9712041},
}

@misc{Moran,
  author       = {Thomas Moran},
  title        = {Higher-level affine wreath product algebras},
  year         = {2026},
       ARCHIVEPREFIX = {arXiv},
       EPRINT = {2605.04303},
}

@article {Wang-Zhao,
    AUTHOR = {Wang, Weiqiang and Zhao, Lei},
     TITLE = {Representations of {L}ie superalgebras in prime characteristic
              {II}: {T}he queer series},
   JOURNAL = {J. Pure Appl. Algebra},
  FJOURNAL = {Journal of Pure and Applied Algebra},
    VOLUME = {215},
      YEAR = {2011},
    NUMBER = {10},
     PAGES = {2515--2532},
      ISSN = {0022-4049,1873-1376},
   MRCLASS = {17B50},
  MRNUMBER = {2793955},
MRREVIEWER = {A.\ L.\ Onishchik},
       DOI = {10.1016/j.jpaa.2011.02.011},
       URL = {https://doi.org/10.1016/j.jpaa.2011.02.011},
       ARCHIVEPREFIX = {arXiv},
       EPRINT = {0902.2758},
}

@article {Kleshchev-Livesey,
    AUTHOR = {Kleshchev, Alexander S. and Livesey, Michael},
     TITLE = {Ro{CK} blocks for double covers of symmetric groups and quiver
              {H}ecke superalgebras},
   JOURNAL = {Mem. Amer. Math. Soc.},
  FJOURNAL = {Memoirs of the American Mathematical Society},
    VOLUME = {309},
      YEAR = {2025},
    NUMBER = {1564},
     PAGES = {v+182},
      ISSN = {0065-9266,1947-6221},
      ISBN = {978-1-4704-7354-9; 978-1-4704-8157-5},
   MRCLASS = {20C20 (18N25 20C25 20C30)},
  MRNUMBER = {4913465},
MRREVIEWER = {Mary\ Schaps},
       DOI = {10.1090/memo/1564},
       URL = {https://doi.org/10.1090/memo/1564},
       ARCHIVEPREFIX = {arXiv},
       EPRINT = {2201.06870},
}

\end{document}